\documentclass[english]{amsart}
\usepackage[T1]{fontenc}
\usepackage[latin9]{inputenc}
\usepackage{prettyref}
\usepackage{mathrsfs}
\usepackage{mathtools}
\usepackage{enumitem}
\usepackage{amstext}
\usepackage{amsthm}
\usepackage{amssymb}
\usepackage{stmaryrd}
\usepackage{esint}

\makeatletter
\numberwithin{equation}{section}
\numberwithin{figure}{section}
\theoremstyle{plain}
\newtheorem{thm}{\protect\theoremname}[section]
  \theoremstyle{plain}
  \newtheorem{lem}[thm]{\protect\lemmaname}
  \theoremstyle{plain}
  \newtheorem{prop}[thm]{\protect\propositionname}
  \theoremstyle{plain}
  \newtheorem{cor}[thm]{\protect\corollaryname}

\makeatletter
\@addtoreset{thm}{section}
\makeatother
\usepackage{prettyref}
\usepackage{refstyle}
\usepackage{hyperref}
\usepackage{url}
\usepackage{mathtools}
\newref{claim}{refcmd={claim \ref{#1}}}

\makeatother

\usepackage{babel}
  \providecommand{\corollaryname}{Corollary}
  \providecommand{\lemmaname}{Lemma}
  \providecommand{\propositionname}{Proposition}
\providecommand{\theoremname}{Theorem}

\begin{document}

\title[Trace formula]{A Gutzwiller type trace formula for the magnetic Dirac operator }

\author{Nikhil Savale}
\begin{abstract}
For manifolds including metric-contact manifolds with non-resonant
Reeb flow, we prove a Gutzwiller type trace formula for the associated
magnetic Dirac operator involving contributions from Reeb orbits on
the base. As an application, we prove a semiclassical limit formula
for the eta invariant. 
\end{abstract}

\address{Laboratoire Jacques Louis Lions (LJLL), Université Pierre et Marie
Curie - Paris 6 (UPMC), 4 place Jussieu 75252 Paris, Cedex 05.}

\email{savale@ljll.math.upmc.fr}

\thanks{The author acknowledges the financial support of the Agence Nationale
de la Recherche, projet ANR-15-CE40-0018 (Sub-Riemannian Geometry
and Interactions).}

\maketitle

\section{Introduction}

The trace formulas of Gutzwiller \cite{Gutzwiller-book} and Duistermaat-Guillemin
\cite{Duistermaat-Guillemin} are a clear statement of the semiclassical
correspondence, expressing the spectrum of ($h$-) pseudo-differential
operators in terms of periodic orbits of the underlying Hamiltonian
dynamics as $h\rightarrow0$. We refer to \cite{CdV-tracesurvey,Uribe-Trfor.}
for a historical survey of trace formulas and the associated calculus
of Fourier integral operators. For non-scalar pseudo-differential
operators this calculus is often unavailable due to the non-diagonalizability
of the principal symbol $\sigma\left(A\right)$. Indeed when the eigenvalues
of $\sigma\left(A\right)$ are not smooth functions on the cotangent
space, their corresponding Hamiltonian dynamics is not well-defined.
The purpose of this article is to investigate the trace formula in
one such case.

More precisely, let $\left(X,g^{TX}\right)$ be an oriented Riemannian
manifold of odd dimension $n=2m+1$ equipped with a spin structure.
Let $S$ be the corresponding spin bundle and let $L$ be an auxiliary
Hermitian line bundle. Fix a unitary connection $A_{0}$ on $L$ and
let $a\in\Omega^{1}\left(X;\mathbb{R}\right)$ be a contact one form
(i.e. one satisfying $a\wedge\left(da\right)^{m}>0$). This gives
a family of unitary connections on $L$ via $\nabla^{h}=A_{0}+\frac{i}{h}a$
and a corresponding family of coupled magnetic Dirac operators 
\begin{equation}
D_{h}\coloneqq hD_{A_{0}}+ic\left(a\right):\:C^{\infty}\left(S\otimes L\right)\rightarrow C^{\infty}\left(S\otimes L\right)\label{eq:Semiclassical Magnetic Dirac}
\end{equation}
for $h\in\left(0,1\right]$. 

Define the contact hyperplane $H=\textrm{ker}\left(a\right)\subset TX$
as well as the Reeb vector field $R$ via $i_{R}da=0$, $i_{R}a=1$.
We shall now further assume that the Reeb flow of $a$ is non-resonant.
To state this assumption, let $\gamma$ denote a Reeb orbit. For a
fixed point $p\in\gamma$, the linearized Poincare return map $P_{\gamma}:T_{p}X\rightarrow T_{p}X$
has $R_{p}$ as an eigenvector with eigenvalue $1$ and restricts
to a symplectic map on the contact hyperplane $P_{\gamma}^{+}:H_{p}\rightarrow H_{p}$.
We call the Reeb orbit $\gamma$ non-degenerate if $P_{\gamma}^{+}$
has $n-1$ distinct eigenvalues not equal to 1. There now exists a
symplectic basis for $H_{p}$ in which $P_{\gamma}^{+}$ decomposes
as
\begin{equation}
P_{\gamma}^{+}=\left[\bigoplus_{j=1}^{N_{e}}P_{\gamma;\beta_{j}}^{+,e}\right]\oplus\left[\bigoplus_{j=1}^{N_{h}^{+}}P_{\gamma;\alpha_{j}^{+}}^{+,h}\right]\oplus\left[\bigoplus_{j=1}^{N_{h}^{-}}-P_{\gamma;\alpha_{j}^{-}}^{+,h}\right]\oplus\left[\bigoplus_{j=1}^{N_{l}}P_{\gamma;\alpha_{j}^{0},\beta_{j}^{0}}^{+,l}\right]\label{eq: symplectic decomposition}
\end{equation}
for 
\begin{align}
P_{\gamma;\beta}^{+,e} & =\begin{bmatrix}\cos\beta & -\sin\beta\\
\sin\beta & \cos\beta
\end{bmatrix},\quad\beta\in\left(0,2\pi\right)\label{eq: elliptic map}\\
P_{\gamma;\alpha}^{+,h} & =\begin{bmatrix}e^{\alpha} & 0\\
0 & e^{-\alpha}
\end{bmatrix},\quad\alpha>0\label{eq: hyperbolic map}\\
P_{\gamma;\alpha^{0},\beta^{0}}^{+,l} & =\begin{bmatrix}e^{-\alpha^{0}}\cos\beta^{0} & 0 & -e^{-\alpha^{0}}\sin\beta^{0} & 0\\
0 & e^{\alpha^{0}}\cos\beta^{0} & 0 & -e^{\alpha^{0}}\sin\beta^{0}\\
e^{-\alpha^{0}}\sin\beta^{0} & 0 & e^{-\alpha^{0}}\cos\beta^{0} & 0\\
0 & e^{\alpha^{0}}\sin\beta^{0} & 0 & e^{\alpha^{0}}\cos\beta^{0}
\end{bmatrix},\quad\alpha^{0}>0,\beta^{0}\in\left(0,\pi\right).\label{eq: loxodromic map}
\end{align}
We note that the summands in the decomposition \prettyref{eq: symplectic decomposition}
each correspond to: a pair of elliptic eigenvalues $e^{\pm i\beta}$
(of $P_{\gamma;\beta}^{+,e}$), a pair of positive/negative hyperbolic
eigenvalues $\pm e^{\pm\alpha}$ (of $\pm P_{\gamma;\alpha}^{+,h}$)
and a quartet of loxodromic eigenvalues $e^{\pm\alpha^{0}\pm i\beta^{0}}$
(of $P_{\gamma;\alpha^{0},\beta^{0}}^{+,l}$). We call the Reeb orbit
$\gamma$ non-resonant if the two sets 
\begin{eqnarray*}
\left\{ \alpha_{j}^{+}\right\} _{j=1}^{N_{h}^{+}}\cup\left\{ \alpha_{j}^{-}\right\} _{j=1}^{N_{h}^{-}}\cup\left\{ \alpha_{j}^{0}\right\} _{j=1}^{N_{l}} &  & \textrm{ and}\\
\left\{ 2\pi\right\} \cup\left\{ \beta_{j}\right\} _{j=1}^{N_{e}}\cup\left\{ \beta_{j}^{0}\right\} _{j=1}^{N_{l}}
\end{eqnarray*}
are rationally ($\mathbb{Q}-$) independent. We call the Reeb flow
of $a$ non-resonant if all its Reeb orbits are non-resonant.

Next, we shall assume that the metric $g$ is $strongly\:suitable$
to the contact form $a$. To define this, consider the contracted
endomorphism $\mathfrak{J}:T_{x}X\rightarrow T_{x}X$ defined at each
point $x\in X$ via 
\begin{equation}
da\left(v_{1},v_{2}\right)=g^{TX}\left(v_{1},\mathfrak{J}v_{2}\right),\quad\forall v_{1},v_{2}\in T_{x}X.\label{eq: frak J definition}
\end{equation}
The contact assumption on the one form $a$ implies that $\mathfrak{J}$
has a one dimensional kernel spanned by the Reeb vector field $R$.
The endomorphism $\mathfrak{J}$ is clearly anti-symmetric with respect
to the metric 
\[
g^{TX}\left(v_{1},\mathfrak{J}v_{2}\right)=-g^{TX}\left(\mathfrak{J}v_{1},v_{2}\right)
\]
and hence its non-zero eigenvalues come in purely imaginary pairs
$\pm i\mu$ ; $\mu>0$. We now say that the metric is $strongly\:suitable$
to the contact form $a$ if the spectrum of $\mathfrak{J}_{x}$ is
independent of $x$: there exist positive constants $0<\mu_{1}\leq\mu_{2}\leq\ldots\leq\mu_{m}$
such that 
\begin{equation}
\textrm{Spec}\left(\mathfrak{J}_{x}\right)=\left\{ 0,\pm i\mu_{1},\pm i\mu_{2},\ldots,\pm i\mu_{m}\right\} ,\quad\forall x\in X.\label{eq:Diagonalizability assumption-1}
\end{equation}
We note that this is a slight strengthening of the suitability assumption
from \cite{Savale2017-Koszul} wherein $\textrm{Spec}\left(\mathfrak{J}_{x}\right)$
was allowed to vary in $x$ with one single function $\nu\left(x\right)\in C^{\infty}\left(X\right)$.
Here are two examples of strongly suitable suitable metrics.
\begin{enumerate}
\item The dimension of the manifold $\textrm{dim }X=3$. In this case a
metric $g^{TX}$ is strongly suitable if the magnetic field $\left|da\right|=\mu_{1}$
has constant strength.
\item There is a smooth endomorphism $J:TX\rightarrow TX$, such that\\
 $\left(X^{2m+1},a,g^{TX},J\right)$ is a metric contact manifold.
That is, we have 
\begin{eqnarray}
J^{2}v_{1} & = & -v_{1}+a\left(v_{1}\right)R,\nonumber \\
g^{TX}\left(v_{1},Jv_{2}\right) & = & da\left(v_{1},v_{2}\right),\quad\forall v_{1},v_{2}\in T_{x}X.\label{eq: metric contact structure-1}
\end{eqnarray}
 In this case the nonzero eigenvalues of $\mathfrak{J}_{x}=J_{x}$
are $\pm i$ (each with multiplicity $m$). For any given contact
form $a$ there exists an infinite dimensional space of $\left(g^{TX},J\right)$
satisfying \prettyref{eq: metric contact structure-1}. This case
in particular includes all strictly pseudo-convex CR manifolds.
\end{enumerate}
Our first result is now a Gutzwiller type trace formula for the magnetic
Dirac operator \prettyref{eq:Semiclassical Magnetic Dirac}. To state
it precisely choose $f\in C_{c}^{\infty}\left(-\sqrt{2\mu_{1}},\sqrt{2\mu_{1}}\right)$.
Let $\theta\in C_{c}^{\infty}\left(\mathbb{R};\left[0,1\right]\right)$
be any compactly supported supported function, such that $\theta=1$
near $0$, and set
\begin{eqnarray*}
\mathcal{F}^{-1}\theta\left(x\right) & \coloneqq & \check{\theta}\left(x\right)=\frac{1}{2\pi}\int e^{ix\xi}\theta\left(\xi\right)d\xi\\
\mathcal{F}_{h}^{-1}\theta\left(x\right) & \coloneqq & \frac{1}{h}\check{\theta}\left(\frac{x}{h}\right)=\frac{1}{2\pi h}\int e^{\frac{i}{h}x\xi}\theta\left(\xi\right)d\xi
\end{eqnarray*}
to be its classical and semi-classical inverse Fourier transforms
respectively. We shall then prove. 
\begin{thm}
\label{thm:main trace expansion-1}Let $a$ be a non-resonant contact
form and $g^{TX}$ a strongly suitable metric. We then have a trace
expansion 
\begin{align}
\textrm{tr}\left[f\left(\frac{D}{\sqrt{h}}\right)\left(\mathcal{F}_{h}^{-1}\theta\right)\left(\lambda\sqrt{h}-D\right)\right] & =\label{eq: Main trace expansion-1}\\
\textrm{tr}\left[f\left(\frac{D}{\sqrt{h}}\right)\frac{1}{h}\check{\theta}\left(\frac{\lambda\sqrt{h}-D}{h}\right)\right] & =h^{-m-1}\left(\sum_{j=0}^{N}f\left(\lambda\right)u_{j}\left(\lambda\right)h^{j/2}\right)\label{eq: gutwiller local term}\\
 & +\sum_{\gamma}e^{\frac{i}{h}T_{\gamma}}e^{i\frac{\pi}{2}\mathfrak{m}_{\gamma}}\sum_{j=0}^{N-2m-2}h^{j/2}\sum_{k=0}^{j}\lambda^{k}A_{\gamma,j,k}\theta\left(L_{\gamma}\right)\label{eq: gutzwiller dynamical term}\\
 & +O\left(h^{N/2-m-1}\right)\label{eq: gutzwiller remainder}
\end{align}
for each $N\in\mathbb{N}$,$\lambda\in\mathbb{R}$. Here the second
line on the right hand side above is a sum over the Reeb orbits of
$a$. Furthermore; the terms appearing on the right hand side are
as follows
\begin{enumerate}
\item each $u_{j}$ is a polynomial function in $\lambda$
\item each $A_{\gamma,j,k}$ is a differential operator on $\mathbb{R}$
of order between $k$ and $j$
\item $T_{\gamma}$ and $L_{\gamma}$ denote the period and Riemannian length
of the Reeb orbit respectively
\item $\mathfrak{m}_{\gamma}$ denotes the Maslov index of a metaplectic
lift of $P_{\gamma}^{+}$.
\end{enumerate}
Finally, the leading contribution of each Reeb orbit $\gamma$ is
given by the multiplication operator
\[
A_{\gamma,0,0}\theta=\frac{L_{\gamma}^{\#}}{2\pi}\frac{1}{\sqrt{\left|\det\left(1-P_{\gamma}^{+}\right)\right|}}\theta
\]
with $L_{\gamma}^{\#}$ denoting the primitive length of the orbit.
\end{thm}
An immediate consequence of the above trace formula is a little o
estimate on the dimension of the kernel of $D_{h}$
\begin{equation}
k_{h}\coloneqq\textrm{dim ker}\left(D_{h}\right)=o\left(h^{-m}\right).\label{eq:dimension kernel estimate}
\end{equation}
As another application, we shall prove a semiclassical limit formula
for the (rescaled) eta invariant of the magnetic Dirac operator $D_{h}$.
To state this, first let $R^{\perp}\subset TX$ denote the $2m$-dimensional
orthogonal complement to the Reeb vector field. We may now define
the endomorphisms $\left(\nabla^{TX}\mathfrak{J}\right)^{0}:R^{\perp}\rightarrow R^{\perp}$,
$\left|\mathfrak{J}\right|:R^{\perp}\rightarrow R^{\perp}$, via
\begin{align}
\left(\nabla^{TX}\mathfrak{J}\right)^{0}v\coloneqq & \left(\nabla_{v}^{TX}\mathfrak{J}\right)R,\quad\forall v\in R^{\perp},\nonumber \\
\left|\mathfrak{J}\right|\coloneqq & \sqrt{-\mathfrak{J}^{2}}.\label{eq: abs. ctct. end.}
\end{align}
 We then have the following.
\begin{thm}
\label{thm: eta semiclassical limit}Let $a$ be a non-resonant contact
form and $g^{TX}$ a strongly suitable metric. The rescaled eta invariant
of the Dirac operator \prettyref{eq:Semiclassical Magnetic Dirac}
satisfies 
\begin{equation}
\lim_{h\rightarrow0}h^{m}\eta\left(D_{h}\right)=-\frac{1}{2}\frac{1}{\left(2\pi\right)^{m+1}}\frac{1}{m!}\int_{X}\left[\textrm{tr }\left|\mathfrak{J}\right|^{-1}\left(\nabla^{TX}\mathfrak{J}\right)^{0}\right]a\wedge\left(da\right)^{m}.\label{eq: eta formula}
\end{equation}
\end{thm}
Before proceeding further we look at the limit formula formula above
in the two special cases mentioned earlier.
\begin{enumerate}
\item The dimension of the manifold $\textrm{dim }X=3$ and $\left|da\right|=\mu_{1}$
has constant strength. In this case the limit \prettyref{eq: eta formula}
is given by the volume integral
\[
\lim_{h\rightarrow0}h^{m}\eta\left(D_{h}\right)=-\frac{\mu_{1}}{8\pi^{2}}\int_{X}\left[i_{R}d^{*}da\right]dx.
\]
\item There is a smooth endomorphism $J:TX\rightarrow TX$, such that\\
 $\left(X^{2m+1},a,g^{TX},J\right)$ is a metric contact manifold
\prettyref{eq: metric contact structure-1}. In this case the limit
\prettyref{eq: eta formula} is simply the volume 
\[
\lim_{h\rightarrow0}h^{m}\eta\left(D_{h}\right)=-\frac{m}{2}\frac{1}{\left(2\pi\right)^{m+1}}\textrm{vol}\left(X\right).
\]
\end{enumerate}
$\qquad$A small time trace formula \prettyref{eq: Main trace expansion-1}
was already proved in \cite{Savale2017-Koszul} assuming $\theta$
to be supported sufficiently close to the origin; much of this article
attempts to extend the arguments therein to large supports. By the
construction of appropriate trapping functions it is shown that the
formula of \cite{Savale2017-Koszul} extends to large time when microlocalized
away from the Reeb orbits. Near the Reeb orbits, the trace is studied
via understanding the Birkhoff normal form of $D_{h}$ near each orbit,
using which it is reduced to the trace of a scalar effective Hamiltonian.
The Birkhoff normal form procedure here combines the one in \cite{Savale2017-Koszul}
with ones for scalar Hamiltonians \cite{Guillemin-wavetraceinv.,Guillemin-Paul2010,Iant-Sjostrand-Zwroski,Zelditch-ellipticgeodesics,Zelditch98-waveinv}
near periodic Hamiltonian orbits and hence requires the non-resonance
assumption. The semiclassical asymptotics for the Dirac operator considered
here were originally motivated by Taubes's proof of the three dimensional
Weinstein conjecture \cite{Taubes-Weinstein} on the existence of
Reeb orbits. The existence of Reeb orbits, or the necessity of dynamical
contributions \prettyref{eq: gutzwiller dynamical term}, is still
unresolved in higher dimensions. 

The behavior of the eta invariant of Dirac operators has been studied
under various operations (cf. \cite{Goette-2012-etasurvey} for a
survey) and the formula \prettyref{eq: eta formula} adds to a long
list. A more precise relation between the eta invariant and the dynamics
of geodesic flow has been studied on compact hyperbolic manifolds
\cite{Millson78} and locally symmetric spaces of non-compact type
\cite{Moscovici-Stanton-eta}. The proof of such precise relations
on general negatively curved manifolds is the subject of the hypo-elliptic
Laplacian program of Bismut \cite{Bismut-hypoelliptic-Dirac,Bismut-hypoelliptic-eta}.

Under the well known correspondence between semi-classical and microlocal
analysis, the operator \prettyref{eq:Semiclassical Magnetic Dirac}
corresponds to a hypo elliptic sub-Riemannian (sR) Dirac operator
on the product $X\times S^{1}$. The Reeb orbits on $X$ correspond
to singular geodesics on the quasi-contact product suggesting a more
general trace formula for sR Dirac operators. The eigenvalues of the
symbol of the sR Dirac operator being the square root of the symbol
of the sR Laplacian up to sign, similar trace formulas could be expected
for the half-wave equation of the sR Laplacian. A systematic study
of spectral asymptotics for sR Laplacians and related dynamics has
been recently undertaken \cite{Colin-de-Verdiere-Hillairet-TrelatI,Colin-de-Verdiere-Hillairet-TrelatII}.

The paper is organized as follows. In \prettyref{sec:Preliminaries}
we begin with the preliminaries of Dirac operators, Clifford representations
and semi-classical analysis used in the paper. In \prettyref{sec:Dynamical-partitions}
we breakup the trace \prettyref{eq: Main trace expansion-1} using
a partition of unity adapted to the Reeb dynamics. By the construction
of appropriate trapping functions it is shown here that the trace
does not have non-local contributions when microlocalized away from
the Reeb orbits. In \prettyref{sec: Birkhoff normal form near Reeb orbit}
we generalize the Birkhoff normal form of \cite{Savale2017-Koszul}
to one in a neighborhood of each Reeb orbit. This normal form is then
used, via the construction of a similar trapping functions to reduce
the trace asymptotics to $S^{1}\times\mathbb{R}^{2m}$ in \prettyref{sec:Reduction to S1 times R2m}
leading to a proof of \prettyref{thm:main trace expansion-1} in \prettyref{sec: Trace Formula}.
In \prettyref{sec:Local trace expansion second term} we compute the
second term in the local trace expansion of \prettyref{eq: gutwiller local term}.
This leads to the semi-classical limit formula for the eta invariant
\prettyref{eq: eta formula} in the final \prettyref{sec:Semiclassical-limit-of eta}. 

\section{\label{sec:Preliminaries}Preliminaries}

\subsection{Spectral invariants of the Dirac operator}

\noindent Here we review the basic facts about Dirac operators used
throughout the paper with \cite{Berline-Getzler-Vergne} providing
a standard reference. Consider a compact, oriented, Riemannian manifold
$\left(X,g^{TX}\right)$ of odd dimension $n=2m+1$. Let $X$ be equipped
with spin structure, i.e. a principal $\textrm{Spin}\left(n\right)$
bundle $\textrm{Spin}\left(TX\right)\rightarrow SO\left(TX\right)$
with an equivariant double covering of the principal $SO\left(n\right)$-bundle
of orthonormal frames $SO\left(TX\right)$. The corresponding spin
bundle $S=\textrm{Spin}\left(TX\right)\times_{\textrm{Spin}\left(n\right)}S_{2m}$
is associated to the unique irreducible representation of $\textrm{Spin}\left(n\right)$.
Let $\nabla^{TX}$ denote the Levi-Civita connection on $TX$. This
lifts to the spin connection $\nabla^{S}$ on the spin bundle $S$.
The Clifford multiplication endomorphism $c:T^{*}X\rightarrow S\otimes S^{*}$
may be defined (see \prettyref{subsec:Clifford algebra}) satisfying
\begin{align*}
c(a)^{2}=-|a|^{2}, & \quad\forall a\in T^{*}X.
\end{align*}
Let $L$ be a Hermitian line bundle on $X$. Let $A_{0}$ be a fixed
unitary connection on $L$ and let $a\in\Omega^{1}(X;\mathbb{R})$
be a 1-form on $X$. This gives a family $\nabla^{h}=A_{0}+\frac{i}{h}a$
of unitary connections on $L$. We denote by $\nabla^{S\otimes L}=\nabla^{S}\otimes1+1\otimes\nabla^{h}$
the tensor product connection on $S\otimes L.$ Each such connection
defines a coupled Dirac operator 
\begin{align*}
D_{h}\coloneqq hD_{A_{0}}+ic\left(a\right)=hc\circ\left(\nabla^{S\otimes L}\right):C^{\infty}(X;S\otimes L)\rightarrow C^{\infty}(X;S\otimes L)
\end{align*}
for $h\in\left(0,1\right]$. The operator $D_{h}$ is elliptic and
self-adjoint. It hence possesses a discrete spectrum of eigenvalues. 

We define the eta function of $D_{h}$ by the formula
\begin{align}
\eta\left(D_{h},s\right)\coloneqq & \sum_{\begin{subarray}{l}
\quad\:\lambda\neq0\\
\lambda\in\textrm{Spec}\left(D_{h}\right)
\end{subarray}}\textrm{sign}(\lambda)|\lambda|^{-s}=\frac{1}{\Gamma\left(\frac{s+1}{2}\right)}\int_{0}^{\infty}t^{\frac{s-1}{2}}\textrm{tr}\left(D_{h}e^{-tD_{h}^{2}}\right)dt,\label{eq:eta invariant definition}
\end{align}
$\forall s\in\mathbb{C}$. Here, and in the remainder of the paper,
we use the convention that $\textrm{Spec}(D_{h})$ denotes a multiset
with each eigenvalue of $D_{h}$ being counted with its multiplicity.
The above series converges for $\textrm{Re}(s)>n.$ It was shown in
\cite{APSI,APSIII} that the eta function possesses a meromorphic
continuation to the entire complex $s$-plane and has no pole at zero.
Its value at zero is defined to be the eta invariant of the Dirac
operator
\[
\eta_{h}\coloneqq\eta\left(D_{h},0\right).
\]
By including the zero eigenvalue in \prettyref{eq:eta invariant definition},
with an appropriate convention, we may define a variant known as the
reduced eta invariant by 
\begin{align*}
\bar{\eta}_{h}\coloneqq & \frac{1}{2}\left\{ k_{h}+\eta_{h}\right\} .
\end{align*}

The eta invariant is unchanged under positive scaling
\begin{equation}
\eta\left(D_{h},0\right)=\eta\left(cD_{h},0\right);\quad\forall c>0.\label{eq: eta scale invariant}
\end{equation}
Let $L_{t,h}$ denote the Schwartz kernel of the operator $D_{h}e^{-tD_{h}^{2}}$
on the product $X\times X$. Throughout the paper all Schwartz kernels
will be defined with respect to the Riemannian volume density. Denote
by $\textrm{tr}\left(L_{t,h}\left(x,x\right)\right)$ the point-wise
trace of $L_{t,h}$ along the diagonal. We may now analogously define
the function 
\begin{align}
\eta\left(D_{h},s,x\right)= & \frac{1}{\Gamma\left(\frac{s+1}{2}\right)}\int_{0}^{\infty}t^{\frac{s-1}{2}}\textrm{tr}\left(L_{t,h}\left(x,x\right)\right)dt,\label{eq:eta function diagonal}
\end{align}
$\forall s\in\mathbb{C},\,x\in X$. In \cite{Bismut-Freed-II} theorem
2.6, it was shown that for $\textrm{Re}(s)>-2$, the function $\eta\left(D_{h},s,x\right)$
is holomorphic in $s$ and smooth in $x$. From \prettyref{eq:eta function diagonal}
it is clear that this is equivalent to 
\begin{align}
\textrm{tr}\left(L_{t,h}\right)= & O\left(t^{\frac{1}{2}}\right),\quad\textrm{as}\:t\rightarrow0.\label{eq:pointwise trace asymp as t->0}
\end{align}
The eta invariant is then given by the convergent integral 
\begin{equation}
\eta_{h}=\int_{0}^{\infty}\frac{1}{\sqrt{\pi t}}\textrm{tr}\left(D_{h}e^{-tD_{h}^{2}}\right)dt.\label{eq: eta integral}
\end{equation}

\subsection{\label{subsec:Clifford algebra}Clifford algebra and and its representations}

Here we review the construction of the spin representation of the
Clifford algebra. The following being standard, is merely used to
setup our conventions.

Consider a real vector space $V$ of even dimension $2m$ with metric
$\left\langle ,\right\rangle $. Recall that its Clifford algebra
$Cl\left(V\right)$ is defined as the quotient of the tensor algebra
$T\left(V\right):=\oplus_{j=0}^{\infty}V^{\otimes j}$ by the ideal
generated from the relations $v\otimes v+\left|v\right|^{2}=0$. Fix
a compatible almost complex structure $J$ and split $V\otimes\mathbb{C}=V^{1,0}\oplus V^{0,1}$
into the $\pm i$ eigenspaces of $J$. The complexification $V\otimes\mathbb{C}$
carries an induced $\mathbb{C}$-bilinear inner product $\left\langle ,\right\rangle _{\mathbb{C}}$
as well as an induced Hermitian inner product $h^{\mathbb{C}}\left(,\right)$.
Next, define $S_{2m}=\Lambda^{*}V^{1,0}$. Clearly $S_{2m}$ is a
complex vector space of dimension $2^{m}$ on which the unique irreducible
(spin)-representation of the Clifford algebra $Cl\left(V\right)\otimes\mathbb{C}$
is defined by the rule
\[
c_{2m}\left(v\right)\omega=\sqrt{2}\left(v^{1,0}\wedge\omega-\iota_{v^{0,1}}\omega\right),\quad v\in V,\omega\in S_{2m}.
\]
The contraction above is taken with respect to $\left\langle ,\right\rangle _{\mathbb{C}}$.
It is clear that $c_{2m}\left(v\right):\Lambda^{\textrm{even/odd}}\rightarrow\Lambda^{\textrm{odd/even}}$
switches the odd and even factors. For the Clifford algebra $Cl\left(W\right)\otimes\mathbb{C}$
of an odd dimensional vector space $W=V\oplus\mathbb{R}\left[e_{0}\right]$
there are exactly two irreducible representations. The first (spin)-representation
$S_{2m+1}=\Lambda^{*}V^{1,0}$ is defined via 
\begin{eqnarray}
c_{2m+1}\left(v\right) & = & c_{2m}\left(v\right),\quad v\in V\nonumber \\
c_{2m+1}\left(e_{0}\right)\omega_{\textrm{even/odd}} & = & \pm\frac{1}{i}\omega_{\textrm{even/odd}}\label{eq:odd clifford representation}
\end{eqnarray}
while the other corresponds to the opposite sign convention in \prettyref{eq:odd clifford representation}
above. We shall often use the shorthand's $c=c_{2m}=c_{2m+1}$ with
the index $2m$, $2m+1$ implicitly understood.

Pick an orthonormal basis $e_{1},e_{2},\ldots,e_{2m}$ for $V$ in
which the almost complex structure is given by $Je_{j}=e_{j+m}$,
$1\leq j\leq m$. An $h^{\mathbb{C}}$-orthonormal basis for $V^{1,0}$
is now given by $w_{j}=\frac{1}{\sqrt{2}}\left(e_{j+m}+ie_{j}\right)$,
$1\leq j\leq m$. A basis for $S_{2m}$ and $S_{2m+1}^{\pm}$ is given
by
\begin{equation}
w_{k}=w_{1}^{k_{1}}\wedge\ldots\wedge w_{m}^{k_{m}},\:\textrm{with }k=\left(k_{1},k_{2},\ldots,k_{m}\right)\in\left\{ 0,1\right\} ^{m}.\label{eq: basis spin representation}
\end{equation}
 Ordering the above chosen bases lexicographically in $k$, we may
define the Clifford matrices, of rank $2^{m}$, via 
\begin{eqnarray*}
\gamma_{j}^{m} & = & c\left(e_{j}\right),\quad0\leq j\leq2m,
\end{eqnarray*}
for each $m$ . We note that the above is a slightly different convention
from \cite{Savale2017-Koszul} adopted to simplify some formulas in
\prettyref{sec:Local trace expansion second term}. Again, we often
write $\gamma_{j}^{m}=\gamma_{j}$ with the index $m$ implicitly
understood. Giving representations of the Clifford algebra, these
matrices satisfy the relation 
\begin{equation}
\gamma_{i}\gamma_{j}+\gamma_{j}\gamma_{i}=-2\delta_{ij}.\label{eq: Clifford relations}
\end{equation}
We also set $\sigma_{j}=i\gamma_{j}$. 

Next, one may further define the Clifford quantization map on the
exterior algebra
\begin{eqnarray}
c:\Lambda^{*}W\otimes\mathbb{C} & \rightarrow & \textrm{End}\left(S_{2m}\right)\nonumber \\
c\left(e_{0}^{k_{0}}\wedge\ldots\wedge e_{2m}^{k_{2m}}\right) & = & c\left(e_{0}\right)^{k_{0}}\ldots c\left(e_{2m}\right)^{k_{2m}}.\label{eq: clifford quantization}
\end{eqnarray}
An easy computation yields 
\begin{eqnarray*}
\gamma_{0}\left(\gamma_{1}\gamma_{m+1}\right)\ldots\left(\gamma_{m}\gamma_{2m}\right) & = & \frac{1}{i^{m+1}}
\end{eqnarray*}
and hence 
\[
\textrm{tr}\left[\gamma_{0}\ldots\gamma_{2m}\right]=\frac{1}{i^{m+1}}2^{m}.
\]

Furthermore, if $e_{0}\wedge\ldots\wedge e_{2m}$ is designated to
give a positive orientation for $W$ then for $\omega\in\Lambda^{k}W$
we have 
\begin{eqnarray}
c\left(\ast\omega\right) & = & i^{m+1}\left(-1\right)^{\frac{k\left(k+1\right)}{2}}c\left(\omega\right)\label{eq:clifford quantization hodge dual}\\
c\left(\omega\right)^{*} & = & \left(-1\right)^{\frac{k\left(k+1\right)}{2}}c\left(\omega\right)\label{eq:clifford quantization adjoint}
\end{eqnarray}
under the Hodge star and $h^{\mathbb{C}}$-adjoint. The Clifford quantization
map \prettyref{eq: clifford quantization} is a linear surjection
with kernel spanned by elements of the form $\ast\omega-i^{m+1}\left(-1\right)^{\frac{k\left(k+1\right)}{2}}\omega$.
Thus, in particular one has linear isomorphisms 
\begin{equation}
c:\Lambda^{\textrm{even/odd}}W\otimes\mathbb{C}\rightarrow\textrm{End}\left(S_{2m}\right).\label{eq:clifford algebra is matrix algebra}
\end{equation}

Next, given $\left(r_{1},\ldots,r_{m}\right)\in\mathbb{R}^{m}\setminus0$,
we define
\begin{eqnarray}
I_{r} & \coloneqq & \left\{ j|r_{j}\neq0\right\} \subset\left\{ 1,2,\ldots,m\right\} \label{eq: I_r-1}\\
Z_{r} & \coloneqq & \left|I_{r}\right|\label{eq: Z_r-1}\\
V_{r} & \coloneqq & \bigoplus_{j\in I_{r}}\mathbb{C}\left[w_{j}\right]\subset V^{1,0}\label{eq: V_r-1}\\
\textrm{and }\quad w_{r} & \coloneqq & \sum_{j=1}^{m}r_{j}w_{j}\in V_{r}.\label{eq: w_r-1}
\end{eqnarray}
Clearly, $\left\Vert w_{r}\right\Vert =\left|r\right|$. Denoting
by $w_{r}^{\perp}$ the $h^{\mathbb{C}}$-orthogonal complement of
$w_{r}\subset V_{r}$, one clearly has $V_{r}=\mathbb{C}\left[w_{r}\right]\oplus w_{r}^{\perp}$.
We set
\begin{eqnarray}
\mathtt{i}_{r}:\Lambda^{*}V_{r} & \rightarrow & \Lambda^{*}V_{r},\quad\textrm{ via}\label{eq: involution}\\
\mathtt{i}_{r}\left(\omega\right) & \coloneqq & \frac{w_{r}}{\left|r\right|}\wedge\omega\nonumber \\
\mathtt{i}_{r}\left(\frac{w_{r}}{\left|r\right|}\wedge\omega\right) & \coloneqq & \omega\nonumber 
\end{eqnarray}
for $\omega\in\Lambda^{*}w_{r}^{\perp}$. Clearly, $\mathtt{i}_{r}^{2}=1$
and $\mathtt{i}_{r}$ is a linear isomorphism between 
\begin{eqnarray*}
\mathtt{i}_{r}:\Lambda^{\textrm{even}}V_{r} & \rightarrow & \Lambda^{\textrm{odd}}V_{r}\\
\mathtt{i}_{r}:\Lambda^{\textrm{odd}}V_{r} & \rightarrow & \Lambda^{\textrm{even}}V_{r}.
\end{eqnarray*}
Next, the endomorphism 
\begin{eqnarray}
c\left(\frac{w_{r}-\bar{w}_{r}}{\sqrt{2}}\right)=\left(w_{r}\wedge+\iota_{\bar{w}_{r}}\right):\Lambda^{\textrm{*}}V_{r} & \rightarrow & \Lambda^{\textrm{*}}V_{r}\label{eq: clifford multiplication partial vector}
\end{eqnarray}
has the form
\begin{equation}
c\left(\frac{w_{r}-\bar{w}_{r}}{\sqrt{2}}\right)=\begin{bmatrix} & \left|r\right|\mathtt{i}_{r}\\
\left|r\right|\mathtt{i}_{r}
\end{bmatrix}\label{eq: clifford multiplication blocks}
\end{equation}
with respect to the decomposition $\Lambda^{\textrm{*}}V_{r}=\Lambda^{\textrm{odd}}V_{r}\oplus\Lambda^{\textrm{even}}V_{r}$.
This finally allows us to write the eigenspaces of \prettyref{eq: clifford multiplication partial vector}
as 
\begin{equation}
V_{r}^{\pm}=\left(1\pm\mathtt{i}_{r}\right)\left(\Lambda^{\textrm{even}}V_{r}\right)\label{eq: eigenspaces clifford mult.}
\end{equation}
with eigenvalue $\pm\left|r\right|$ respectively.

Finally we shall need an almost diagonalizability result for the restriction
of Clifford multiplication to the sphere. Define $S\left(W\right)=\left\{ v\in W|\left|v\right|=1\right\} $
as well as the restriction 
\begin{align}
c:S\left(W\right) & \rightarrow\mathfrak{u}\left(S_{2m+1}\right)\label{eq:restriction clifford sphere}\\
c\left(v\right)^{2}= & -\textrm{Id}.\nonumber 
\end{align}
The restriction of the spin bundle $S_{2m+1}$ to the sphere $S\left(W\right)$
splits $\left.S_{2m+1}\right|_{S\left(W\right)}=S_{+}\left(W\right)\oplus S_{-}\left(W\right)$
into the $\pm i$ eigenspaces of the $c$ respectively. The summands
$S_{+}\left(W\right),S_{-}\left(W\right)$ maybe identified with the
(non-trivial) bundle of positive and negative spinors on the sphere.
The restriction $c$ \prettyref{eq:restriction clifford sphere} is
hence not globally diagonalizable over the sphere. We now identify
$S\left(W\right)=\left\{ \theta_{0}e_{0}+\ldots\theta_{2m}e_{2m}\in W|\theta_{0}^{2}+\ldots+\theta_{2m}^{2}=1\right\} $
with the standard sphere in $S^{n-1}\subset\mathbb{R}^{n}$ using
the chosen basis for $W$; with the induced basis \prettyref{eq: basis spin representation}
of $S_{2m+1}$ giving identifications $\mathfrak{u}\left(S_{2m+1}\right)=\mathfrak{u}\left(\mathbb{C}^{2^{m}}\right)$,
$U\left(S_{2m+1}\right)=U\left(\mathbb{C}^{2^{m}}\right)$. Thus
\begin{equation}
c\left(\theta\right)\coloneqq c\left(\theta_{0}e_{0}+\ldots\theta_{2m}e_{2m}\right)=\sum_{j=0}^{2m}\theta_{j}\gamma_{j}\in C^{\infty}\left(S^{n-1},\mathfrak{u}\left(\mathbb{C}^{2^{m}}\right)\right)\label{eq:clifford mult rest. sphere}
\end{equation}
 in this trivialization/coordinates. We now have.
\begin{lem}
\label{lem: almost diagonalization lemma}For each $\rho\in\left(0,\frac{1}{8}\right)$,
there exist smooth family of maps/functions $\mathtt{v}_{t}^{\rho}\in C^{\infty}\left(S^{n-1};U\left(\mathbb{C}^{2^{m}}\right)\right)$
; $a_{0,t}^{\rho},a_{1,t}^{\rho}\in C^{\infty}\left(\left[-1,1\right]_{\theta_{0}}\right)$
, $t\in\left[0,1\right]$, such that 
\begin{enumerate}
\item $\left|a_{j,t}^{\rho}\right|\leq\left(\frac{8}{\rho}\right)^{1/2},\left|\partial_{\theta_{0}}a_{j,t}^{\rho}\right|\leq\left(\frac{8}{\rho}\right)^{2}$
, $t\in\left[0,1\right]$, $j=0,1$.
\item $\left\Vert \partial_{t}\mathtt{v}_{t}^{\rho}\right\Vert \leq\left(\frac{8}{\rho}\right)^{2}$,
$\left\Vert \partial_{\theta_{j}}\mathtt{v}_{t}^{\rho}\right\Vert \leq\left(\frac{8}{\rho}\right)^{4}$,
$t\in\left[0,1\right]$, $j=0,\ldots,2m$.
\item $\,$
\begin{equation}
a_{0,t}^{\rho}\left(\theta_{0}\right)=\begin{cases}
\theta_{0}; & t\in\left[0,\frac{1}{2}\right]\\
1; & t=1,\:\theta_{0}<1-\rho,
\end{cases}\label{eq: a0 interpolates}
\end{equation}
\begin{equation}
a_{1,t}^{\rho}\left(\theta_{0}\right)=\begin{cases}
-1; & t\in\left[0,\frac{1}{2}\right]\\
0; & t=1,\:\theta_{0}<1-\rho,
\end{cases}\label{eq: a1 interpolates}
\end{equation}
\begin{equation}
\mathtt{v}_{t}^{\rho}=\sigma_{0};\quad t\in\left[0,\frac{1}{2}\right],\label{eq: v interpolates}
\end{equation}
\item we have the almost diagonalizability equation
\begin{equation}
\mathtt{v}_{t}^{\rho}\left(\theta\right)^{*}c\left(\theta\right)\mathtt{v}_{t}^{\rho}\left(\theta\right)=a_{0,t}^{\rho}\left(\theta_{0}\right)\gamma_{0}+a_{1,t}^{\rho}\left(\theta_{0}\right)\left[\sum_{j=1}^{2m}\theta_{j}\gamma_{j}\right].\label{eq:almost diagonalization}
\end{equation}
\end{enumerate}
\end{lem}
\begin{proof}
The matrix 
\begin{align}
\mathtt{v}: & S^{n-1}\setminus\left\{ \theta_{0}=1\right\} \rightarrow U\left(\mathbb{C}^{2^{m}}\right)\\
\mathtt{v}\left(\theta\right)\coloneqq & \sqrt{\frac{\left(1-\theta_{0}\right)}{2}}\sigma_{0}-\frac{\theta_{j}}{\sqrt{2\left(1-\theta_{0}\right)}}\sigma_{j}\label{eq: expansion v}
\end{align}
diagonalizes 
\begin{equation}
\mathtt{v}^{*}c\left(\theta\right)\mathtt{v}=-\gamma_{0}\label{eq:diagonalization away pole}
\end{equation}
away from the north-pole $\left\{ \theta_{0}=1\right\} $. To get
a map defined on the entire sphere, let $\chi_{1}^{\rho}\in C^{\infty}\left(\left[-1,1\right]_{\theta_{0}};\left[-1,1-\frac{\rho}{2}\right]\right)$
such that 
\begin{equation}
\chi_{1}^{\rho}\left(\theta_{0}\right)=\begin{cases}
\theta_{0}; & -1\leq\theta_{0}<1-\rho,\\
-1; & 1-\frac{\rho}{2}\leq\theta_{0}\leq1,
\end{cases}\label{eq:def rho1}
\end{equation}
with $\left|\left(\chi_{1}^{\rho}\right)'\right|\leq\frac{4}{\rho}$.
Further let $\chi_{0}\in C_{c}^{\infty}\left(\left[-1,1\right]_{t};\left[0,1\right]\right)$
with $\chi_{0}=1$ on $\left(-\frac{1}{2},\frac{1}{2}\right)$ and
$\left|\partial_{t}\chi_{0}\right|\leq4$. Finally set $\chi_{1,t}^{\rho}=\left[1-\chi_{0}\left(t\right)\right]^{2}\chi_{1}^{\rho}-\left[1-\left(1-\chi_{0}\left(t\right)\right)^{2}\right]\in C^{\infty}\left(\left[-1,1\right]_{\theta_{0}};\left[-1,1-\frac{\rho}{2}\right]\right)$
satisfying $\left|\left(\chi_{1,t}^{\rho}\right)'\right|\leq\frac{4}{\rho}$
, $\left|\partial_{t}\chi_{1,t}^{\rho}\right|\leq8$. Now $\chi_{2,t}^{\rho}\left(\theta_{0}\right)=\sqrt{\frac{1-\chi_{1,t}^{\rho}\left(\theta_{0}\right)^{2}}{1-\theta_{0}^{2}}}\in C^{\infty}\left(\left[-1,1\right]_{\theta_{0}}\right)$
satisfies $\left|\chi_{2,t}^{\rho}\right|\leq\left(\frac{2}{\rho}\right)^{1/2}$,
$\left|\left(\chi_{2,t}^{\rho}\right)'\right|\leq\left(\frac{4}{\rho}\right)^{2}$,
$\left|\partial_{t}\chi_{2,t}^{\rho}\right|\leq\left(\frac{4}{\rho}\right)^{2}$.
The family
\begin{align}
\chi_{t}^{\rho}:S^{n-1} & \rightarrow S^{n-1}\setminus\left\{ \theta_{0}=1\right\} \nonumber \\
\chi_{t}^{\rho}\left(\theta\right) & \coloneqq\left(\chi_{1,t}^{\rho}\left(\theta_{0}\right),\chi_{2,t}^{\rho}\left(\theta_{0}\right)\theta_{1},\ldots,\chi_{2,t}^{\rho}\left(\theta_{0}\right)\theta_{2m}\right)\label{eq:def Tepsilon}
\end{align}
now defines a family of maps on the entire sphere
\begin{align}
\mathtt{v}_{t}^{\rho}:S^{n-1} & \rightarrow U\left(\mathbb{C}^{2^{m}}\right)\nonumber \\
\mathtt{v}_{t}^{\rho}\left(\theta\right) & \coloneqq\mathtt{v}\left(\chi_{t}^{\rho}\left(\theta\right)\right).\label{eq: approximate diag  matrix.}
\end{align}
The equation \prettyref{eq:almost diagonalization} now follows from
\prettyref{eq: expansion v}, \prettyref{eq:diagonalization away pole},
\prettyref{eq:def Tepsilon} and \prettyref{eq: approximate diag  matrix.}
with 
\begin{align*}
a_{0,t}^{\rho}= & -\theta_{0}\chi_{1,t}^{\rho}-\left(1-\theta_{0}^{2}\right)\chi_{2,t}^{\rho}\\
a_{1,t}^{\rho}= & \chi_{1,t}^{\rho}-\theta_{0}\chi_{2,t}^{\rho}.
\end{align*}
\end{proof}

\subsubsection{Magnetic Dirac operator on $\mathbb{R}^{m}$}

Here we recall the spectrum of the magnetic Dirac operator 
\begin{equation}
D_{\mathbb{R}^{m}}=\sum_{j=1}^{m}\left(\frac{\mu_{j}}{2}\right)^{\frac{1}{2}}\left[\gamma_{2j}\left(h\partial_{x_{j}}\right)+i\gamma_{2j-1}x_{j}\right]\in\Psi_{\textrm{cl}}^{1}\left(\mathbb{R}^{m};\mathbb{C}^{2^{m}}\right).\label{eq: magnetic Dirac Rm}
\end{equation}
on $\mathbb{R}^{m}$ computed in \cite{Savale2017-Koszul}. Its square
is computed in terms of the harmonic oscillator 
\begin{eqnarray}
D_{\mathbb{R}^{m}}^{2} & = & \mathtt{H}_{2}-ih\mathtt{R}_{2m+1},\:\textrm{with}\label{eq:square Euclidean Dirac}\\
\mathtt{H}_{2} & =\frac{1}{2} & \sum_{j=1}^{m}\mu_{j}\left[-\left(h\partial_{x_{j}}\right)^{2}+x_{j}^{2}\right]\label{eq:Harmonic oscillator}\\
\mathtt{R}_{2m+1} & =\frac{1}{2} & \sum_{j=1}^{m}\mu_{j}\left[\gamma_{2j-1}\gamma_{2j}\right].\nonumber 
\end{eqnarray}
Define the lowering and raising operators $A_{j}=h\partial_{x_{j}}+x_{j},$
$A_{j}^{*}=-h\partial_{x_{j}}+x_{j}$ for $1\leq j\leq m$, and the
Hermite functions
\begin{align}
\psi_{\tau,k}\left(x\right) & \coloneqq\psi_{\tau}\left(x\right)\otimes w_{k}\nonumber \\
\psi_{\tau}\left(x\right) & \coloneqq\frac{1}{\left(\pi h\right)^{\frac{m}{4}}\left(2h\right)^{\frac{\left|\tau\right|}{2}}\sqrt{\tau!}}\left[\Pi_{j=1}^{m}\left(A_{j}^{*}\right)^{\tau_{j}}\right]e^{-\frac{\left|x\right|^{2}}{2h}},\label{eq: Hermite functions}\\
 & \qquad\qquad\qquad\qquad\qquad\textrm{for }\tau=\left(\tau_{1},\tau_{2},\ldots,\tau_{m}\right)\in\mathbb{N}_{0}^{m}.\nonumber 
\end{align}
We also set 
\[
E_{\tau}\coloneqq\bigoplus_{b\in\left\{ 0,1\right\} ^{I_{\tau}}}\mathbb{C}\left[\prod_{j\in I_{\tau}}\left(\frac{c\left(w_{j}\right)A_{j}}{\sqrt{2\tau_{j}h}}\right)^{b_{j}}\psi_{\tau,0}\right]
\]
with $I_{\tau}$, $V_{\tau}$ as in \prettyref{eq: I_r-1}, \prettyref{eq: V_r-1}.
One clearly has an isomorphism 
\begin{eqnarray*}
\mathscr{I}_{\tau}:\Lambda^{*}V_{\tau} & \rightarrow & E_{\tau}\\
\mathscr{I}_{\tau}\left(\bigwedge_{j\in I_{\tau}}w_{j}^{b_{j}}\right) & \coloneqq & \prod_{j\in I_{\tau}}\left(\frac{c\left(w_{j}\right)A_{j}}{\sqrt{2\tau_{j}h}}\right)^{b_{j}}\psi_{\tau,0}.
\end{eqnarray*}
If $\mathtt{i}_{\tau}\coloneqq\mathscr{I}_{\tau}\mathtt{i}_{r_{\tau}}\mathscr{I}_{\tau}^{-1}:E_{\tau}^{\textrm{even/odd}}\rightarrow E_{\tau}^{\textrm{odd/even}}$,
the restriction of $D_{\mathbb{R}^{m}}$ to $E_{\tau}$ is of the
form 
\begin{eqnarray}
D_{\mathbb{R}^{m}} & = & \begin{bmatrix} & \left|r_{\tau}\right|\mathtt{i}_{\tau}\\
\left|r_{\tau}\right|\mathtt{i}_{\tau}
\end{bmatrix}.\label{eq: Dirac operator 2 by 2 block}
\end{eqnarray}
We may set 
\begin{align}
E_{\tau}^{\textrm{even/odd}} & \coloneqq\mathscr{I}_{\tau}\left(\Lambda^{\textrm{even/odd}}V_{\tau}\right)\nonumber \\
E_{\tau}^{\pm} & =\mathscr{I}_{\tau}\left(V_{\tau}^{\pm}\right)\label{eq: eigenspaces model Dirac}
\end{align}
and observe the Landau decomposition 
\begin{equation}
L^{2}\left(\mathbb{R}^{m};\mathbb{C}^{2^{m}}\right)=\mathbb{C}\left[\psi_{0,0}\right]\oplus\bigoplus_{\tau\in\mathbb{N}_{0}^{m}\setminus0}\left(E_{\tau}^{\textrm{even}}\oplus E_{\tau}^{\textrm{odd}}\right).\label{eq: Landau Levels}
\end{equation}
The spectrum of \prettyref{eq: magnetic Dirac Rm} is given by Prop.
2.1 of \cite{Savale2017-Koszul}.
\begin{prop}
\label{prop:eigenspaces magnetic Dirac}An orthogonal decomposition
of $L^{2}\left(\mathbb{R}^{m};\mathbb{C}^{2^{m}}\right)$ consisting
of eigenspaces of the magnetic Dirac operator $D_{\mathbb{R}^{m}}$
\prettyref{eq: magnetic Dirac Rm} is given by
\[
L^{2}\left(\mathbb{R}^{m};\mathbb{C}^{2^{m}}\right)=\mathbb{C}\left[\psi_{0,0}\right]\oplus\bigoplus_{\tau\in\mathbb{N}_{0}^{m}\setminus0}\left(E_{\tau}^{+}\oplus E_{\tau}^{-}\right).
\]
Here $E_{\tau}^{\pm}$, as in \prettyref{eq: eigenspaces model Dirac},
have dimension $2^{Z_{\tau}-1}$ and correspond to the eigenvalues
$\pm\sqrt{\mu.\tau h}$ respectively.
\end{prop}

\subsection{The Semi-classical calculus}

Finally, here we review the semi-classical pseudo-differential calculus
used throughout the paper with \cite{GuilleminSternberg-Semiclassical,Zworski}
being the detailed references. Much of this being reviewed in \cite{Savale2017-Koszul},
we only highlight some new aspects. Let $\mathfrak{gl}\left(l\right)$
denote the space of all $l\times l$ complex matrices. For $A=\left(a_{ij}\right)\in\mathfrak{gl}\left(l\right)$
we denote $\left|A\right|=\max_{ij}\left|a_{ij}\right|$. Denote by
$\mathcal{S}\left(\mathbb{R}^{n};\mathbb{C}^{l}\right)$ the space
of Schwartz maps $f:\mathbb{R}^{n}\rightarrow\mathbb{C}^{l}$. We
define the symbol space $S^{m}\left(\mathbb{R}^{2n};\mathbb{C}^{l}\right)$
as the space of maps $a:\left(0,1\right]_{h}\rightarrow C^{\infty}\left(\mathbb{R}_{x,\xi}^{2n};\mathfrak{gl}\left(l\right)\right)$
such that each of the semi-norms 
\[
\left\Vert a\right\Vert _{\alpha,\beta}:=\text{sup}_{\substack{x,\xi}
,h}\langle\xi\rangle^{-m+|\beta|}\left|\partial_{x}^{\alpha}\partial_{\xi}^{\beta}a(x,\xi;h)\right|
\]
is finite $\forall\alpha,\beta\in\mathbb{N}_{0}^{n}$. Such a symbol
is said to lie in the more refined class $a\in S_{\textrm{cl}}^{m}\left(\mathbb{R}^{2n};\mathbb{C}^{l}\right)$
if there exists an $h$-independent sequence $a_{k}$, $k=0,1,\ldots$
of symbols such that $a-\left(\sum_{k=0}^{N}h^{k}a_{k}\right)\in h^{N+1}S^{m}\left(\mathbb{R}^{2n};\mathbb{C}^{l}\right),\;\forall N.$
The symbol classes $S^{m}\left(\mathbb{R}^{2n};\mathbb{C}^{l}\right)$,
$S_{\textrm{cl}}^{m}\left(\mathbb{R}^{2n};\mathbb{C}^{l}\right)$
as above can be Weyl quantized to define one-parameter families of
operators $a^{W}\in\Psi^{m}\left(\mathbb{R}^{2n};\mathbb{C}^{l}\right),\Psi_{\textrm{cl}}^{m}\left(\mathbb{R}^{2n};\mathbb{C}^{l}\right)$
with Schwartz kernels given by 
\[
a^{W}\coloneqq\frac{1}{\left(2\pi h\right)^{n}}\int e^{i\left(x-y\right).\xi/h}a\left(\frac{x+y}{2},\xi;h\right)d\xi
\]
This class of operators is closed under the standard operations of
composition and formal-adjoint. Furthermore the class is invariant
under changes of coordinates and basis for $\mathbb{C}^{l}$. This
allows one to define invariant classes of operators $\Psi^{m}\left(X;E\right),\Psi_{\textrm{cl}}^{m}\left(X;E\right)$
on $C^{\infty}\left(X;E\right)$ associated to any complex, Hermitian
vector bundle $\left(E,h^{E}\right)$ on a smooth compact manifold
$X$. 

For $A\in\Psi_{\textrm{cl}}^{m}\left(X;E\right)$, its principal symbol
is well-defined as an element in $\sigma\left(A\right)\in S^{m}\left(X;\textrm{End}\left(E\right)\right)\subset C^{\infty}\left(X;\textrm{End}\left(E\right)\right).$
One has that $\sigma\left(A\right)=0$ if and only if $A\in h\Psi_{\textrm{cl}}^{m}\left(X;E\right)$.
We remark that $\sigma\left(A\right)$ is the restriction of standard
symbol in \cite{Zworski} to the refined class $\Psi_{\textrm{cl}}^{m}\left(X;E\right)$
and is locally given by the first coefficient $a_{0}$ in the expansion
in $h$ of its Weyl symbol. The principal symbol satisfies the basic
relations $\sigma\left(AB\right)=\sigma\left(A\right)\sigma\left(B\right)$,
$\sigma\left(A^{*}\right)=\sigma\left(A\right)^{*}$ with the formal
adjoints being defined with respect to the same Hermitian metric $h^{E}$.
The principal symbol map has an inverse given by the quantization
map $\textrm{Op}:S^{m}\left(X;\textrm{End}\left(E\right)\right)\rightarrow\Psi_{\textrm{cl}}^{m}\left(X;E\right)$
satisfying $\sigma\left(\textrm{Op}\left(a\right)\right)=a\in S^{m}\left(X;\textrm{End}\left(E\right)\right)$.
We remark that this quantization map is non-canonical and depends
on the choice of an open cover, with local trivializations for $E$,
and a subordinate partition of unity. We often use the alternate notation
$\textrm{Op}\left(a\right)=a^{W}$. For a scalar function $b\in S^{0}\left(X\right)$,
it is clear from the multiplicative property of the symbol that $\left[a^{W},b^{W}\right]\in h\Psi_{\textrm{cl}}^{m-1}\left(X;E\right)$
and we define $H_{b}\left(a\right)\coloneqq\frac{i}{h}\sigma\left(\left[a^{W},b^{W}\right]\right)\in S^{m-1}\left(X;\textrm{End}\left(E\right)\right)$.
We note that $H_{b}\left(a\right)$ depends on the quantization scheme,
in particular the local trivializations used in defining $\textrm{Op}$.
However one has $H_{b}\left(a\right)=\left\{ a,b\right\} $ is given
by the Poisson bracket when both sides are computed in the same defining
trivialization.

The wavefront set of an operator $A\in\Psi_{\textrm{cl}}^{m}\left(X;E\right)$
can be defined invariantly as a subset $WF\left(A\right)\subset\overline{T^{*}X}$
of the fibrewise radial compactification of its cotangent bundle.
If the local Weyl symbol of $A$ is given by $a$ then $\left(x_{0},\xi_{0}\right)\notin WF\left(A\right)$
if and only if there exists an open neighborhood $\left(x_{0},\xi_{0};0\right)\in U\subset\overline{T^{*}X}\times\left(0,1\right]_{h}$
such that $a\in h^{\infty}\left\langle \xi\right\rangle ^{-\infty}C^{k}\left(U;\mathbb{C}^{l}\right)$
for all $k$. The wavefront set satisfies the basic properties $WF\left(A+B\right)\subset WF\left(A\right)\cap WF\left(B\right)$,
$WF\left(AB\right)\subset WF\left(A\right)\cap WF\left(B\right)$
and $WF\left(A^{*}\right)=WF\left(A\right)$. The wavefront set $WF\left(A\right)=\emptyset$
is empty if and only if $A\in h^{\infty}\Psi^{-\infty}\left(X;E\right)$.
We say that two operators $A=B$ microlocally on $U\subset\overline{T^{*}X}$
if $WF\left(A-B\right)\cap U=\emptyset$. 

An operator $A\in\Psi_{\textrm{cl}}^{m}\left(X;E\right)$ is said
to be elliptic if $\left\langle \xi\right\rangle ^{m}\sigma\left(A\right)^{-1}$
exists and is uniformly bounded on $T^{*}X$. If $A\in\Psi_{\textrm{cl}}^{m}\left(X;E\right)$,
$m>0$, is formally self-adjoint such that $A+i$ is elliptic then
it is essentially self-adjoint (with domain $C_{c}^{\infty}\left(X;E\right)$)
as an unbounded operator on $L^{2}\left(X;E\right)$. Its resolvent
$\left(A-z\right)^{-1}\in\Psi_{\textrm{cl}}^{-m}\left(X;E\right)$,
$z\in\mathbb{C}$, $\textrm{Im}z\neq0$, now exists and is pseudo-differential
by an application of Beals's lemma. Given a Schwartz function $f\in\mathcal{S}\left(\mathbb{R}\right)$,
the Helffer-Sjöstrand formula now expresses the function $f\left(A\right)$
of such an operator in terms of its resolvent and an almost analytic
continuation $\tilde{f}$ via
\[
f\left(A\right)=\frac{1}{\pi}\int_{\mathbb{C}}\bar{\partial}\tilde{f}\left(z\right)\left(A-z\right)^{-1}dzd\bar{z}.
\]
We then also have $WF\left(f\left(A\right)\right)\subset\Sigma_{\textrm{spt}\left(f\right)}^{A}\coloneqq\bigcup_{\lambda\in\textrm{spt}\left(f\right)}\Sigma_{\lambda}^{A}$
where 
\begin{equation}
\Sigma_{\lambda}^{A}=\left\{ \left(x,\xi\right)\in T^{*}X|\det\left(\sigma\left(A\right)\left(x,\xi\right)-\lambda I\right)=0\right\} .\label{eq:energy level}
\end{equation}
is classical $\lambda$-energy level of $A$. 

\subsubsection{The class $\Psi_{\delta}^{m}\left(X;E\right)$}

We shall need also more exotic class of scalar symbols $S_{\delta}^{m}\left(\mathbb{R}^{2n};\mathbb{C}\right)$
defined for each $0<\delta<\frac{1}{2}$. A function $a:\left(0,1\right]_{h}\rightarrow C^{\infty}\left(\mathbb{R}_{x,\xi}^{2n};\mathbb{C}\right)$
is said to be in this class if and only if 
\begin{equation}
\left\Vert a\right\Vert _{\alpha,\beta}:=\text{sup}_{\substack{x,\xi}
,h}\langle\xi\rangle^{-m+|\beta|}h^{\left(\left|\alpha\right|+\left|\beta\right|\right)\delta}\left|\partial_{x}^{\alpha}\partial_{\xi}^{\beta}a(x,\xi;h)\right|\label{eq: delta pseudodifferential estimates}
\end{equation}
is finite $\forall\alpha,\beta\in\mathbb{N}_{0}^{n}.$ This class
of operators is closed under the standard operations of composition,
adjoint and changes of coordinates allowing the definition of the
exotic pseudo-differential algebra $\Psi_{\delta}^{m}\left(X\right)$
on a compact manifold. The class $S_{\delta}^{m}\left(X\right)$ is
a family of functions $a:\left(0,1\right]_{h}\rightarrow C^{\infty}\left(T^{*}X;\mathbb{C}\right)$
satisfying the estimates \prettyref{eq: delta pseudodifferential estimates}
in every coordinate chart and induced trivialization. Such a family
can be quantized to $a^{W}\in\Psi_{\delta}^{m}\left(X\right)$ satisfying
$a^{W}b^{W}=\left(ab\right)^{W}+h^{1-2\delta}\Psi_{\delta}^{m+m'-1}\left(X\right)$,
$\frac{i}{h^{1-2\delta}}\sigma\left(\left[a^{W},b^{W}\right]\right)=\left[\left\{ a,b\right\} \right]$
for another $b\in S_{\delta}^{m'}\left(X\right)$. The operators in
$\Psi_{\delta}^{0}\left(X\right)$ are uniformly bounded on $L^{2}\left(X\right)$.
Finally, the wavefront an operator $A\in\Psi_{\delta}^{m}\left(X;E\right)$
is similarly defined and satisfies the same basic properties as before.

\section{\label{sec:Dynamical-partitions}Dynamical partitions}

The trace formula of \prettyref{thm:main trace expansion-1} was proved
in \cite{Savale2017-Koszul} assuming $\theta$ to be supported in
a sufficiently small interval near $0$. In this case only the local
contribution to the trace \prettyref{eq: gutwiller local term} appears.
It now thus suffices to consider $\theta$ supported away from $0$
and prove the following. 
\begin{lem}
\label{lem: O(h infty) LEMMA} For $\theta\in C_{c}^{\infty}\left(\left(T_{0},\infty\right);\left[0,1\right]\right)$,
$T_{0}>0$, one has 
\begin{align}
\textrm{tr}\left[f\left(\frac{D}{\sqrt{h}}\right)\left(\mathcal{F}_{h}^{-1}\theta\right)\left(\lambda\sqrt{h}-D\right)\right] & =\label{eq: non-local dynamical trace}\\
\textrm{tr}\left[f\left(\frac{D}{\sqrt{h}}\right)\frac{1}{h}\check{\theta}\left(\frac{\lambda\sqrt{h}-D}{h}\right)\right] & =\sum_{\gamma}e^{\frac{i}{h}L_{\gamma}}e^{i\frac{\pi}{2}\mathfrak{m}_{\gamma}}\sum_{j=0}^{N-2m-2}h^{j/2}\sum_{k=0}^{j}\lambda^{k}A_{\gamma,j,k}\theta\left(T_{\gamma}\right)\nonumber \\
 & +O\left(h^{N/2-m-1}\right)\label{eq: gutzwiller dyn. term. (hinfty lemma)}
\end{align}
 for all $\lambda\in\mathbb{R}$, with the right hand side above being
the same as the dynamical contribution \prettyref{eq: gutzwiller dynamical term}
in \prettyref{eq: gutzwiller remainder}.
\end{lem}
To prove \prettyref{lem: O(h infty) LEMMA} we shall split the trace
via a microlocal partition of unity adapted to the Reeb dynamics.
To this end we first need a description of the contact form in a neighborhood
of each Reeb orbit.

\subsection{Normal structure for the contact form}

Let $\gamma\subset X$ be a primitive closed Reeb orbit with period
$T_{\gamma}$. For a point $p\in\gamma$, the linearized Poincare
return map $P_{\gamma}^{+}:H_{p}\rightarrow H_{p}$ restricted to
the contact hyperplane then has the decomposition \prettyref{eq: symplectic decomposition}
as before. For each corresponding eigenvalue in this decomposition,
define the following model quadratic functions on $\mathbb{R}^{2m}$
\begin{eqnarray}
\textrm{Elliptic case:} & Q_{j}^{e}= & \frac{1}{2}\left(x_{j}^{2}+x_{j+m}^{2}\right),\quad1\leq j\leq N_{e}\nonumber \\
\textrm{Hyperbolic case:} & Q_{j}^{h}= & x_{N_{e}+j}x_{N_{e}+j+m},\quad1\leq j\leq N_{h}\nonumber \\
\textrm{Loxodromic case:} & Q_{j}^{l,\textrm{Re}}= & x_{m-2j+2}x_{2m-2j+1}-x_{m-2j+1}x_{2m-2j+2},\quad1\leq j\leq N_{l}\nonumber \\
 & Q_{j}^{l,\textrm{Im}}= & x_{m-2j+1}x_{2m-2j+1}+x_{m-2j+2}x_{2m-2j+2},\quad1\leq j\leq N_{l}\label{eq: quadratics-1}
\end{eqnarray}
Also let $Q^{h,-}=\frac{\pi}{2}\sum_{j=1}^{N_{h}^{-}}\left(x_{N_{e}+j}^{2}+x_{N_{e}+j+m}^{2}\right)$
be the quadratic whose Hamiltonian flow rotates negative hyperbolic
blocks by $\pi$. 

In the theorem below we let $\gamma^{0}\coloneqq S^{1}\times\left\{ 0\right\} \subset S^{1}\times\mathbb{R}^{2m}$.
We shall use $\theta$ or $x_{0}$ interchangeably to denote the circular
$S^{1}$ variable. We also let $\chi^{-}\in C_{c}^{\infty}\left(0,\frac{1}{2}\right)_{\theta}$,
$\chi^{+}\in C_{c}^{\infty}\left(\frac{1}{2},1\right)_{\theta}$ be
non-negative functions with total integral $1$. We now have the following
normal structure for the contact form $a$ near a nonresonant $\gamma$. 
\begin{prop}
\label{prop:normal structure contact form}There exists a diffeomorphism
$\kappa:\varOmega_{\gamma}^{0}\rightarrow\varOmega_{\gamma}$ between
some neighborhood of $\gamma^{0}\subset\varOmega_{\gamma}^{0}$ and
some neighborhood of the Reeb orbit $\gamma\subset\varOmega_{\gamma}$
such that 
\begin{equation}
\kappa^{*}a=\underbrace{\left(T_{\gamma}+\chi^{-}Q^{h,-}+\chi^{+}\varphi^{+}\right)}_{\eqqcolon\varphi}d\theta+\frac{1}{2}\sum_{j=1}^{m}\left(x_{j}dx_{j+m}-x_{j+m}dx_{j}\right)\label{eq: normal form contact a-1}
\end{equation}
modulo $O\left(Q^{\infty}\right)$. Here $\varphi^{+}=\varphi^{+}\left(Q\right)$
in \prettyref{eq: normal form contact a-1} is a function on $\mathbb{R}^{2m}$
of the quadratics \prettyref{eq: quadratics-1} with linear term
\begin{equation}
\varphi^{+}=\sum_{j=1}^{N_{e}}\beta_{j}Q_{j}^{e}+\sum_{j=1}^{N_{h}}\alpha_{j}Q_{j}^{h}+\sum_{j=1}^{N_{l}}\left(\alpha_{j}^{0}Q_{j}^{l,\textrm{Re}}+\beta_{j}^{0}Q_{j}^{l,\textrm{Im}}\right)+O\left(Q^{2}\right).\label{eq: function of quadratics-1}
\end{equation}
\end{prop}
\begin{proof}
Choose Darboux coordinates $\left(x,y;z\right)$ centered at $p$
in which $a=dz+\frac{1}{2}\sum_{j=1}^{m}\left(x_{j}dx_{j+m}-x_{j+m}dx_{j}\right)$.
Then $\Sigma=\left\{ z=0\right\} \subset X$ defines a local Poincare
section transverse to the Reeb vector field $\partial_{z}$ in these
coordinates. The Reeb flow gives rise to a symplectic return map and
a return time function 
\begin{eqnarray}
P_{\Sigma}:\left(\Sigma,da\right) & \rightarrow & \left(\Sigma,da\right)\nonumber \\
T_{\Sigma}:\Sigma & \rightarrow & \mathbb{R}\label{eq: return map and time}
\end{eqnarray}
which satisfy the relation 
\begin{equation}
P_{\Sigma}^{*}a-a=dT_{\Sigma}\label{eq: return time map relation-1}
\end{equation}
 (cf. \cite{Francoise-Guillemin} Prop. 2.1). The linearization of
$P_{\Sigma}$ at $0$ being $P_{\gamma}^{+}$, has the same spectrum
$\textrm{Spec}\left(P_{\gamma}^{+}\right)$. Under the nonresonance
assumption, such a symplectic map is a composition of the Hamiltonian
diffeomorphisms 
\begin{equation}
P_{\Sigma}=e^{H_{\varphi^{+}}}\circ e^{H_{Q^{h,-}}},\label{eq: Poincare map is Hamiltonian-1}
\end{equation}
 modulo $O\left(Q^{\infty}\right)$, for a function $\varphi^{+}$
of the form \prettyref{eq: function of quadratics-1} (cf. \cite{Iantchenko-Sjostrand,Sternberg-61-dynbook}).
We now compute $\left(e^{H_{Q^{h,-}}}\right)^{*}a=a$ and
\begin{eqnarray}
\left.\frac{d}{dt}\left(e^{tH_{\varphi^{+}}}\right)^{*}a\right|_{t=0} & = & i_{H_{\varphi^{+}}}da+di_{H_{\varphi^{+}}}a\nonumber \\
 & = & d\varphi^{+}-d\left[\frac{1}{2}\sum_{j=1}^{m}\left(x_{j}\varphi_{x_{j}}^{+}+x_{j+m}\varphi_{x_{j+m}}^{+}\right)\right].\label{eq: lie derivative a-1}
\end{eqnarray}
From \prettyref{eq: return time map relation-1}, \prettyref{eq: Poincare map is Hamiltonian-1}
and \prettyref{eq: lie derivative a-1} we now have 
\begin{equation}
T_{\Sigma}=T_{\gamma}+\underbrace{\varphi^{+}-\frac{1}{2}\sum_{j=1}^{m}\left(x_{j}\varphi_{x_{j}}^{+}+x_{j+m}\varphi_{x_{j+m}}^{+}\right)}_{T_{\Sigma}^{+}\coloneqq}.\label{eq: return time-1}
\end{equation}

Next, let us compute the return map and return time, associated to
the Poincare section $\Sigma_{0}=\left\{ \theta=0\right\} $, for
the model contact form \prettyref{eq: normal form contact a-1} on
$S^{1}\times\mathbb{R}^{2m}$. Its Reeb vector field $R_{0}$ is easily
computed 
\begin{align}
R_{0} & =\begin{cases}
\frac{1}{T_{\gamma}}\left(\partial_{\theta}+\chi^{-}H_{Q^{h,-}}\right), & \theta\in\left(0,\frac{1}{2}\right)\\
\frac{1}{T_{\gamma}+\chi^{+}T_{\Sigma}^{+}}\left(\partial_{\theta}+\chi^{+}H_{\varphi^{+}}\right), & \theta\in\left(\frac{1}{2},1\right).
\end{cases}\label{eq: model Reeb-1}
\end{align}
To compute the return map and time, first note that each of the quadratics
\prettyref{eq: quadratics-1} Poisson commutes with $\varphi^{+}$
of the form \prettyref{eq: function of quadratics-1}. Hence each
of these quadratics is constant along the Hamilton flow of $H_{\varphi^{+}}$.
An easy calculation upon differentiating \prettyref{eq: function of quadratics-1}
yields that the quantity $\frac{1}{2}\sum_{j=1}^{m}\left(x_{j}\varphi_{x_{j}}^{+}+x_{j+m}\varphi_{x_{j+m}}^{+}\right)$
maybe expressed in terms of the same quadratic functions and is thus
also constant along the Hamilton flow of $H_{\varphi^{+}}$. Thus
$T_{\Sigma}^{+}$ \prettyref{eq: return time-1} is constant along
the Hamilton flow of $H_{\varphi^{+}}$. The return map and time of
\prettyref{eq: model Reeb-1} are now easily computed to be $e^{H_{\varphi^{+}}}\circ e^{H_{Q^{h,-}}}$
and $T_{\Sigma}$ respectively. 

Finally, with the return map and time of the Poincare section $\Sigma$
being the same as in the model case, a Moser style argument maybe
applied to complete the proof.
\end{proof}
In the proof above we have modified arguments from \cite{Guillemin-wavetraceinv.}
Thm. 2.7 from the elliptic case. A general non-degenerate case appears
for geodesic flows in \cite{Zelditch98-waveinv}. We shall call a
chart $\kappa:\varOmega_{\gamma}^{0}\rightarrow\varOmega_{\gamma}$
given by the Proposition above a Darboux-Reeb chart near $\gamma$. 

Next fix a constant $\delta\in\left(0,\frac{1}{2}\right)$. Define
a dilation on each Darboux-Reeb chart
\begin{align*}
\varrho_{\delta}:\varOmega_{\gamma}^{0} & \rightarrow\varOmega_{\gamma}^{0}\\
\varrho_{\delta}\left(x_{0};x_{1},\ldots,x_{2m}\right) & =\left(x_{0};h^{\delta}x_{1},\ldots,h^{\delta}x_{2m}\right)
\end{align*}
and also denote by $\varrho_{\delta}:\varOmega_{\gamma}\rightarrow\varOmega_{\gamma}$
the corresponding dilation of $\varOmega_{\gamma}$. For each subset
$S$ of $\varOmega_{\gamma}^{0}$ (or $\varOmega_{\gamma}$) we denote
by $S^{\delta}\coloneqq\varrho_{\delta}\left(S\right)$ its ($h$-dependent)
image under the dilation. We also denote by $\tilde{S}\subset T^{*}X$
the inverse image under the projection $\pi:T^{*}X\rightarrow X$.
Letting $\Gamma\coloneqq\left\{ \gamma_{v}\right\} _{v=1}^{M}$ be
the set of all primitive Reeb orbits, we set $\varOmega\coloneqq\cup_{v=1}^{M}\varOmega_{\gamma_{v}}$.
Below let $\Gamma\subset\Omega\subset\varOmega$ be any subcover of
the system of Darboux-Reeb charts and denote $C_{\varepsilon,T}\coloneqq B_{\mathbb{R}^{2m}}\left(\varepsilon\right)\times\left(-T,T\right)_{x_{0}}\subset\mathbb{R}_{x}^{n}$
the cylinder of radius $\varepsilon$ and height $T$ in Euclidean
space. We now have the following elementary lemma.
\begin{lem}
\label{lem: construction of Darboux charts}For each $\delta\in\left(0,\frac{1}{2}\right)$,
$T>0$ there exists an $\varepsilon>0$ of the following significance:
each point $x\in X\setminus\Omega^{\delta}$ has a Darboux chart $\varphi_{x}:N_{x}\xrightarrow{\sim}C_{\varepsilon h^{\delta},T}\subset\mathbb{R}^{n}$,
$N_{x}\subset X\setminus\Gamma$, centered at $x$ satisfying 
\begin{equation}
\left(\varphi_{x}^{-1}\right)^{*}a=dx_{0}+\sum_{j=1}^{m}\left(x_{j}dx_{j+m}-x_{j+m}dx_{j}\right).\label{eq: contact form long tubes}
\end{equation}
\end{lem}
\begin{proof}
The Reeb trajectory $\gamma_{x}\coloneqq e^{tR}\left(x\right),-T<t<T$,
$x\in X\setminus\Omega^{\delta}$, being non-self-intersecting the
existence of a chart of height $T$ is similar to the Darboux theorem.
It only remains to show that one may choose a chart into a cylinder
of radius $\varepsilon h^{\delta}$ for $\varepsilon$ uniform in
$h$. By compactness, a radius of an $h$-independent size $\varepsilon=O\left(1\right)$
works for points in the $h$-independent set $x\in X\setminus\Omega_{0}$,
for $\Omega_{0}\subset\Omega$. For points $x\in\bar{\Omega}_{0}\setminus\Omega^{\delta}$,
non-resonance implies that the linearizations $\left(P_{\gamma}^{+}\right)^{k}-\left(P_{\gamma}^{+}\right)^{l}$,
$k,l\in\mathbb{Z}$, of the Poincare return maps $P_{\Sigma}^{k}-P_{\Sigma}^{l}$
\prettyref{eq: return map and time} at $0$ are invertible. Here
the Poincare sections are again given by $\left\{ x_{0}=0\right\} $
in terms of the Darboux-Reeb coordinates on $\Omega_{0}$. One may
hence shrink $\Omega_{0}$ to arrange $\left\Vert P_{\Sigma}^{k}\left(x\right)-P_{\Sigma}^{l}\left(x\right)\right\Vert \geq C\left\Vert \left(x_{1},\ldots x_{2m}\right)\right\Vert ,\:\forall x\in\bar{\Omega}_{0}$,
$\left|k\right|\leq N_{T},\left|l\right|\leq N_{T}$, where $N_{T}\coloneqq\max_{\gamma\in\Gamma}\frac{T}{T_{\gamma}}$.
From here one finds a uniform $\varepsilon$ such that $\forall x\in\left(\bar{\Omega}_{0}\setminus\Omega_{\gamma}^{\delta}\right)\cap\left\{ x_{0}=0\right\} $
the first $N$ iterates under $P_{\Sigma}$ of the ball $B_{\mathbb{R}^{2m}}\left(\varepsilon_{x}h^{\delta}\right)$
are disjoint. The Reeb flow-outs $e^{tR}\left[B_{\mathbb{R}^{2m}}\left(\varepsilon_{x}h^{\delta}\right)\right],-T<t<T$,
of the balls being non-self-intersecting, a chart satisfying \prettyref{eq: contact form long tubes}
comes from a Moser style argument. 
\end{proof}
For each Darboux chart $\varphi_{x}:N_{x}\xrightarrow{\sim}C_{\varepsilon h^{\delta},T}\subset\mathbb{R}^{n}$
as above we set $N_{x}^{0}\coloneqq\varphi_{x}^{-1}\left(C_{\frac{\varepsilon h^{\delta}}{8},\frac{T}{8}}\right)$.
The chart is called trivialized if it comes equipped with an orthonormal
trivialization of the spin bundle. Below for each $h$-independent
constant $c$ we denote by a shorthand the $h$-dependent constant
$c_{\delta}\coloneqq ch^{\delta}$ .

We now come to the construction of dynamical partitions. Below, the
energy levels $\Sigma_{I}^{D}$ above are as in \prettyref{eq:energy level}.
Let $T>0$, $\tau>0$, $\delta\in\left(0,\frac{1}{2}\right)$ and
$\Gamma\subset\Omega\subset\varOmega$ be a subcover of the system
of Darboux-Reeb charts as before. A $\left(\Omega,\tau,\delta\right)$-microlocal
partition of unity is defined to be a collection of zeroth-order pseudo-differential
operators $\mathcal{P}=\left\{ A_{u}\in\Psi_{\delta}^{0}\left(X\right)|0\leq u\leq N_{h}\right\} \cup\left\{ B_{v}\in\Psi_{\delta}^{0}\left(X\right)|1\leq v\leq M\right\} $
satisfying 
\begin{eqnarray}
\sum_{u=0}^{N_{h}}A_{u}+\sum_{v=1}^{M}B_{v} & = & 1\nonumber \\
N_{h} & = & O\left(h^{-\delta}\right)\nonumber \\
WF\left(A_{0}\right)\subset & U_{0}\subset & \overline{T^{*}X}\setminus\Sigma_{\left[-\frac{\tau_{\delta}}{64},\frac{\tau_{\delta}}{64}\right]}^{D}\nonumber \\
WF\left(A_{u}\right)\Subset & U_{u}\subset & \Sigma_{\left[-\tau_{\delta},\tau_{\delta}\right]}^{D}\cap\tilde{N}_{x_{u}}^{0},\;1\leq u\leq N\nonumber \\
WF\left(B_{v}\right)\Subset & V_{v}\subset & \Sigma_{\left[-\tau_{\delta},\tau_{\delta}\right]}^{D}\cap\tilde{\Omega}_{\gamma_{v}}^{\delta},\;1\leq v\leq M\label{eq: microlocal partition of unity-1-1-1}
\end{eqnarray}
for some open cover $\left\{ U_{u}\right\} _{u=0}^{N}\cup\left\{ V_{v}\right\} _{v=1}^{M}$
of $T^{*}X$ and for some collection of trivialized Darboux charts
$N\coloneqq\left\{ N_{x_{u}}\right\} _{u=1}^{N}\subset X\setminus\Gamma$
. For such a partition $\mathcal{P}$ define the pairs of indices
\begin{align}
I_{\mathcal{P}} & =\left\{ \left(u,u'\right)|u\leq u',\:WF\left(A_{u}\right)\cap WF\left(A_{u'}\right)\neq\emptyset\right\} \nonumber \\
J_{\mathcal{P}} & =\left\{ \left(u,v\right)|WF\left(A_{u}\right)\cap WF\left(B_{v}\right)\neq\emptyset\right\} .\label{eq: index sets}
\end{align}
 An augmentation $\left(\mathcal{P};\mathcal{V},\mathcal{W}\right)$
of this partition consists of an additional collection of open sets
$\mathcal{V}=\left\{ V_{uu'}^{1}\right\} _{\left(u,u'\right)\in I_{\mathcal{P}}}\cup\left\{ V_{uv}^{2}\right\} _{\left(u,v\right)\in J_{\mathcal{P}}}$,
$\mathcal{W}=\left\{ W_{uu'}^{1}\right\} _{\left(u,u'\right)\in I_{\mathcal{P}}}\cup\left\{ W_{uv}^{2}\right\} _{\left(u,v\right)\in J_{\mathcal{P}}}$
satisfying 
\begin{eqnarray}
WF\left(A_{u}\right)\cap WF\left(A_{u'}\right) & \subset & W_{uu'}^{1}\nonumber \\
 &  & \cap\nonumber \\
WF\left(A_{u}\right)\cup WF\left(A_{u'}\right) & \subset & V_{uu'}^{1}\Subset\Sigma_{\left[-2\tau_{\delta},2\tau_{\delta}\right]}^{D}\cap\tilde{N}_{x_{u}},\nonumber \\
WF\left(A_{u}\right)\cap WF\left(B_{v}\right) & \subset & W_{uv}^{2}\nonumber \\
 &  & \cap\nonumber \\
WF\left(A_{u}\right)\cup WF\left(B_{v}\right) & \subset & V_{uv}^{2}\Subset\Sigma_{\left[-2\tau_{\delta},2\tau_{\delta}\right]}^{D}\cap\tilde{N}_{x_{u}}.\label{eq:augmentation}
\end{eqnarray}
Next with $d=\sigma\left(D\right)$, for each pair of indices in \prettyref{eq: index sets}
we set 
\begin{align}
T_{uu'} & \coloneqq\frac{1}{\inf_{\left(g,\mathtt{v}\right)\in\mathcal{G}_{uu'}\times S_{\delta}^{0}\left(X;U\left(S\right)\right)}\left|H_{g,\mathtt{v}}d\right|},\label{eq: exit time 1}\\
S_{uv} & \coloneqq\frac{1}{\inf_{\left(g,\mathtt{v}\right)\in\mathcal{H}_{uv}\times S_{\delta}^{0}\left(X;U\left(S\right)\right)}\left|H_{g,\mathtt{v}}d\right|},\quad\textrm{ with}\label{eq: exit time 2}\\
\mathcal{G}_{uu'} & \coloneqq\left\{ g\in S_{\delta}^{0}\left(X;\left[0,1\right]\right)|\left.g\right|_{W_{uu'}^{1}}=1,\;\left.g\right|_{\left(V_{uu'}^{1}\right)^{c}}=0\right\} \label{eq: set of exit functions 1}\\
\mathcal{H}_{uv} & \coloneqq\left\{ g\in S_{\delta}^{0}\left(X;\left[0,1\right]\right)|\left.g\right|_{W_{uv}^{2}}=1,\;\left.g\right|_{\left(V_{uv}^{2}\right)^{c}}=0\right\} \label{eq: set of exit functions 2}
\end{align}
and $\left|H_{g,\mathtt{v}}d\right|\coloneqq\sup\left\Vert \left\{ \mathtt{v}^{*}d\mathtt{v},g\right\} \right\Vert $
with the bracket being computed in terms of the chosen and induced
trivialization/coordinates on $N_{x_{u}},\tilde{N}_{x_{u}}$. A function
in $\mathcal{G}_{uu'}$ or $\mathcal{H}_{uv}$ shall be referred to
as a trapping/microlocal weight function.

Finally, the \textit{extension/trapping time} of an augmented $\left(\Omega,\tau,\delta\right)$-partition
$\left(\mathcal{P};\mathcal{V},\mathcal{W}\right)$ is set to be 
\begin{equation}
T_{\left(\mathcal{P};\mathcal{V},\mathcal{W}\right)}\coloneqq\min\left\{ \min\left\{ T_{uu'}\right\} _{\left(u,u'\right)\in I_{\mathcal{P}}},\min\left\{ S_{uv}\right\} _{\left(u,v\right)\in J_{\mathcal{P}}}\right\} .\label{eq: total extension time}
\end{equation}

\begin{prop}
\label{prop: Large Extension time}Let $T>0$, $\delta\in\left(0,\frac{1}{2}\right)$
and $\Gamma\subset\Omega\subset\varOmega$ be a subcover. Then for
each $\tau$ sufficiently small one has an augmented $\left(\Omega,\tau,\delta\right)$-partition
of unity $\left(\mathcal{P};\mathcal{V},\mathcal{W}\right)$ with
\begin{equation}
T_{\left(\mathcal{P};\mathcal{V},\mathcal{W}\right)}>T.\label{eq: extension time arb large-1-1}
\end{equation}
\end{prop}
\begin{proof}
By \prettyref{lem: construction of Darboux charts} there exists $\varepsilon>0$
such that each $x\in X\setminus\Omega^{\delta}$ has a Darboux chart
$\varphi_{x}:N_{x}\xrightarrow{\sim}C_{\varepsilon h^{\delta},T}\subset\mathbb{R}^{n}$
centered at $x$ of radius $\varepsilon_{\delta}=\varepsilon h^{\delta}$
and height $T$. Next with $\left(x'',\xi''\right)=\left(x_{m+1},\ldots,x_{2m};\xi_{m+1},\ldots,\xi_{2m}\right)$
being a subset of the coordinates on $\mathbb{R}_{x,\xi}^{2n}$ set
$C_{\varepsilon_{\delta},T}''\coloneqq\left\{ x''^{2}+\xi''^{2}<\varepsilon_{\delta}^{2},\:-T<x_{0}<T\right\} \subset\mathbb{R}_{x,\xi}^{2n}$.
Also for each $\tau>0$, set 
\begin{align}
U_{\varepsilon_{\delta},\tau_{\delta},T}\coloneqq & \left\{ \xi_{0}^{2}+2\sum_{j=1}^{m}\mu_{j}\left(x_{j}^{2}+\xi_{j}^{2}\right)<\tau_{\delta}^{2},\right.\nonumber \\
 & \left.\qquad\:x''^{2}+\xi''^{2}<\left(\frac{\varepsilon_{\delta}}{2}\right)^{2},\:-\frac{T}{2}<x_{0}<\frac{T}{2}\right\} \subset C_{\varepsilon_{\delta},T}''.\label{eq: collar nbhd of charc}
\end{align}
Then by the preliminary Birkhoff normal form procedure of \cite{Savale2017-Koszul}
Sec. 5 (eqns 5.1, 5.5, 5.6, 5.7, 5.8) there exists $0<\tau\ll1$ sufficiently
small of the following significance: there is a neighborhood $M_{u}\subset\tilde{N}_{x_{u}}$
of $\tilde{N}_{x_{u}}^{0}\cap\Sigma_{0}^{D}$ , a Hamiltonian symplectomorphism
\begin{align*}
\kappa_{u}\coloneqq e^{H_{f_{1}}}\circ e^{H_{f_{0}}}: & U_{\varepsilon_{\delta},\tau_{\delta},T}\rightarrow M_{u}\\
\kappa_{u}\left(x_{0},0,x'';0,0,\xi''\right) & =\left(x_{0},-\frac{\xi''}{\sqrt{2}},\frac{x''}{\sqrt{2}};-1,\frac{x''}{\sqrt{2}},\frac{\xi''}{\sqrt{2}}\right)
\end{align*}
(see \cite{Savale2017-Koszul} pgs. 1812-1813 for $f_{0},f_{1}$)
in terms of the chosen coordinates on each, a self-adjoint endomorphism
$c_{A}\in C^{\infty}\left(U_{\varepsilon_{\delta},\tau_{\delta},T};i\mathfrak{u}\left(2^{m}\right)\right)$
and functions $\left\{ r_{j}\in C^{\infty}\left(U_{\varepsilon_{\delta},\tau_{\delta},T}\right)\right\} _{j=0}^{2m}$
vanishing to second order along $\Sigma_{0}^{D}$ such that 
\begin{align}
e^{ic_{A}}\left[\left(e^{H_{f_{1}}}\circ e^{H_{f_{0}}}\right)^{*}d\right]e^{-ic_{A}} & =H_{1}+\sigma_{j}r_{j},\quad\textrm{ with}\nonumber \\
H_{1} & \coloneqq\xi_{0}\sigma_{0}+\sum_{j=1}^{m}\left(2\mu_{j}\right)^{\frac{1}{2}}\left(x_{j}\sigma_{2j-1}+\xi_{j}\sigma_{2j}\right).\label{eq: conjugate symbol}
\end{align}
Taylor expand $r_{0}=\sum r_{00}\left(x_{0},x'';\xi_{0},\xi''\right)\xi_{0}^{2}+r_{j}^{1}x_{j}+r_{j}^{2}\xi_{j}$,
with $r_{j}^{1},r_{j}^{2}$ vanishing to first order along $\Sigma_{0}^{D}$.
A further conjugation of the above \prettyref{eq: conjugate symbol}
by $e^{\left[r_{j}^{1}\left(2\mu_{j}\right)^{-\frac{1}{2}}\sigma_{2j-1}+r_{j}^{2}\left(2\mu_{j}\right)^{-\frac{1}{2}}\sigma_{2j}\right]\sigma_{0}}$
sets $r_{j}^{1}=r_{j}^{2}=0$ while a symplectic change of variables
in $x_{0}$ sets $r_{00}=0$. Now set 
\begin{align}
\left(\tilde{\theta}_{0},\tilde{\theta}_{1},\ldots,\tilde{\theta}_{2m}\right) & \coloneqq\left(\xi_{0},\left(2\mu_{1}\right)^{\frac{1}{2}}x_{1},\left(2\mu_{1}\right)^{\frac{1}{2}}\xi_{1},\ldots,\left(2\mu_{m}\right)^{\frac{1}{2}}x_{m},\left(2\mu_{m}\right)^{\frac{1}{2}}\xi_{m}\right)\label{eq: components symbol}\\
 & \qquad+\left(0,r_{1},\ldots,r_{2m}\right)\nonumber \\
\tilde{\theta}' & =\left(\tilde{\theta}_{1},\ldots,\tilde{\theta}_{2m}\right)\nonumber 
\end{align}
 and note from \prettyref{eq: conjugate symbol} that the eigenvalues
of the symbol $d$ are $\pm\left|\tilde{\theta}\right|$. We clearly
have
\begin{align}
\kappa_{u}^{-1}\left(M_{u}\right)\cap\Sigma_{0}^{D} & =U_{\varepsilon_{\delta},\tau_{\delta},T}\cap\left\{ \tilde{\theta}=0\right\} \nonumber \\
 & =U_{\varepsilon_{\delta},\tau_{\delta},T}\cap\left\{ \xi_{0}=x_{1}=\xi_{1}=\ldots=x_{m}=\xi_{m}=0\right\} \label{eq:cords on charac}
\end{align}
and we may set 
\begin{align}
\theta_{j} & =\frac{\tilde{\theta}_{j}}{\left|\tilde{\theta}\right|}\in C^{\infty}\left(U_{\varepsilon_{\delta},\tau_{\delta},T}\setminus\Sigma_{0}^{D};S^{n-1}\right).\label{eq:circle valued function}
\end{align}
If we denote by $o_{N}$ the set of functions that vanish to order
$N$ along $\Sigma_{0}^{D}$, we have 
\begin{align}
\left\{ \tilde{\theta}_{0},x_{0}\right\} -1=o_{1}\nonumber \\
\left\{ \tilde{\theta}_{j},x_{0}\right\} =o_{1}, & j\geq1,\nonumber \\
\left\{ \tilde{\theta}_{j},x''\right\} =o_{1}, & j\geq0,\nonumber \\
\left\{ \tilde{\theta}_{j},\xi''\right\} =o_{1}, & j\geq0,\nonumber \\
\left\{ \tilde{\theta}_{0},\tilde{\theta}_{j}\right\} =o_{2}, & j\geq0,\nonumber \\
\left\{ \tilde{\theta}_{j},\tilde{\theta}_{k}\right\} \quad\textrm{or }\left\{ \tilde{\theta}_{j},\tilde{\theta}_{k}\right\} -1=o_{1}, & k>j\geq0,\nonumber \\
r_{j}=o_{2}, & j\geq0.\label{eq: orders of commutators}
\end{align}
By \prettyref{eq: collar nbhd of charc}, \prettyref{eq:cords on charac}
$U_{\varepsilon_{\delta},\tau_{\delta},T}$ denotes a collar neighborhood
of radius $\tau_{\delta}$ of $\Sigma_{0}^{D}$. Hence by shrinking
$\tau$ if necessary, we may assume
\begin{align*}
\left|\left\{ \tilde{\theta}_{j},x_{0}\right\} \right| & \leq2,\qquad j\geq0,\\
\left|\left\{ \left|\tilde{\theta}\right|,x_{0}\right\} \right| & \leq2,\qquad\left|\tilde{\theta}\right|\neq0,\\
\left|\left\{ \frac{\tilde{\theta}_{j}}{\left|\tilde{\theta}\right|},x_{0}\right\} \right| & \leq\frac{4}{\left|\tilde{\theta}\right|},\qquad\left|\tilde{\theta}\right|\neq0,\,j\geq0,\\
\left|\frac{1}{\varepsilon_{\delta}}\left\{ \tilde{\theta}_{j},\left|\left(x'',\xi''\right)\right|\right\} \right| & \leq\frac{1}{T},\qquad j\geq0,\\
\left|\frac{1}{\varepsilon_{\delta}}\left\{ \left|\tilde{\theta}\right|,\left|\left(x'',\xi''\right)\right|\right\} \right| & \leq\frac{1}{T},\qquad\left|\tilde{\theta}\right|\neq0,\\
\left|\frac{1}{\varepsilon_{\delta}}\left\{ \frac{\tilde{\theta}_{j}}{\left|\tilde{\theta}\right|},\left|\left(x'',\xi''\right)\right|\right\} \right| & \leq\frac{1}{T\left|\tilde{\theta}\right|},\qquad\left|\tilde{\theta}\right|\neq0,\,j\geq0,\\
\left|\left\{ \tilde{\theta}_{0},\tilde{\theta}_{j}\right\} \right| & \leq\frac{\left|\tilde{\theta}\right|}{T},\qquad j\geq0,\\
\left|\left\{ \tilde{\theta}_{0},\left|\tilde{\theta}\right|\right\} \right| & \leq\frac{\left|\tilde{\theta}\right|}{T},\qquad\left|\tilde{\theta}\right|\neq0,\ \ 
\end{align*}
\begin{align}
\left|\left\{ \tilde{\theta}_{0},\frac{\tilde{\theta}_{j}}{\left|\tilde{\theta}\right|}\right\} \right| & \leq\frac{1}{T},\qquad\left|\tilde{\theta}\right|\neq0,\,j\geq0,\nonumber \\
\frac{1}{4}\left[\xi_{0}^{2}+2\sum_{j=1}^{m}\mu_{j}\left(x_{j}^{2}+\xi_{j}^{2}\right)\right] & \leq\sum_{j=0}^{2m}\tilde{\theta}_{j}^{2}\leq4\left[\xi_{0}^{2}+2\sum_{j=1}^{m}\mu_{j}\left(x_{j}^{2}+\xi_{j}^{2}\right)\right]\label{eq: shrinking tau consequences}
\end{align}
on $U_{\varepsilon_{\delta},\tau_{\delta},T}$ and set 
\[
\tilde{U}_{\varepsilon_{\delta},\tau_{\delta},T}\coloneqq\left\{ \sum_{j=0}^{2m}\tilde{\theta}_{j}^{2}<\left(\frac{\tau_{\delta}}{8}\right)^{2},\:x''^{2}+\xi''^{2}<\left(\frac{\varepsilon_{\delta}}{8}\right)^{2},\:-\frac{T}{8}<x_{0}<\frac{T}{8}\right\} \subset U_{\varepsilon_{\delta},\tau_{\delta},T}.
\]

It is clear from the above construction that a finite set $\left\{ \kappa_{p_{u}}\left(\tilde{U}_{\frac{2\varepsilon_{\delta}}{3},\frac{2\tau_{\delta}}{3},\frac{2T}{3}}\right)\right\} _{u=1}^{N_{h}}$,
$N_{h}=O\left(h^{-\delta}\right)$, covers $\Sigma_{\left[-\frac{\tau_{\delta}}{16},\frac{\tau_{\delta}}{16}\right]}^{D}\setminus\tilde{\Omega}_{\gamma_{v}}^{\delta}$.
Next define 
\begin{eqnarray*}
U_{0} & = & \overline{T^{*}X}\setminus\Sigma_{\left[-\frac{\tau_{\delta}}{32},\frac{\tau_{\delta}}{32}\right]}^{D}\\
U_{u} & = & \kappa_{p_{u}}\left(\tilde{U}_{\frac{2\varepsilon_{\delta}}{3},\frac{2\tau_{\delta}}{3},\frac{2T}{3}}\right)\\
V_{v} & = & \Sigma_{\left[-\frac{\tau_{\delta}}{8},\frac{\tau_{\delta}}{8}\right]}^{D}\cap\tilde{\Omega}_{\gamma_{v}}^{\delta}.
\end{eqnarray*}
Choose $\mathcal{P}=\left\{ A_{u}\in\Psi_{\delta}^{0}\right\} _{0\leq u\leq N}\cup\left\{ B_{v}\in\Psi_{\delta}^{0}\left(X\right)\right\} _{1\leq v\leq M}$
to be any microlocal partition of unity subordinate to this cover.
We then augment this partition by

\begin{align*}
W_{uu'}^{1} & =\kappa_{p_{u}}\left(\tilde{U}_{\varepsilon_{\delta},\tau_{\delta},T}\right)\subset\tilde{N}_{x_{u}}\\
W_{uv}^{2} & =\kappa_{p_{u}}\left(\tilde{U}_{\varepsilon_{\delta},\tau_{\delta},T}\right)\subset\tilde{N}_{x_{u}}\\
V_{uu'}^{1} & =\kappa_{p_{u}}\left(\tilde{U}_{4\varepsilon_{\delta},4\tau_{\delta},4T}\right)\subset\tilde{N}_{x_{u}}\\
V_{uv}^{2} & =\kappa_{p_{u}}\left(\tilde{U}_{4\varepsilon_{\delta},4\tau_{\delta},4T}\right)\subset\tilde{N}_{x_{u}}
\end{align*}
where $\left(u,u'\right)\in I_{\mathcal{P}}$ and $\left(u,v\right)\in J_{\mathcal{P}}$
lie in the corresponding index sets. Clearly the above satisfy \prettyref{eq: microlocal partition of unity-1-1-1},
\prettyref{eq:augmentation}. 

It remains to verify \prettyref{eq: extension time arb large-1-1}.
To this end, let $\chi\in C_{c}^{\infty}\left(\left[-4,4\right];\left[0,1\right]\right)$,
be a cutoff such that $\chi=1$ on $\left[-2,2\right]$ and $\left|\chi'\right|\leq1$.
For $\rho\in\left(0,\frac{1}{8}\right)$ fixed, define a function
$\varphi_{\rho}\in C^{\infty}\left(\left[-1,1\right]_{\theta_{0}};\left[0,1\right]\right)$
such that $\varphi_{\rho}\left(\theta_{0}\right)=\begin{cases}
1; & \textrm{for }\theta_{0}\in\left[1-\rho,1\right]\\
0; & \textrm{for }\theta_{0}\in\left[-1,1-2\rho\right]
\end{cases}$ and $\left|\varphi_{\rho}'\right|\leq\frac{2}{\rho}$ . Set 
\begin{align*}
\beta\left(\tilde{\theta}\right) & \coloneqq\sqrt{\left|\tilde{\theta}\right|^{2}-\varphi_{\rho}\left(\theta_{0}\right)\left|\tilde{\theta}'\right|^{2}}\\
 & =\sqrt{\left|\tilde{\theta}_{0}\right|^{2}+\left(1-\varphi_{\rho}\right)\left|\tilde{\theta}'\right|^{2}}.
\end{align*}
For \textbf{$\theta_{0}\in\left[-1,1-2\rho\right]$}, $\varphi_{\rho}=0$
and $\beta\left(\tilde{\theta}\right)=\left|\tilde{\theta}\right|$.
While for \textbf{$\theta_{0}\in\left[1-2\rho,1\right]$,} we have
$\left|\tilde{\theta}\right|\geq\sqrt{\left|\tilde{\theta}\right|^{2}-\varphi_{\rho}\left(\theta_{0}\right)\left|\tilde{\theta}'\right|^{2}}=\beta\left(\tilde{\theta}\right)=\sqrt{\left|\tilde{\theta}_{0}\right|^{2}+\left(1-\varphi_{\rho}\right)\left|\tilde{\theta}'\right|^{2}}\geq\left|\tilde{\theta}_{0}\right|=\theta_{0}\left|\tilde{\theta}\right|\geq\frac{1}{2}\left|\tilde{\theta}\right|$
for $\rho\in\left(0,\frac{1}{8}\right)$ as chosen. Thus $\left|\tilde{\theta}\right|\geq\beta\left(\tilde{\theta}\right)\geq\frac{1}{2}\left|\tilde{\theta}\right|$in
both cases and we may for each $1\leq u\leq N$, define the microlocal
weight function 
\begin{align*}
g_{u}\coloneqq & \left(\kappa_{p_{u}}^{-1}\right)^{*}\chi\left(\frac{16\beta\left(\tilde{\theta}\right)}{\tau_{\delta}}\right)\chi\left(\frac{16x_{0}}{T}\right)\chi\left(\frac{16\left|\left(x'',\xi''\right)\right|}{\varepsilon_{\delta}}\right)\in C_{c}^{\infty}\left(\kappa_{p_{u}}\left(\tilde{U}_{\varepsilon_{\delta},\tau_{\delta},T}\right)\right)
\end{align*}
 in terms of the relevant coordinates on $\tilde{U}_{\varepsilon_{\delta},\tau_{\delta},T}$. 

Next, with $\mathtt{v}_{t}^{\rho}\in C^{\infty}\left(S^{n-1};U\left(\mathbb{C}^{2^{m}}\right)\right)$
as in \prettyref{lem: almost diagonalization lemma}, we choose for
each $1\leq u\leq N$ a symbol $\tilde{\mathtt{v}}_{u}\in S_{\delta}^{0}\left(X;U\left(S\right)\right)$
satisfying
\[
\tilde{\mathtt{v}}_{u}\coloneqq\begin{cases}
\mathtt{v}_{8\left|\tilde{\theta}\right|/\tau_{\delta}}^{\rho}\left(\theta\right); & \left|\tilde{\theta}\right|<\frac{\tau_{\delta}}{8}\\
\mathtt{v}_{1}^{\rho}\left(\theta\right); & \left|\tilde{\theta}\right|\geq\frac{\tau_{\delta}}{8},
\end{cases}
\]
on $\kappa_{p_{u}}\left(\tilde{U}_{\varepsilon_{\delta},\tau_{\delta},T}\right)$,
with $\tilde{\theta},\theta$ given by \prettyref{eq: components symbol},
\prettyref{eq:circle valued function}. Since the conjugate of the
symbol $d$ of the Dirac operator is $e^{ic_{A}}de^{-ic_{A}}=\sigma_{j}\tilde{\theta}_{j}=i\left|\tilde{\theta}\right|c\left(\theta\right)$
by \prettyref{eq: conjugate symbol} on $\kappa_{p_{u}}\left(\tilde{U}_{\varepsilon,\tau,T}\right)$,
we may compute from \prettyref{lem: almost diagonalization lemma}
\begin{equation}
\left(\tilde{\mathtt{v}}_{u}\right)^{*}e^{ic_{A}}de^{-ic_{A}}\tilde{\mathtt{v}}_{u}=\begin{cases}
\tilde{\theta}_{0}\sigma_{0}-\left[\sum_{j=1}^{2m}\tilde{\theta}_{j}\sigma_{j}\right]; & \left|\tilde{\theta}\right|\leq\frac{\tau}{16}\\
\left|\tilde{\theta}\right|\mathtt{v}_{1}^{\rho}\left(\theta\right)^{*}c\left(\theta\right)\mathtt{v}_{1}^{\rho}\left(\theta\right); & \left|\tilde{\theta}\right|\geq\frac{\tau}{8}
\end{cases}\label{eq: conjugate symbol interpolation}
\end{equation}
on $\kappa_{p_{u}}\left(\tilde{U}_{\varepsilon,\tau,T}\right)$. Furthermore;
\prettyref{lem: almost diagonalization lemma} also gives 
\[
\left|\tilde{\theta}\right|\mathtt{v}_{1}^{\rho}\left(\theta\right)^{*}c\left(\theta\right)\mathtt{v}_{1}^{\rho}\left(\theta\right)=\begin{cases}
\left|\tilde{\theta}\right|\sigma_{0}; & \theta_{0}\leq1-\rho,\\
\left|\tilde{\theta}\right|\left[a_{0,1}^{\rho}\left(\theta_{0}\right)\sigma_{0}+a_{1,1}^{\rho}\left(\theta_{0}\right)\sum_{j=1}^{2m}\theta_{j}\sigma_{j}\right]; & \theta_{0}>1-\rho.
\end{cases}
\]
Choose $\mathtt{v}_{u}\in S_{\delta}^{0}\left(X;U\left(S\right)\right)$
to be a symbol satisfying
\begin{equation}
\mathtt{v}_{u}=e^{-ic_{A}}\tilde{\mathtt{v}}_{u}\label{eq: actual conjugate symbol}
\end{equation}
 on $\kappa_{p_{u}}\left(\tilde{U}_{\varepsilon,\tau,T}\right)$. 

We now compute for $\left|\tilde{\theta}\right|>\frac{\tau_{\delta}}{8}$,
$\theta_{0}\leq1-2\rho$;
\begin{align}
\left|H_{g_{u},\mathtt{v}_{u}}\left(d\right)\right| & =\left|\left\{ \mathtt{v}_{u}^{*}d\mathtt{v}_{u},g_{u}\right\} \right|\nonumber \\
 & =\left|\left\{ \left|\tilde{\theta}\right|,g_{u}\right\} \sigma_{0}\right|\nonumber \\
 & =\left|\frac{16}{T}\left\{ \left|\tilde{\theta}\right|,x_{0}\right\} \frac{\chi'\left(\frac{16x_{0}}{T}\right)}{\chi\left(\frac{16x_{0}}{T}\right)}g_{u}\right.\nonumber \\
 & +\left.\frac{16}{\varepsilon_{\delta}}\left\{ \left|\tilde{\theta}\right|,\left|\left(x'',\xi''\right)\right|\right\} \frac{\chi'\left(\frac{16\left|\left(x'',\xi''\right)\right|}{\varepsilon_{\delta}}\right)}{\chi\left(\frac{16\left|\left(x'',\xi''\right)\right|}{\varepsilon_{\delta}}\right)}g_{u}\right|\nonumber \\
 & \leq\frac{64}{T}\label{eq: comm comp 1}
\end{align}
using \prettyref{eq: shrinking tau consequences}.

While for $\left|\tilde{\theta}\right|>\frac{\tau_{\delta}}{8}$,
$1-2\rho\leq\theta_{0}\leq1-\rho$;
\begin{align}
\left|H_{g_{u},\mathtt{v}_{u}}\left(d\right)\right| & =\left|\left\{ \mathtt{v}_{u}^{*}d\mathtt{v}_{u},g_{u}\right\} \right|\nonumber \\
 & =\left|\left\{ \left|\tilde{\theta}\right|,g_{u}\right\} \sigma_{0}\right|\nonumber \\
 & =\left|\frac{16}{T}\left\{ \left|\tilde{\theta}\right|,x_{0}\right\} \frac{\chi'\left(\frac{16x_{0}}{T}\right)}{\chi\left(\frac{16x_{0}}{T}\right)}g_{u}+\right.\nonumber \\
 & +\frac{16}{\varepsilon_{\delta}}\left\{ \left|\tilde{\theta}\right|,\left|\left(x'',\xi''\right)\right|\right\} \frac{\chi'\left(\frac{16\left|\left(x'',\xi''\right)\right|}{\varepsilon_{\delta}}\right)}{\chi\left(\frac{16\left|\left(x'',\xi''\right)\right|}{\varepsilon_{\delta}}\right)}g_{u}\nonumber \\
 & +\frac{8\varphi_{\rho}}{\beta\tau_{\delta}}\left\{ \left|\tilde{\theta}\right|,\left|\tilde{\theta}'\right|^{2}\right\} \frac{\chi'\left(\frac{16\beta\left(\tilde{\theta}\right)}{\tau_{\delta}}\right)}{\chi\left(\frac{16\beta\left(\tilde{\theta}\right)}{\tau_{\delta}}\right)}g_{u}\nonumber \\
 & \left.+\frac{8\varphi_{\rho}'}{\beta\tau_{\delta}}\left\{ \left|\tilde{\theta}\right|,\frac{\tilde{\theta}_{0}}{\left|\tilde{\theta}\right|}\right\} \left|\tilde{\theta}'\right|^{2}\frac{\chi'\left(\frac{16\beta\left(\tilde{\theta}\right)}{\tau_{\delta}}\right)}{\chi\left(\frac{16\beta\left(\tilde{\theta}\right)}{\tau_{\delta}}\right)}g_{u}\right|\nonumber \\
 & \leq\frac{8^{4}}{\rho T}\label{eq: comm comp 2}
\end{align}
using \prettyref{eq: shrinking tau consequences}.

Now for $\left|\tilde{\theta}\right|>\frac{\tau_{\delta}}{8}$, $\theta_{0}>1-\rho$;
we compute
\begin{align*}
\left|H_{g_{u},\mathtt{v}_{u}}\left(d\right)\right| & =\left|\left\{ \mathtt{v}_{u}^{*}d\mathtt{v}_{u},g_{u}\right\} \right|\\
 & =\left|\left\{ \left|\tilde{\theta}\right|\left[a_{0,1}^{\rho}\left(\theta_{0}\right)\sigma_{0}+a_{1,1}^{\rho}\left(\theta_{0}\right)\sum_{j=1}^{2m}\theta_{j}\sigma_{j}\right],g_{u}\right\} \right|\\
 & \leq\left|\left\{ \left|\tilde{\theta}\right|,g_{u}\right\} \right|+\left|\left|\tilde{\theta}\right|\left\{ a_{0,1}^{\rho}\left(\theta_{0}\right)\sigma_{0}+a_{1,1}^{\rho}\left(\theta_{0}\right)\sum_{j=1}^{2m}\theta_{j}\sigma_{j},g_{u}\right\} \right|
\end{align*}
with
\begin{align}
\left|\left\{ \left|\tilde{\theta}\right|,g_{u}\right\} \right| & =\left|\frac{16}{T}\left\{ \left|\tilde{\theta}\right|,x_{0}\right\} \frac{\chi'\left(\frac{16x_{0}}{T}\right)}{\chi\left(\frac{16x_{0}}{T}\right)}g_{u}\right.\nonumber \\
 & +\frac{16}{\varepsilon_{\delta}}\left\{ \left|\tilde{\theta}\right|,\left|\left(x'',\xi''\right)\right|\right\} \frac{\chi'\left(\frac{16\left|\left(x'',\xi''\right)\right|}{\varepsilon_{\delta}}\right)}{\chi\left(\frac{16\left|\left(x'',\xi''\right)\right|}{\varepsilon_{\delta}}\right)}g_{u}\nonumber \\
 & \left.+\frac{8}{\beta\tau_{\delta}}\left\{ \left|\tilde{\theta}\right|,\tilde{\theta}_{0}^{2}\right\} \frac{\chi'\left(\frac{16\beta\left(\tilde{\theta}\right)}{\tau_{\delta}}\right)}{\chi\left(\frac{16\beta\left(\tilde{\theta}\right)}{\tau_{\delta}}\right)}g_{u}\right|\nonumber \\
 & \leq\frac{8^{4}}{T}\quad\label{eq:comm comp 3}
\end{align}
and
\begin{align*}
 & \left|\left|\tilde{\theta}\right|\left\{ a_{0,1}^{\rho}\left(\theta_{0}\right)\sigma_{0}+a_{1,1}^{\rho}\left(\theta_{0}\right)\sum_{j=1}^{2m}\theta_{j}\sigma_{j},g_{u}\right\} \right|\\
 & =\left|\frac{16\left|\tilde{\theta}\right|}{T}\left(a_{0,1}^{\rho}\right)'\sigma_{0}\left\{ \frac{\tilde{\theta}_{0}}{\left|\tilde{\theta}\right|},x_{0}\right\} \frac{\chi'\left(\frac{16x_{0}}{T}\right)}{\chi\left(\frac{16x_{0}}{T}\right)}g_{u}\right.\\
 & +\frac{16\left|\tilde{\theta}\right|}{T}\left(a_{1,1}^{\rho}\right)'\left(\sum_{j=1}^{2m}\theta_{j}\sigma_{j}\right)\left\{ \frac{\tilde{\theta}_{0}}{\left|\tilde{\theta}\right|},x_{0}\right\} \frac{\chi'\left(\frac{16x_{0}}{T}\right)}{\chi\left(\frac{16x_{0}}{T}\right)}g_{u}\\
 & +\frac{16\left|\tilde{\theta}\right|}{T}a_{1,1}^{\rho}\left(\sum_{j=1}^{2m}\left\{ \frac{\tilde{\theta}_{j}}{\left|\tilde{\theta}\right|},x_{0}\right\} \sigma_{j}\right)\frac{\chi'\left(\frac{16x_{0}}{T}\right)}{\chi\left(\frac{16x_{0}}{T}\right)}g_{u}\\
 & +16\left|\tilde{\theta}\right|\left(a_{0,1}^{\rho}\right)'\sigma_{0}\frac{1}{\varepsilon_{\delta}}\left\{ \frac{\tilde{\theta}_{0}}{\left|\tilde{\theta}\right|},\left|\left(x'',\xi''\right)\right|\right\} \frac{\chi'\left(\frac{16\left|\left(x'',\xi''\right)\right|}{\varepsilon_{\delta}}\right)}{\chi\left(\frac{16\left|\left(x'',\xi''\right)\right|}{\varepsilon_{\delta}}\right)}g_{u}\\
 & +16\left|\tilde{\theta}\right|\left(a_{1,1}^{\rho}\right)'\left(\sum_{j=1}^{2m}\theta_{j}\sigma_{j}\right)\frac{1}{\varepsilon_{\delta}}\left\{ \frac{\tilde{\theta}_{0}}{\left|\tilde{\theta}\right|},\left|\left(x'',\xi''\right)\right|\right\} \frac{\chi'\left(\frac{16\left|\left(x'',\xi''\right)\right|}{\varepsilon_{\delta}}\right)}{\chi\left(\frac{16\left|\left(x'',\xi''\right)\right|}{\varepsilon_{\delta}}\right)}g_{u}\\
 & +16\left|\tilde{\theta}\right|a_{1,1}^{\rho}\left(\sum_{j=1}^{2m}\frac{1}{\varepsilon_{\delta}}\left\{ \frac{\tilde{\theta}_{j}}{\left|\tilde{\theta}\right|},\left|\left(x'',\xi''\right)\right|\right\} \sigma_{j}\right)\frac{\chi'\left(\frac{16\left|\left(x'',\xi''\right)\right|}{\varepsilon_{\delta}}\right)}{\chi\left(\frac{16\left|\left(x'',\xi''\right)\right|}{\varepsilon_{\delta}}\right)}g_{u}\\
 & +\left|\tilde{\theta}\right|\frac{8}{\beta\tau_{\delta}}\left(a_{0,1}^{\rho}\right)'\sigma_{0}\left\{ \frac{\tilde{\theta}_{0}}{\left|\tilde{\theta}\right|},\tilde{\theta}_{0}^{2}\right\} \frac{\chi'\left(\frac{16\beta\left(\tilde{\theta}\right)}{\tau_{\delta}}\right)}{\chi\left(\frac{16\beta\left(\tilde{\theta}\right)}{\tau_{\delta}}\right)}g_{u}
\end{align*}
\begin{align}
 & +\left|\tilde{\theta}\right|\frac{8}{\beta\tau_{\delta}}\left(a_{0,1}^{\rho}\right)'\left(\sum_{j=1}^{2m}\theta_{j}\sigma_{j}\right)\left\{ \frac{\tilde{\theta}_{0}}{\left|\tilde{\theta}\right|},\tilde{\theta}_{0}^{2}\right\} \frac{\chi'\left(\frac{16\beta\left(\tilde{\theta}\right)}{\tau_{\delta}}\right)}{\chi\left(\frac{16\beta\left(\tilde{\theta}\right)}{\tau_{\delta}}\right)}g_{u}\nonumber \\
 & \left.+\left|\tilde{\theta}\right|\frac{8}{\beta\tau_{\delta}}a_{1,1}^{\rho}\left(\sum_{j=1}^{2m}\left\{ \frac{\tilde{\theta}_{j}}{\left|\tilde{\theta}\right|},\tilde{\theta}_{0}^{2}\right\} \sigma_{j}\right)\frac{\chi'\left(\frac{16\beta\left(\tilde{\theta}\right)}{\tau_{\delta}}\right)}{\chi\left(\frac{16\beta\left(\tilde{\theta}\right)}{\tau_{\delta}}\right)}g_{u}\right|\nonumber \\
 & \leq\left(\frac{8}{\rho}\right)^{2}\frac{8^{4}}{T}.\label{eq: comm comp 4}
\end{align}
using \prettyref{lem: almost diagonalization lemma} and \prettyref{eq: shrinking tau consequences}. 

Now for $\frac{\tau_{\delta}}{16}\leq\left|\tilde{\theta}\right|\leq\frac{\tau_{\delta}}{8}$,
$\chi\left(\frac{16\beta\left(\tilde{\theta}\right)}{\tau_{\delta}}\right)=1$
and we may compute 
\begin{align*}
\left|H_{g_{u},\mathtt{v}_{u}}\left(d\right)\right| & =\left|\left\{ \mathtt{v}_{u}^{*}d\mathtt{v}_{u},g_{u}\right\} \right|\\
 & =\left|\left\{ \mathtt{v}_{8\left|\tilde{\theta}\right|/\tau_{\delta}}^{\rho}{}^{*}\left[\sigma_{j}\tilde{\theta}_{j}\right]\mathtt{v}_{8\left|\tilde{\theta}\right|/\tau_{\delta}}^{\rho},g_{u}\right\} \right|\\
 & =\left|\frac{16}{T}\mathtt{v}_{8\left|\tilde{\theta}\right|/\tau_{\delta}}^{\rho}{}^{*}\left[\sigma_{j}\left\{ \tilde{\theta}_{j},x_{0}\right\} \right]\mathtt{v}_{8\left|\tilde{\theta}\right|/\tau_{\delta}}^{\rho}\frac{\chi'\left(\frac{16x_{0}}{T}\right)}{\chi\left(\frac{16x_{0}}{T}\right)}g_{u}\right.\\
 & +\frac{128}{\tau_{\delta}T}\left[\left.\partial_{t}\mathtt{v}_{t}^{\rho}\right|_{t=8\left|\tilde{\theta}\right|/\tau_{\delta}}\right]^{*}\left[\sigma_{j}\tilde{\theta}_{j}\right]\left[\mathtt{v}_{8\left|\tilde{\theta}\right|/\tau_{\delta}}^{\rho}\right]\left\{ \left|\tilde{\theta}\right|,x_{0}\right\} \frac{\chi'\left(\frac{16x_{0}}{T}\right)}{\chi\left(\frac{16x_{0}}{T}\right)}g_{u}\\
 & +\frac{128}{\tau_{\delta}T}\left[\mathtt{v}_{8\left|\tilde{\theta}\right|/\tau_{\delta}}^{\rho}\right]^{*}\left[\sigma_{j}\tilde{\theta}_{j}\right]\left[\left.\partial_{t}\mathtt{v}_{t}^{\rho}\right|_{t=8\left|\tilde{\theta}\right|/\tau_{\delta}}\right]\left\{ \left|\tilde{\theta}\right|,x_{0}\right\} \frac{\chi'\left(\frac{16x_{0}}{T}\right)}{\chi\left(\frac{16x_{0}}{T}\right)}g_{u}\\
 & +\frac{16}{T}\left[\left.\partial_{\theta_{k}}\mathtt{v}_{t}^{\rho}\right|_{t=8\left|\tilde{\theta}\right|/\tau_{\delta}}\right]^{*}\left[\sigma_{j}\tilde{\theta}_{j}\right]\left[\mathtt{v}_{8\left|\tilde{\theta}\right|/\tau_{\delta}}^{\rho}\right]\left\{ \frac{\tilde{\theta}_{k}}{\left|\tilde{\theta}\right|},x_{0}\right\} \frac{\chi'\left(\frac{16x_{0}}{T}\right)}{\chi\left(\frac{16x_{0}}{T}\right)}g_{u}\\
 & +\frac{16}{T}\left[\mathtt{v}_{8\left|\tilde{\theta}\right|/\tau_{\delta}}^{\rho}\right]^{*}\left[\sigma_{j}\tilde{\theta}_{j}\right]\left[\left.\partial_{\theta_{k}}\mathtt{v}_{t}^{\rho}\right|_{t=8\left|\tilde{\theta}\right|/\tau_{\delta}}\right]\left\{ \frac{\tilde{\theta}_{k}}{\left|\tilde{\theta}\right|},x_{0}\right\} \frac{\chi'\left(\frac{16x_{0}}{T}\right)}{\chi\left(\frac{16x_{0}}{T}\right)}g_{u}\\
 & +\frac{16}{\varepsilon_{\delta}}\mathtt{v}_{8\left|\tilde{\theta}\right|/\tau_{\delta}}^{\rho}{}^{*}\left[\sigma_{j}\left\{ \tilde{\theta}_{j},\left|\left(x'',\xi''\right)\right|\right\} \right]\mathtt{v}_{8\left|\tilde{\theta}\right|/\tau_{\delta}}^{\rho}\frac{\chi'\left(\frac{16\left|\left(x'',\xi''\right)\right|}{\varepsilon_{\delta}}\right)}{\chi\left(\frac{16\left|\left(x'',\xi''\right)\right|}{\varepsilon_{\delta}}\right)}g_{u}
\end{align*}
\begin{align}
 & +\frac{128}{\tau_{\delta}\varepsilon_{\delta}}\left[\left.\partial_{t}\mathtt{v}_{t}^{\rho}\right|_{t=8\left|\tilde{\theta}\right|/\tau_{\delta}}\right]^{*}\left[\sigma_{j}\tilde{\theta}_{j}\right]\left[\mathtt{v}_{8\left|\tilde{\theta}\right|/\tau_{\delta}}^{\rho}\right]\left\{ \left|\tilde{\theta}\right|,\left|\left(x'',\xi''\right)\right|\right\} \frac{\chi'\left(\frac{16\left|\left(x'',\xi''\right)\right|}{\varepsilon_{\delta}}\right)}{\chi\left(\frac{16\left|\left(x'',\xi''\right)\right|}{\varepsilon_{\delta}}\right)}g_{u}\nonumber \\
 & +\frac{128}{\tau_{\delta}\varepsilon_{\delta}}\left[\mathtt{v}_{8\left|\tilde{\theta}\right|/\tau_{\delta}}^{\rho}\right]^{*}\left[\sigma_{j}\tilde{\theta}_{j}\right]\left[\left.\partial_{t}\mathtt{v}_{t}^{\rho}\right|_{t=8\left|\tilde{\theta}\right|/\tau_{\delta}}\right]\left\{ \left|\tilde{\theta}\right|,\left|\left(x'',\xi''\right)\right|\right\} \frac{\chi'\left(\frac{16\left|\left(x'',\xi''\right)\right|}{\varepsilon_{\delta}}\right)}{\chi\left(\frac{16\left|\left(x'',\xi''\right)\right|}{\varepsilon_{\delta}}\right)}g_{u}\nonumber \\
 & +\frac{16}{\varepsilon_{\delta}}\left[\left.\partial_{\theta_{k}}\mathtt{v}_{t}^{\rho}\right|_{t=8\left|\tilde{\theta}\right|/\tau_{\delta}}\right]^{*}\left[\sigma_{j}\tilde{\theta}_{j}\right]\left[\mathtt{v}_{8\left|\tilde{\theta}\right|/\tau_{\delta}}^{\rho}\right]\left\{ \frac{\tilde{\theta}_{k}}{\left|\tilde{\theta}\right|},\left|\left(x'',\xi''\right)\right|\right\} \frac{\chi'\left(\frac{16\left|\left(x'',\xi''\right)\right|}{\varepsilon_{\delta}}\right)}{\chi\left(\frac{16\left|\left(x'',\xi''\right)\right|}{\varepsilon_{\delta}}\right)}g_{u}\nonumber \\
 & \left.+\frac{16}{\varepsilon_{\delta}}\left[\mathtt{v}_{8\left|\tilde{\theta}\right|/\tau_{\delta}}^{\rho}\right]^{*}\left[\sigma_{j}\tilde{\theta}_{j}\right]\left[\left.\partial_{\theta_{k}}\mathtt{v}_{t}^{\rho}\right|_{t=8\left|\tilde{\theta}\right|/\tau_{\delta}}\right]\left\{ \frac{\tilde{\theta}_{k}}{\left|\tilde{\theta}\right|},\left|\left(x'',\xi''\right)\right|\right\} \frac{\chi'\left(\frac{16\left|\left(x'',\xi''\right)\right|}{\varepsilon_{\delta}}\right)}{\chi\left(\frac{16\left|\left(x'',\xi''\right)\right|}{\varepsilon_{\delta}}\right)}g_{u}\right|\nonumber \\
\leq & \left(\frac{8}{\rho}\right)^{4}\frac{8^{4}}{T}\label{eq: comm comp 5}
\end{align}
using \prettyref{lem: almost diagonalization lemma} and \prettyref{eq: shrinking tau consequences}.

Finally for $\left|\tilde{\theta}\right|\leq\frac{\tau_{\delta}}{16}$
again $\chi\left(\frac{16\beta\left(\tilde{\theta}\right)}{\tau_{\delta}}\right)=1$
and we may use \prettyref{eq: conjugate symbol interpolation} to
compute 
\begin{align}
\left|H_{g_{u},\mathtt{v}_{u}}\left(d\right)\right| & =\left|\left\{ \mathtt{v}_{u}^{*}d\mathtt{v}_{u},g_{u}\right\} \right|\nonumber \\
 & =\left|\left\{ \tilde{\theta}_{0}\sigma_{0}-\left[\sum_{j=1}^{2m}\tilde{\theta}_{j}\sigma_{j}\right],g_{u}\right\} \right|\nonumber \\
 & =\left|\frac{16}{T}\left\{ \tilde{\theta}_{0},x_{0}\right\} \sigma_{0}\frac{\chi'\left(\frac{16x_{0}}{T}\right)}{\chi\left(\frac{16x_{0}}{T}\right)}g_{u}\right.\nonumber \\
 & -\frac{16}{T}\left(\sum_{j=1}^{2m}\left\{ \tilde{\theta}_{j},x_{0}\right\} \sigma_{j}\right)\frac{\chi'\left(\frac{16x_{0}}{T}\right)}{\chi\left(\frac{16x_{0}}{T}\right)}g_{u}\nonumber \\
 & +\frac{16}{\varepsilon_{\delta}}\left\{ \tilde{\theta}_{0},\left|\left(x'',\xi''\right)\right|\right\} \sigma_{0}\frac{\chi'\left(\frac{16\left|\left(x'',\xi''\right)\right|}{\varepsilon_{\delta}}\right)}{\chi\left(\frac{16\left|\left(x'',\xi''\right)\right|}{\varepsilon_{\delta}}\right)}g_{u}\nonumber \\
 & \left.-\frac{16}{\varepsilon_{\delta}}\left(\sum_{j=1}^{2m}\left\{ \tilde{\theta}_{j},\left|\left(x'',\xi''\right)\right|\right\} \sigma_{j}\right)\frac{\chi'\left(\frac{16\left|\left(x'',\xi''\right)\right|}{\varepsilon_{\delta}}\right)}{\chi\left(\frac{16\left|\left(x'',\xi''\right)\right|}{\varepsilon_{\delta}}\right)}g_{u}\right|\nonumber \\
 & \leq\frac{8^{2}}{T}\label{eq: comm comp 6}
\end{align}
using \prettyref{eq: shrinking tau consequences}. Since $\rho\in\left(0,\frac{1}{8}\right)$
is fixed and $T$ arbitrary, the proposition follows from \prettyref{eq: comm comp 1}-\prettyref{eq: comm comp 6}.
\end{proof}
Next, given an augmented $\left(\Omega,\tau,\delta\right)$-partition
of unity $\left(\mathcal{P};\mathcal{V},\mathcal{W}\right)$ the trace
\prettyref{eq: non-local dynamical trace} from the Helffer-Sjöstrand
formula is clearly the sum of traces of the following four kinds
\begin{align}
\mathcal{T}_{A_{u},A_{v}}^{\theta}\left(D\right)\coloneqq & \frac{1}{\pi}\int_{\mathbb{C}}\bar{\partial}\tilde{f}\left(z\right)\check{\theta}\left(\frac{\lambda-z}{\sqrt{h}}\right)\textrm{tr }\left[A_{u}\left(\frac{1}{\sqrt{h}}D-z\right)^{-1}A_{v}\right]dzd\bar{z}\nonumber \\
\mathcal{T}_{A_{u},B_{v}}^{\theta}\left(D\right)\coloneqq & \frac{1}{\pi}\int_{\mathbb{C}}\bar{\partial}\tilde{f}\left(z\right)\check{\theta}\left(\frac{\lambda-z}{\sqrt{h}}\right)\textrm{tr }\left[A_{u}\left(\frac{1}{\sqrt{h}}D-z\right)^{-1}B_{v}\right]dzd\bar{z}\nonumber \\
\mathcal{T}_{B_{v},A_{u}}^{\theta}\left(D\right)\coloneqq & \frac{1}{\pi}\int_{\mathbb{C}}\bar{\partial}\tilde{f}\left(z\right)\check{\theta}\left(\frac{\lambda-z}{\sqrt{h}}\right)\textrm{tr }\left[B_{v}\left(\frac{1}{\sqrt{h}}D-z\right)^{-1}A_{u}\right]dzd\bar{z}\nonumber \\
\mathcal{T}_{B_{u},B_{v}}^{\theta}\left(D\right)\coloneqq & \frac{1}{\pi}\int_{\mathbb{C}}\bar{\partial}\tilde{f}\left(z\right)\check{\theta}\left(\frac{\lambda-z}{\sqrt{h}}\right)\textrm{tr }\left[B_{u}\left(\frac{1}{\sqrt{h}}D-z\right)^{-1}A_{v}\right]dzd\bar{z}.\label{eq: four traces}
\end{align}
Next we state a modification of \cite{Savale2017-Koszul} Lemma 3.3.
Below $V_{uu'}^{1},\,W_{uu'}^{1},\,T_{uu'}$ are as in \prettyref{eq:augmentation},
\prettyref{eq: exit time 1}.
\begin{lem}
\label{lem:changing symbol near crit set} Let $D'\in\Psi_{\textrm{cl}}^{1}\left(X;S\right)$
be essentially self-adjoint such that $D=D'$ microlocally on $V_{uu'}^{1}$
. Then for $\theta\in C_{c}^{\infty}\left(\left(T_{0},T_{uu'}\right);\left[0,1\right]\right)$
one has 
\[
\mathcal{T}_{A_{u},A_{v}}^{\theta}\left(D\right)=\mathcal{T}_{A_{u},A_{v}}^{\theta}\left(D'\right)\quad\textrm{mod }h^{\infty}.
\]
\end{lem}
\begin{proof}
The lemma is essentially the same as \cite{Savale2017-Koszul} Lemma
3.3 with a couple of changes. First our cutoffs lie in the more exotic
class $\Psi_{\delta}^{0}\left(X\right)$. However these have the same
basic composition and wavefront properties needed in the proof of
\cite{Savale2017-Koszul}. Next our definition of trapping time \prettyref{eq: exit time 1}
is more general than that in \cite{Savale2017-Koszul} eq. 3.5 since
an additional conjugation by a unitary symbol $\mathtt{v}\in S_{\delta}^{0}\left(X;U\left(E\right)\right)$
is allowed in the definition \prettyref{eq: exit time 1} here. This
is however easily overcome; let $\theta\in C_{c}^{\infty}\left(\left(T_{0}',T_{uu'}^{0}\right);\left[0,1\right]\right)$
be such that $T_{0}<T_{0}',\:T_{uu'}^{0}<T_{uu'}$. There hence exists
$\left(g,\mathtt{v}\right)\in\mathcal{G}_{uu'}\times S^{0}\left(X;U\left(S\right)\right)$
with $\left|H_{g,\mathtt{v}}d\right|<\frac{1}{S_{uu'}}$. We choose
$\mathtt{V}\in\Psi_{\delta}^{0}\left(X;S\right)$ unitary with $\sigma\left(\mathtt{V}\right)=\left[\mathtt{v}\right]$
and note $H_{g,\mathtt{v}}d=H_{g}\left(\mathtt{V}^{*}d\mathtt{V}\right)$
in terms a quantization defined using the chosen coordinates/trivialization
on $N_{x_{u}}$. Now, the proof of \cite{Savale2017-Koszul} Lemma
3.3 carries through with the conjugates $\mathtt{V}^{*}D\mathtt{V}$,
$\mathtt{V}^{*}D'\mathtt{V}$, $\mathtt{V}^{*}A_{u}\mathtt{V}$ and
$\mathtt{V}^{*}A_{v}\mathtt{V}$.
\end{proof}
We also note that similar lemmas as above hold for the traces $\mathcal{T}_{A_{u},B_{v}}^{\theta}\left(D\right)$
and $\mathcal{T}_{B_{v},A_{u}}^{\theta}\left(D\right)$ in \prettyref{eq: four traces}.
Next we show that the first three traces in \prettyref{eq: four traces}
are $O\left(h^{\infty}\right)$ when $\textrm{spt}\left(\theta\right)$
is contained within the extension time.
\begin{lem}
\label{lem: three traces h infty}Let $\left(\mathcal{P};\mathcal{V},\mathcal{W}\right)$
be an augmented $\left(\Omega,\tau,\delta\right)$-partition of unity.
Then for each $\theta\in C_{c}^{\infty}\left(\left(T_{0},T\right);\left[-1,1\right]\right)$
with $T<T_{\left(\mathcal{P};\mathcal{V},\mathcal{W}\right)}$ one
has 
\[
\mathcal{T}_{A_{u},A_{v}}^{\theta}\left(D\right),\:\mathcal{T}_{A_{u},B_{v}}^{\theta}\left(D\right),\textrm{ }\mathcal{T}_{B_{v},A_{u}}^{\theta}\left(D\right)=O\left(h^{\infty}\right).
\]
\end{lem}
\begin{proof}
The proof is the same as \cite{Savale2017-Koszul} Lemma 3.1 (cf.
eq. 3.2). One only has to quantify the smallness of $\textrm{spt}\left(\theta\right)$
assumed therein. The proof in \cite{Savale2017-Koszul} carries through
in so far as $\textrm{spt}\left(\theta\right)$ is contained in each
of $\left\{ \left(T_{0},T_{uu'}\right)\right\} _{\left(u,u'\right)\in I_{\mathcal{P}}}$
, $\left\{ \left(T_{0},S_{uv}\right)\right\} _{\left(u,v\right)\in J_{\mathcal{P}}}$
as required by \prettyref{lem:changing symbol near crit set}. This
is guaranteed for $T<T_{\left(\mathcal{P};\mathcal{V},\mathcal{W}\right)}$
by \prettyref{eq: total extension time}.
\end{proof}
Given $\theta,\Omega$ there exists by \prettyref{prop: Large Extension time}
an $\left(\Omega,\tau,\delta\right)$-partition of unity with an extension
time large enough to guarantee the hypothesis of \prettyref{lem: three traces h infty}.
Splitting the trace in such fashion, it then suffices to consider
the asymptotics of the fourth trace $\mathcal{T}_{B_{u},B_{v}}^{\theta}\left(D\right)$
in \prettyref{eq: four traces}. Since $B_{u}$ and $B_{v}$ have
disjoint micro-supports for $u\neq v$; it suffices to consider $\mathcal{T}_{B_{v},B_{v}}^{\theta}\left(D\right)$.
Since these are localized near the Reeb orbits, they shall first require
an understanding of the Birkhoff normal form for $D$ near each orbit
done in the next section. We shall return to $\mathcal{T}_{B_{v},B_{v}}^{\theta}\left(D\right)$
in \prettyref{sec:Reduction to S1 times R2m}.

\section{\label{sec: Birkhoff normal form near Reeb orbit}Birkhoff normal
form near a Reeb orbit}

In this section we derive a Birkhoff normal form for the Dirac operator
in a neighborhood of each Reeb orbit. First, consider a Darboux-Reeb
chart near $\gamma$ and choose an orthonormal frame $\left\{ e_{j}=w_{j}^{k}\partial_{x_{k}}\right\} ,0\leq j\leq2m$
for the tangent bundle on $\varOmega_{\gamma}$. Here we use the convention
that $x_{0}=\theta$ is the circular variable on $\varOmega_{\gamma}^{0}\subset S^{1}\times\mathbb{R}^{2m}$
and shall use these interchangeably. We hence have 
\begin{equation}
w_{j}^{k}g_{kl}w_{r}^{l}=\delta_{jr},\label{eq: diagonalizing metric}
\end{equation}
 where $g_{kl}$ is the metric in these coordinates and the Einstein
summation convention is being used. Let $\Gamma_{jk}^{l}$ be the
Christoffel symbols for the Levi-Civita connection in the orthonormal
frame $e_{i}$ satisfying $\nabla_{e_{j}}e_{k}=\Gamma_{jk}^{l}e_{l}$.
This orthonormal frame induces an orthonormal frame $u_{j}$, $1\leq j\leq2^{m}$,
for the spin bundle $S$. We further choose a local orthonormal section
$\mathtt{l}\left(x\right)$ for the Hermitian line bundle $L$ and
define via $\nabla_{e_{j}}^{A_{0}}\mathtt{l}=\Upsilon_{j}\left(x\right)\mathtt{l}$,
$0\leq j\leq2m$ the Christoffel symbols of the unitary connection
$A_{0}$ on $L$. In terms of the induced frame $u_{j}\otimes\mathtt{l}$,
$1\leq j\leq2^{m}$, for $S\otimes L$ the Dirac operator \prettyref{eq:Semiclassical Magnetic Dirac}
has the form (cf. \cite{Berline-Getzler-Vergne} Section 3.3) 
\begin{eqnarray}
D & = & \gamma^{j}w_{j}^{k}P_{k}+h\left(\frac{1}{4}\Gamma_{jk}^{l}\gamma^{j}\gamma^{k}\gamma_{l}+\Upsilon_{j}\gamma^{j}\right),\quad\mbox{where}\label{eq: Dirac in coords}\\
P_{k} & = & h\partial_{x_{k}}+ia_{k},\label{eq: cov diff in cords}
\end{eqnarray}
and the one form $a$ is given by \prettyref{eq: normal form contact a-1}. 

The expression in \prettyref{eq: Dirac in coords} is formally self-adjoint
with respect to the Riemannian density $e^{0}\wedge\ldots\wedge e^{2m}=\sqrt{g}dx\coloneqq\sqrt{g}dx^{0}\wedge\ldots\wedge dx^{2m}$
with $g=\det\left(g_{ij}\right)$. To get an operator self-adjoint
with respect to the Euclidean density $dx$ one expresses the Dirac
operator in the framing $g^{\frac{1}{4}}u_{j}\otimes\mathtt{l},1\leq j\leq2^{m}$.
In this new frame the expression \prettyref{eq: Dirac in coords}
for the Dirac operator needs to be conjugated by $g^{\frac{1}{4}}$
and hence the term $h\gamma^{j}w_{j}^{k}g^{-\frac{1}{4}}\left(\partial_{x_{k}}g^{\frac{1}{4}}\right)$
added. Hence, the Dirac operator in the new frame has the form 
\[
D=\left[\sigma^{j}w_{j}^{k}\left(\xi_{k}+a_{k}\right)\right]^{W}+hE\in\Psi_{\textrm{cl}}^{1}\left(\varOmega_{\gamma}^{0};\mathbb{C}^{2^{m}}\right),
\]
with $\sigma^{j}=i\gamma^{j}$, for some self-adjoint endomorphism
$E\left(x\right)\in C^{\infty}\left(\varOmega_{\gamma}^{0};i\mathfrak{u}\left(\mathbb{C}^{2^{m}}\right)\right)$.

The one form $a$ is given in terms of these Darboux-Reeb coordinates
by the same formula \prettyref{eq: normal form contact a-1}
\[
a=\varphi d\theta+\frac{1}{2}\sum_{j=1}^{m}\left(x_{j}dx_{j+m}-x_{j+m}dx_{j}\right)+a_{\gamma}^{\infty}
\]
with $a_{\gamma}^{\infty}$ denoting a form on $\varOmega_{\gamma}$
vanishing to infinite order along $\gamma$. Picking a cutoff $\chi_{\gamma}\in C_{c}^{\infty}\left(\varOmega_{\gamma}\right)$
that equals $1$ on $\Omega_{\gamma}$ we may extend the one form
to all of $S^{1}\times\mathbb{R}^{2m}$ via 
\[
a=\underbrace{\varphi d\theta+\frac{1}{2}\sum_{j=1}^{m}\left(x_{j}dx_{j+m}-x_{j+m}dx_{j}\right)}_{\eqqcolon a^{0}}+\chi_{\gamma}a_{\gamma}^{\infty}
\]

The functions $w_{j}^{k}$ are extended such that
\[
\left.\left(w_{j}^{k}\partial_{x_{k}}\otimes dx^{j}\right)\right|_{\left(K_{s}^{0}\right)^{c}}=\partial_{x_{0}}\otimes dx^{0}+\sum_{j=1}^{m}\mu_{j}^{\frac{1}{2}}\left(\partial_{x_{j}}\otimes dx^{j}+\partial_{x_{j+m}}\otimes dx^{j+m}\right)
\]
 (and hence $\left.g\right|_{\left(K_{s}^{0}\right)^{c}}=dx_{0}^{2}+\sum_{j=1}^{m}\mu_{j}\left(dx_{j}^{2}+dx_{j+m}^{2}\right)$)
outside a compact neighborhood $\varOmega_{\gamma}^{0}\Subset K_{s}^{0}$.
The endomorphism $E\left(x\right)\in C_{c}^{\infty}\left(\mathbb{R}^{n};i\mathfrak{u}\left(\mathbb{C}^{2^{m}}\right)\right)$
is extended to an arbitrary self-adjoint endomorphism of compact support.
This now gives the operator
\begin{align}
D_{0} & =\left[\sigma^{j}w_{j}^{k}\left(\xi_{k}+a_{k}^{0}\right)\right]^{W}+\chi_{\gamma}\sigma^{j}a_{\gamma,j}^{\infty}+hE\in\Psi_{\textrm{cl}}^{1}\left(S^{1}\times\mathbb{R}^{2m};\mathbb{C}^{2^{m}}\right)\label{eq:Dirac op weyl quantization without error}
\end{align}
as a well defined formally self adjoint operators on $S^{1}\times\mathbb{R}^{2m}$.
Furthermore, the symbols of $D_{0}+i$ being elliptic in the class
$S^{0}\left(g\right)$ for the order functions $g=\sqrt{1+\sum_{k=0}^{2m}\left(\xi_{k}+a_{k}\right)^{2}}$
it is essentially self adjoint (see \cite{Dimassi-Sjostrand} Ch.
8).

\subsection{\label{subsec:Birkhoff-normal-form}Birkhoff normal form for the
Dirac operator}

Next, we derive a Birkhoff normal form for the Dirac operator \prettyref{eq:Dirac op weyl quantization without error}
on $S^{1}\times\mathbb{R}^{2m}$. First consider the function 
\[
f_{0}:=\sum_{j=1}^{m}\left(x_{j}x_{j+m}+\xi_{j}\xi_{j+m}\right)\in C^{\infty}\left(\mathbb{R}^{2m}\right).
\]
If $H_{f_{0}}$ and $e^{tH_{f_{0}}}$ denote the Hamilton vector field
and time $t$ flow of $f_{0}$ respectively then it is easy to compute
\begin{eqnarray*}
e^{\frac{\pi}{4}H_{f_{0}}}\left(x_{j},\xi_{j};x_{j+m}\xi_{j+m}\right) & = & \left(\frac{x_{j}+\xi_{j+m}}{\sqrt{2}},\frac{-x_{j+m}+\xi_{j}}{\sqrt{2}};\frac{x_{j+m}+\xi_{j}}{\sqrt{2}},\frac{-x_{j}+\xi_{j+m}}{\sqrt{2}}\right).
\end{eqnarray*}
We abbreviate $\left(x',\xi'\right)=\left(x_{1},\ldots,x_{m};\xi_{1},\ldots,\xi_{m}\right)$,
\\
$\left(x'',\xi''\right)=\left(x_{m+1},\ldots,x_{2m};\xi_{m+1},\ldots,\xi_{2m}\right)$
and $\left(x,\xi\right)=\left(x_{0},x',x'';\xi_{0},\xi',\xi''\right)$. 

Using Egorov's theorem, the operator \prettyref{eq:Dirac op weyl quantization without error}
is conjugated to 
\begin{align}
e^{\frac{i\pi}{4h}f_{0}^{W}}D_{0}e^{-\frac{i\pi}{4h}f_{0}^{W}} & =d_{0}^{W},\quad\textrm{ with}\label{eq:first conjugation normal form}\\
d_{0} & =\sigma^{j}w_{j,f_{0}}^{0}\left(\xi_{0}+\varphi_{f_{0}}\right)+\sqrt{2}\left(\sigma^{j}w_{j,f_{0}}^{k}\xi_{k}+\sigma^{j}w_{j,f_{0}}^{k+m}x_{k}\right)+\sigma^{j}r_{j}^{\infty}+O\left(h\right)\label{eq:symbol Dirac operator}\\
\textrm{where }\;w_{j,f_{0}}^{k} & =\left(e^{-\frac{\pi}{4}H_{f_{0}}}\right)^{*}w_{j}^{k}\\
\varphi_{f_{0}} & =\left(e^{-\frac{\pi}{4}H_{f_{0}}}\right)^{*}\varphi\nonumber \\
r_{j}^{\infty} & =\left(e^{-\frac{\pi}{4}H_{f_{0}}}\right)^{*}\chi_{\gamma}a_{\gamma,j}^{\infty}
\end{align}
Using the formulas \prettyref{eq: quadratics-1}, \prettyref{eq: function of quadratics-1}
we may also calculate 
\begin{eqnarray*}
\varphi_{f_{0}} & = & T_{\gamma}+\chi^{-}\tilde{Q}^{h,-}+\chi^{+}\varphi^{+}\left(\tilde{Q}\right)\quad\textrm{with},\\
\tilde{Q}_{j}^{e} & = & \frac{1}{4}\left[\left(x_{j}-\xi_{j+m}\right)^{2}+\left(x_{j+m}-\xi_{j}\right)^{2}\right]\\
\tilde{Q}_{j}^{h} & = & \frac{1}{2}\left(x_{N_{e}+j}-\xi_{N_{e}+j+m}\right)\left(x_{N_{e}+j+m}-\xi_{N_{e}+j}\right)\\
\tilde{Q}_{j}^{l,\textrm{Re}} & = & \frac{1}{2}\left(x_{m-2j+2}-\xi_{2m-2j+2}\right)\left(x_{2m-2j+1}-\xi_{m-2j+1}\right)\\
 &  & -\frac{1}{2}\left(x_{m-2j+1}-\xi_{2m-2j+1}\right)\left(x_{2m-2j+2}-\xi_{m-2j+2}\right)\\
\tilde{Q}_{j}^{l,\textrm{Im}} & = & \frac{1}{2}\left(x_{m-2j+1}-\xi_{2m-2j+1}\right)\left(x_{2m-2j+1}-\xi_{m-2j+1}\right)\\
 &  & +\frac{1}{2}\left(x_{m-2j+2}-\xi_{2m-2j+2}\right)\left(x_{2m-2j+2}-\xi_{m-2j+2}\right)\;\textrm{ and}\\
\tilde{Q}^{h,-} & = & \frac{\pi}{4}\sum_{j=1}^{N_{h}^{-}}\left[\left(x_{N_{e}+j}-\xi_{N_{e}+j+m}\right)^{2}+\left(x_{N_{e}+j+m}-\xi_{N_{e}+j}\right)^{2}\right]
\end{eqnarray*}

Next, set 
\begin{eqnarray}
\bar{\varphi}_{f_{0}}=\bar{\varphi} & = & T_{\gamma}+\chi^{-}\bar{Q}^{h,-}+\chi^{+}\varphi^{+}\left(\bar{Q}\right)\quad\textrm{with},\label{eq: effective hamiltonian}\\
\bar{Q}_{j}^{e} & = & \frac{1}{4}\left[\xi_{j+m}^{2}+x_{j+m}^{2}\right]\nonumber \\
\bar{Q}_{j}^{h} & = & -\frac{1}{2}x_{N_{e}+j+m}\xi_{N_{e}+j+m}\nonumber \\
\bar{Q}_{j}^{l,\textrm{Re}} & = & \frac{1}{2}\left(x_{2m-2j+2}\xi_{2m-2j+1}-x_{2m-2j+1}\xi_{2m-2j+2}\right)\nonumber \\
\bar{Q}_{j}^{l,\textrm{Im}} & = & -\frac{1}{2}\left(x_{2m-2j+1}\xi_{2m-2j+1}+x_{2m-2j+2}\xi_{2m-2j+2}\right)\;\textrm{ and}\nonumber \\
\bar{Q}^{h,-} & = & \frac{\pi}{4}\sum_{j=1}^{N_{h}^{-}}\left[\xi_{N_{e}+j+m}^{2}+x_{N_{e}+j+m}^{2}\right]\label{eq: conjugated quadratics}
\end{eqnarray}
Below denote by $o_{N}',o_{N}''\subset S_{\textrm{cl}}^{1}\left(T^{*}S^{1}\times\mathbb{R}^{4m};\mathbb{C}^{l}\right)$
the subspace of self-adjoint symbols $a:\left(0,1\right]_{h}\rightarrow C^{\infty}\left(T^{*}S^{1}\times\mathbb{R}^{4m};i\mathfrak{u}\left(2^{m}\right)\right)$
such that each of the coefficients $a_{k}$, $k=0,1,2,\ldots$ in
its symbolic expansion vanishes to order $N$ in $\left(\xi_{0}+\bar{\varphi},x',\xi'\right)$
and $\left(x'',\xi''\right)$ respectively. We also denote by $o_{N}',o_{N}''$
the space of Weyl quantizations of the respective symbols. One clearly
has $\varphi_{f_{0}}=\bar{\varphi}+o_{1}'o_{1}''$. A Taylor expansion
of $d_{0}$ \prettyref{eq:symbol Dirac operator} now gives $r_{j}^{0}\in o_{2}'$,
$r_{j}^{1}\in o_{1}'o_{1}''$, $r_{j}^{\infty}\in o_{\infty}''$,
$0\leq j\leq2m$, such that 
\begin{eqnarray*}
d_{0} & = & \sqrt{2}\sigma^{j}\left(\bar{w}_{j}^{0}\left(\xi_{0}+\bar{\varphi}\right)+\bar{w}_{j}^{k}\xi_{k}+\bar{w}_{j}^{k+m}x_{k}\right)+\sigma^{j}\left(r_{j}^{0}+r_{j}^{1}+r_{j}^{\infty}\right)+O\left(h\right)
\end{eqnarray*}
and where $\bar{w}_{j}^{k}\left(x_{0}\right)=w_{j}^{k}\left(x_{0},0,0\right)$.

On squaring using \prettyref{eq: diagonalizing metric} we obtain
\begin{align}
\left(d_{0}^{W}\right)^{2} & =Q_{0}^{W}+o_{2}'o_{1}''+o_{\infty}''+O\left(h\right),\:\textrm{ with}\nonumber \\
Q_{0} & =\begin{bmatrix}\xi_{0}+\bar{\varphi} & \xi' & x'\end{bmatrix}\begin{bmatrix}\bar{g}^{00}\left(x_{0}\right) & \bar{g}^{k0}\left(x_{0}\right) & \bar{g}^{\left(k+m\right)0}\left(x_{0}\right)\\
\bar{g}^{0l}\left(x_{0}\right) & \bar{g}^{kl}\left(x_{0}\right) & \bar{g}^{k\left(l+m\right)}\left(x_{0}\right)\\
\bar{g}^{0\left(l+m\right)}\left(x_{0}\right) & \bar{g}^{\left(k+m\right)l}\left(x_{0}\right) & \bar{g}^{\left(k+m\right)\left(l+m\right)}\left(x_{0}\right)
\end{bmatrix}\begin{bmatrix}\xi_{0}+\bar{\varphi}\\
\xi'\\
x'
\end{bmatrix}.\label{eq: Q0 square Dirac}
\end{align}
Here $\bar{g}^{kl}\left(x_{0}\right)=2g^{kl}\left(x_{0},0,0\right)$
and $g^{kl}$ the components of the inverse metric in Reeb Darboux
coordinates along the orbit and 
\[
\bar{g}^{00}\left(x_{0}\right)=\frac{1}{T_{\gamma}^{2}\left|R\right|^{2}}.
\]

Next we consider another function $f_{1}$ of the form 
\[
f_{1}=\frac{1}{2}\begin{bmatrix}x' & \xi'\end{bmatrix}\begin{bmatrix}\alpha_{m\times m}\left(x_{0}\right) & \gamma_{m\times m}\left(x_{0}\right)\\
\gamma_{m\times m}^{t}\left(x_{0}\right) & \beta_{m\times m}\left(x_{0}\right)
\end{bmatrix}\begin{bmatrix}x'\\
\xi'
\end{bmatrix}
\]
where $\alpha,\beta$ and $\gamma$ are matrix valued functions of
the given orders with $\alpha,\beta$ symmetric. An easy computation
now shows 
\begin{eqnarray*}
\left(e^{H_{f_{1}}}\right)^{*}\begin{bmatrix}\xi_{0}+\bar{\varphi}\\
x'\\
\xi'
\end{bmatrix} & = & e^{\Lambda}\begin{bmatrix}\xi_{0}+\bar{\varphi}\\
x'\\
\xi'
\end{bmatrix}+o_{2}'\quad\textrm{with}\\
\Lambda\left(x_{0}\right) & = & \begin{bmatrix}0 & 0 & 0\\
0 & 0 & -I_{m\times m}\\
0 & I_{m\times m} & 0
\end{bmatrix}\begin{bmatrix}0 & 0 & 0\\
0 & \alpha_{m\times m}\left(x_{0}\right) & \gamma_{m\times m}\left(x_{0}\right)\\
0 & \gamma_{m\times m}^{t}\left(x_{0}\right) & \beta_{m\times m}\left(x_{0}\right)
\end{bmatrix}.
\end{eqnarray*}
From the suitability assumption \prettyref{eq:Diagonalizability assumption-1},
we have that there exists a smooth matrix valued functions $\alpha,\beta$
and $\gamma$ such that 
\begin{align}
\left(e^{H_{f_{1}}}\right)^{*}Q_{0} & =\begin{bmatrix}\xi_{0}+\bar{\varphi} & \xi' & x'\end{bmatrix}e^{\Lambda^{t}}\begin{bmatrix}\bar{g}^{00}\left(x_{0}\right) & \bar{g}^{k0}\left(x_{0}\right) & \bar{g}^{\left(k+m\right)0}\left(x_{0}\right)\\
\bar{g}^{0l}\left(x_{0}\right) & \bar{g}^{kl}\left(x_{0}\right) & \bar{g}^{k\left(l+m\right)}\left(x_{0}\right)\\
\bar{g}^{0\left(l+m\right)}\left(x_{0}\right) & \bar{g}^{\left(k+m\right)l}\left(x_{0}\right) & \bar{g}^{\left(k+m\right)\left(l+m\right)}\left(x_{0}\right)
\end{bmatrix}e^{\Lambda}\begin{bmatrix}\xi_{0}+\bar{\varphi}\\
\xi'\\
x'
\end{bmatrix}\nonumber \\
=Q_{1} & \coloneqq\bar{g}^{00}\left(x_{0}\right)\left(\xi_{0}+\bar{\varphi}\right)^{2}+\left[\sum_{j=1}^{m}\mu_{j}\left(x_{j}^{2}+\xi_{j}^{2}\right)\right]\nonumber \\
 & +2\sum_{j=1}^{m}\left(\xi_{0}+\bar{\varphi}\right)\left[h_{j}^{0}\left(x_{0}\right)\xi_{j}+h_{j}^{1}\left(x_{0}\right)x_{j}\right]+o_{3}'\label{eq:Q1 square Dirac}
\end{align}
and where
\begin{eqnarray*}
\begin{bmatrix}\bar{g}^{00}\left(x_{0}\right)\\
h_{j}^{0}\left(x_{0}\right)\\
h_{j}^{1}\left(x_{0}\right)
\end{bmatrix} & = & e^{\Lambda^{t}}\begin{bmatrix}\bar{g}^{00}\left(x_{0}\right)\\
\bar{g}^{0l}\left(x_{0}\right)\\
\bar{g}^{0\left(l+m\right)}\left(x_{0}\right)
\end{bmatrix}.
\end{eqnarray*}
Next, if 
\[
f_{2}=\left(\xi_{0}+\bar{\varphi}\right)\begin{bmatrix}\frac{1}{\mu}\xi' & \frac{1}{\mu}x'\end{bmatrix}\begin{bmatrix}0 & -I_{m\times m}\\
I_{m\times m} & 0
\end{bmatrix}\begin{bmatrix}h_{j}^{0}\left(x_{0}\right)\\
h_{j}^{1}\left(x_{0}\right)
\end{bmatrix}
\]
we may compute 
\begin{equation}
\left(e^{H_{f_{2}}}\right)^{*}Q_{1}=Q_{2}\coloneqq\bar{g}^{00}\left(x_{0}\right)\left(\xi_{0}+\bar{\varphi}\right)^{2}+\left[\sum_{j=1}^{m}\mu_{j}\left(x_{j}^{2}+\xi_{j}^{2}\right)\right]+o_{3}'.\label{eq: Q2 square Dirac}
\end{equation}
Finally, letting $L_{\gamma}$ denote the length of the Reeb orbit
note
\begin{eqnarray*}
L_{\gamma} & = & \exp\left\{ -\frac{1}{2}\frac{\int_{0}^{1}dx_{0}\left(g^{00}\right)^{-1/2}\left(\ln g^{00}\right)}{\int_{0}^{1}dx_{0}\left(g^{00}\right)^{-1/2}}\right\} 
\end{eqnarray*}
and set 
\[
a\left(x_{0}\right)\coloneqq\left(g^{00}\right)^{1/2}\int_{0}^{\theta}d\theta'\left(g^{00}\right)^{-1/2}\ln\left[T_{\gamma}L_{\gamma}\left(g^{00}\right)^{1/2}\right]
\]
to compute 
\begin{equation}
\left(e^{H_{a\xi}}\right)^{*}Q_{2}=\frac{1}{L_{\gamma}^{2}}\left(\xi_{0}+\bar{\varphi}\right)^{2}+\left[\sum_{j=1}^{m}\mu_{j}\left(x_{j}^{2}+\xi_{j}^{2}\right)\right]+o_{3}'.\label{eq: Q3 square Dirac}
\end{equation}

Letting 
\[
H_{2}=\frac{1}{2}\sum_{j=1}^{m}\mu_{j}\left(x_{j}^{2}+\xi_{j}^{2}\right),
\]
 using \prettyref{eq: Q0 square Dirac}, \prettyref{eq:Q1 square Dirac},
\prettyref{eq: Q2 square Dirac} and \prettyref{eq: Q3 square Dirac}
Egorov's theorem now gives
\begin{align}
d_{00}^{W}\coloneqq e^{\frac{i}{h}a\xi^{W}}e^{\frac{i}{h}f_{2}^{W}}e^{\frac{i}{h}f_{1}^{W}}d_{0}^{W}e^{-\frac{i}{h}f_{1}^{W}}e^{-\frac{i}{h}f_{2}^{W}}e^{-\frac{i}{h}a\xi{}^{W}} & =\left(\sum_{j=0}^{2m}\sigma_{j}b_{j}\right)^{W}+ho_{0}\quad\textrm{with }\label{eq:symbol dirac d_1}\\
\sum_{j=0}^{2m}b_{j}^{2} & =\left(\frac{1}{L_{\gamma}^{2}}\left(\xi_{0}+\bar{\varphi}\right)^{2}+2H_{2}\right)^{W}+o_{2}'o_{1}''+o_{\infty}''.\nonumber 
\end{align}
Another Taylor expansion in the variables $\left(\xi_{0}+\bar{\varphi},x',\xi';x'',\xi''\right)$
gives $A=\left(a_{jk}\left(x_{0}\right)\right)\in C^{\infty}\left(S^{1};\mathfrak{so}\left(n\right)\right)$
and $r_{j,0}\in o_{1}'o_{1}''$,$r_{j,1}\in o_{2}'$, $r_{j,\infty}\in o_{\infty}''$,
$j=0,\ldots,2m$, such that 
\[
e^{-A}\begin{bmatrix}b_{0}\\
\vdots\\
b_{2m}
\end{bmatrix}=\begin{bmatrix}\frac{1}{L_{\gamma}}\left(\xi_{0}+\bar{\varphi}\right)\\
\left(2\mu_{1}\right)^{\frac{1}{2}}x_{1}\\
\left(2\mu_{1}\right)^{\frac{1}{2}}\xi_{1}\\
\vdots\\
\left(2\mu_{m}\right)^{\frac{1}{2}}x_{m}\\
\left(2\mu_{m}\right)^{\frac{1}{2}}\xi_{m}
\end{bmatrix}+\begin{bmatrix}r_{0,0}\\
\vdots\\
r_{2m,0}
\end{bmatrix}+\begin{bmatrix}r_{0,1}\\
\vdots\\
r_{2m,1}
\end{bmatrix}+\begin{bmatrix}r_{0,\infty}\\
\vdots\\
r_{2m,\infty}
\end{bmatrix}.
\]
We may now set $c_{A}=\frac{1}{i}a_{jk}\sigma^{j}\sigma^{k}\in C^{\infty}\left(S^{1};i\mathfrak{u}\left(2^{m}\right)\right)$
and compute 
\begin{align}
e^{ic_{A}^{W}}d_{00}^{W}e^{-ic_{A}^{W}} & =d_{1}^{W},\quad\textrm{where}\label{eq: d0 =000026 d1 are conjugate}\\
d_{1} & =H_{1}+\sigma^{j}\left(r_{j,0}+r_{j,1}+r_{j,\infty}\right)+O\left(h\right),\quad\textrm{and}\label{eq:dirac operator for formal birkhoff}\\
H_{1} & \coloneqq\frac{1}{L_{\gamma}}\left(\xi_{0}+\bar{\varphi}\right)\sigma_{0}+\sum_{j=1}^{m}\left(2\mu_{j}\right)^{\frac{1}{2}}\left(x_{j}\sigma_{2j-1}+\xi_{j}\sigma_{2j}\right).\label{eq: model Dirac symbol}
\end{align}
Finally, if we further Taylor expand $r_{0,0}+r_{0,1}=l_{00}\left(x_{0},\frac{1}{L}\left(\xi_{0}+\bar{\varphi}\right),x'',\xi''\right)+x_{1}l_{01}+\xi_{1}l_{02}+\ldots+\xi_{m}l_{0\left(2m\right)}$,
then a further conjugation of $d_{1}^{W}$ by $e^{ic_{2}^{W}}$; $c_{2}=\frac{1}{i}l_{0k}\sigma^{0}\sigma^{k}$,
it is possible to make $r_{0,0}+r_{0,1}$ independent of $\left(x',\xi'\right)$
in \prettyref{eq:dirac operator for formal birkhoff}.

\subsubsection{\label{subsec: Weyl product and Koszul}Weyl product and Koszul complexes}

We now derive a formal Birkhoff normal form for the symbol $d_{1}$
in \prettyref{eq:dirac operator for formal birkhoff}. Since much
of what follows here proceeds in a similar fashion to \cite{Savale2017-Koszul}
Section 5, we refer there for necessary modifications to avoid repetition
of arguments. First denote by $R=C^{\infty}\left(S_{x_{0}}^{1}\right)$
the ring of real valued functions on the circle. Further define 
\[
S\coloneqq R\left\llbracket \xi_{0}+\bar{\varphi},x',\xi';x'',\xi'';h\right\rrbracket 
\]
the ring of formal power series in the further given $4m+2$ variables
with coefficients in $R$. The ring $S\otimes\mathbb{C}$ is now equipped
with the Weyl product 
\[
a\ast b\coloneqq\left[e^{\frac{ih}{2}\left(\partial_{r_{1}}\partial_{s_{2}}-\partial_{r_{2}}\partial_{s_{1}}\right)}\left(a\left(s_{1},r_{1};h\right)b\left(s_{2},r_{2};h\right)\right)\right]_{x=s_{1}=s_{2},\xi=r_{1}=r_{2}},
\]
(again using the convention $\theta=x_{0}$) corresponding to the
composition formula for pseudo-differential operators, with 
\begin{eqnarray*}
\left[a,b\right] & \coloneqq & a\ast b-b\ast a
\end{eqnarray*}
being the corresponding Weyl bracket. It is an easy exercise to show
that for $a,b\in S$ real valued, the commutator $i\left[a,b\right]\in S$
is real valued.

Next, we define a filtration on $S$. Each monomial $h^{k}\left(\xi_{0}+T_{\gamma}\right)^{a}\left(x'\right)^{\alpha'}\left(\xi'\right)^{\beta'}\left(x''\right)^{\alpha''}\left(\xi''\right)^{\beta''}$
in $S$ is given the weight $2k+a+\left|\alpha'\right|+\left|\beta'\right|+\left|\alpha''\right|+\left|\beta''\right|$.
The ring $S$ is equipped with a decreasing filtration 
\begin{eqnarray*}
S=O_{0} & \supset & O_{1}\supset\ldots\supset O_{N}\supset\ldots,\\
\bigcap_{N}O_{N} & = & \left\{ 0\right\} ,
\end{eqnarray*}
where $O_{N}$ consists of those power series with monomials of weight
$N$ or more. Similar filtrations
\begin{align*}
S & =O_{0}'\supset O_{1}'\supset\ldots\supset O_{N}'\supset\ldots\\
S & =O_{0}''\supset O_{1}''\supset\ldots\supset O_{N}''\supset\ldots
\end{align*}
maybe defined with $O_{N}'$, $O_{N}''$ consisting of power series
in those monomials with $2k+a+\left|\alpha'\right|+\left|\beta'\right|\geq N$
and $2k+\left|\alpha''\right|+\left|\beta''\right|\geq N$ respectively.
It is an exercise to show that 
\begin{eqnarray*}
O_{N}\ast O_{M} & \subset & O_{N+M}\\
\left[O_{N},O_{M}\right] & \subset & ihO_{N+M-2}
\end{eqnarray*}
and similar inclusions holding for its primed versions. The associated
grading is given by 
\[
S=\bigoplus_{N=0}^{\infty}S_{N}
\]
where $S_{N}$ consists of those power series with monomials of weight
exactly $N$ . We also define the quotient ring $D_{N}\coloneqq S/O_{N+1}$
whose elements may be identified with the set of homogeneous polynomials
with monomials of weight at most $N$. The ring $D_{N}$ is also similarly
graded and filtered. In similar vein, we may also define the ring
\[
S\left(m\right)=S\otimes\mathfrak{gl}_{\mathbb{C}}\left(2^{m}\right)
\]
of $R\otimes\mathfrak{gl}_{\mathbb{C}}\left(2^{m}\right)$ valued
formal power series in $\left(\xi_{0}+\bar{\varphi},x',\xi';h\right)$.
The ring $S\left(m\right)$ is equipped with an induced product $\ast$
and decreasing filtration
\begin{eqnarray*}
O_{0}\left(m\right) & \supset & O_{1}\left(m\right)\supset\ldots\supset O_{N}\left(m\right)\supset\ldots,\\
\bigcap_{N}O_{N}\left(m\right) & = & \left\{ 0\right\} ,
\end{eqnarray*}
where $O_{N}\left(m\right)=O_{N}\otimes\mathfrak{gl}_{\mathbb{C}}\left(2^{m}\right)$.
It is again a straightforward exercise to show that for $a,b\in S\otimes i\mathfrak{u}_{\mathbb{C}}\left(2^{m}\right)$
self-adjoint, the commutator $i\left[a,b\right]\in S\otimes i\mathfrak{u}_{\mathbb{C}}\left(2^{m}\right)$
is self-adjoint.

\subsubsection{Koszul complexes}

Let us now again consider the $2m$ and $2m+1$ dimensional real inner
product spaces $V=\mathbb{R}\left[e_{1},\ldots,e_{2m}\right]$ and
$W=\mathbb{R}\left[e_{0}\right]\oplus V$ from \prettyref{subsec:Clifford algebra}.
Considering the chain groups $D_{N}\otimes\Lambda^{k}V$, $k=0,1,\ldots,n$,
one may define four differentials 
\begin{eqnarray*}
w_{x}^{0} & = & \sum_{j=1}^{m}\mu_{j}^{\frac{1}{2}}\left(x_{j}e_{2j-1}\land+\xi_{j}e_{2j}\land\right)\\
i_{x}^{0} & = & \sum_{j=1}^{m}\mu_{j}^{\frac{1}{2}}\left(x_{j}i_{e_{2j-1}}+\xi_{j}i_{e_{2j}}\right)\\
w_{\partial}^{0} & = & \sum_{j=1}^{m}\mu_{j}^{\frac{1}{2}}\left(\partial_{x_{j}}e_{2j-1}\land+\partial_{\xi_{j}}e_{2j}\land\right)\\
i_{\partial}^{0} & = & \sum_{j=1}^{m}\mu_{j}^{\frac{1}{2}}\left(\partial_{x_{j}}i_{e_{2j-1}}+\partial_{\xi_{j}}i_{e_{2j}}\right).
\end{eqnarray*}

Similarly, we may consider the chain groups $D_{N}\otimes\Lambda^{k}W$,
$k=0,1,\ldots,n$, one may define four differentials
\begin{eqnarray*}
w_{x} & = & \frac{1}{L_{\gamma}}\left(\xi_{0}+\bar{\varphi}\right)e_{0}\land+2^{\frac{1}{2}}w_{x}^{0}\\
i_{x} & = & \frac{1}{L_{\gamma}}\left(\xi_{0}+\bar{\varphi}\right)i_{e_{0}}+2^{\frac{1}{2}}i_{x}^{0}\\
w_{\partial} & = & \partial_{\xi_{0}}e_{0}\land+2^{\frac{1}{2}}w_{\partial}^{0}\\
i_{\partial} & = & \partial_{\xi_{0}}i_{e_{0}}+2^{\frac{1}{2}}i_{\partial}^{0}.
\end{eqnarray*}

Next, we define twisted Koszul differentials on $D_{N}\otimes\Lambda^{k}V$
via
\begin{eqnarray*}
\tilde{w}_{\partial}^{0} & = & \frac{i}{h}\sum_{j=1}^{m}\mu_{j}^{\frac{1}{2}}\left(\textrm{ad}_{x_{j}}e_{2j-1}\land+\textrm{ad}_{\xi_{j}}e_{2j}\land\right)=\sum_{j=1}^{m}\mu_{j}^{\frac{1}{2}}\left(\partial_{x_{j}}e_{2j}\land-\partial_{\xi_{j}}e_{2j-1}\land\right)\\
\tilde{i}_{\partial}^{0} & = & \frac{i}{h}\sum_{j=1}^{m}\mu_{j}^{\frac{1}{2}}\left(\textrm{ad}_{x_{j}}i_{e_{2j-1}}+\textrm{ad}_{\xi_{j}}i_{e_{2j}}\right)=\sum_{j=1}^{m}\mu_{j}^{\frac{1}{2}}\left(\partial_{x_{j}}i_{e_{2j}}-\partial_{\xi_{j}}i_{e_{2j-1}}\right).
\end{eqnarray*}
We note that the above are symplectic adjoints to their untwisted
counterparts with respect to the symplectic pairing $\sum_{j=1}^{m}e_{2j-1}\wedge e_{2j}$
on $V$. 

Similar twisted Koszul differentials on $D_{N}\otimes\Lambda^{k}W$
are defined via

\begin{eqnarray*}
\tilde{w}_{\partial} & = & \frac{1}{L_{\gamma}}\left(\textrm{ad}_{\xi_{0}+\bar{\varphi}}\right)e_{0}\wedge+2^{\frac{1}{2}}\tilde{w}_{\partial}^{0}\\
\tilde{i}_{\partial} & = & \frac{1}{L_{\gamma}}\left(\textrm{ad}_{\xi_{0}+\bar{\varphi}}\right)i_{e_{0}}+2^{\frac{1}{2}}\tilde{i}_{\partial}^{0}.
\end{eqnarray*}
We note that in what follows works with any leading terms replacing
$e_{0}\wedge$ and $i_{e_{0}}$ above that would serve as differentials. 

We now compute the twisted combinatorial Laplacian to be 
\begin{eqnarray*}
\tilde{\Delta}^{0} & = & \tilde{w}_{\partial}^{0}i_{x}^{0}+i_{x}^{0}\tilde{w}_{\partial}^{0}\\
 & = & -\left(w_{x}^{0}\tilde{i}_{\partial}^{0}+\tilde{i}_{\partial}^{0}w_{x}^{0}\right)\\
 & = & \sum_{j=1}^{m}\mu_{j}\left[\xi_{j}\partial_{x_{j}}-x_{j}\partial_{\xi_{j}}+e_{2j}i_{e_{2j-1}}-e_{2j-1}i_{e_{2j}}\right].
\end{eqnarray*}
One may similarly define $\tilde{\Delta}$. Next, we define the space
of twisted $\tilde{\Delta}^{0}$-harmonic, $\bar{\varphi}$-commuting,
$x_{0}$- independent elements 
\begin{eqnarray*}
\mathcal{H}_{N}^{k} & = & \left\{ \omega\in D_{N}\otimes\Lambda^{k}W|\,\tilde{\Delta}^{0}\omega=0,\,\partial_{x_{0}}\omega=0,\,\textrm{ad}_{\bar{\varphi}}\omega=0\right\} \\
\mathcal{H}^{k} & = & \left\{ \omega\in S\otimes\Lambda^{k}W|\,\tilde{\Delta}^{0}\omega=0,\,\partial_{x_{0}}\omega=0,\,\textrm{ad}_{\bar{\varphi}}\omega=0\right\} .
\end{eqnarray*}
The following version of the Hodge decomposition theorem follows in
a similar fashion to \cite{Savale2017-Koszul} Lemma 5.1. We only
note that the $\xi_{0}$-independence in the definition of $\mathcal{H}_{N}^{k}$
from \cite{Savale2017-Koszul} is here replaced by the condition $\textrm{ad}_{\xi_{0}+\bar{\varphi}}\omega=0$,
which on account of non-resonance is equivalent to $\textrm{ad}_{\xi_{0}}\omega=\partial_{x_{0}}\omega=0,\,\textrm{ad}_{\bar{\varphi}}\omega=0$.
\begin{lem}
\label{lem:Koszul-Hodge decomposition }The $k$-th chain group is
spanned by the three subspaces 
\[
D_{N}\otimes\Lambda^{k}W=\mathbb{R}\left[\textrm{Im}\left(i_{x}\tilde{w}_{\partial}\right),\textrm{Im}\left(\tilde{w}_{\partial}i_{x}\right),\mathcal{H}_{N}^{k}\right].
\]

\end{lem}

\subsubsection{Formal Birkhoff normal form}

As in \cite{Savale2017-Koszul} section 5.2 the Koszul complexes now
allow us to complete the Birkhoff normal form procedure for the symbol
$d_{1}$ in \prettyref{eq:dirac operator for formal birkhoff}. Define
the Clifford quantization of an element in $a\in S\otimes\Lambda^{k}W$,
using \prettyref{eq: clifford quantization} as an element in 
\[
c_{0}\left(a\right)\coloneqq i^{\frac{k\left(k+1\right)}{2}}c\left(a\right)\in S\left(m\right).
\]
This gives an isomorphism
\begin{equation}
c_{0}:S\otimes\Lambda^{\textrm{odd/even}}W\rightarrow S\otimes i\mathfrak{u}_{\mathbb{C}}\left(2^{m}\right)\label{eq: clifford isomorphism}
\end{equation}
 of real elements of the even or odd exterior algebra with self-adjoint
elements in $S\left(m\right)$. In a fashion similar to \cite{Savale2017-Koszul}
we may now prove the following formal Birkhoff normal form for the
symbol $d_{1}$. Below the symbol $H_{1}$is as in \prettyref{eq: model Dirac symbol}.
\begin{prop}
\label{prop: normal form d1}There exist $f\in O_{1}'\cap O_{3}$,
$a\in O_{2}\otimes\Lambda^{\textrm{even}}W$ and $\omega\in\mathcal{H}^{\textrm{odd}}\cap O_{1}'\cap O_{2}$
such that 
\begin{equation}
e^{ic_{0}\left(a\right)}e^{\frac{i}{h}f}d_{1}e^{-\frac{i}{h}f}e^{-ic_{0}\left(a\right)}=H_{1}+c_{0}\left(\omega\right).\label{eq: normal form d1}
\end{equation}
\end{prop}

$\quad$

\section{\label{sec:Reduction to S1 times R2m}Reduction to $S^{1}\times\mathbb{R}^{2m}$}

We now return to the study of the traces $\mathcal{T}_{B_{v},B_{v}}^{\theta}\left(D\right)$
of the fourth kind in \prettyref{eq: four traces}. The asymptotics
of these traces can be reduced to $S^{1}\times\mathbb{R}^{2m}$. This
however first requires a modification lemma as \prettyref{lem:changing symbol near crit set}
and the definition and construction of another trapping time/function. 

Let $\Gamma\subset\Omega\subset\varOmega$ be any subcover $\delta\in\left(0,\frac{1}{2}\right)$
and $\tau>0$ as before. We define an trapping time in a similar fashion
to \prettyref{eq: exit time 1}
\begin{align*}
T_{v}\coloneqq & \frac{1}{\inf_{\left(g,\mathtt{v}\right)\in\mathcal{G}_{v}\times S^{0}\left(X;U\left(S\right)\right)}\left|H_{g,\mathtt{v}}d\right|}\\
\mathcal{G}_{v}\coloneqq & \left\{ g\in S_{\delta}^{0}\left(X;\left[0,1\right]\right)|\left.g\right|_{\Sigma_{\left[-\tau,\tau\right]}^{D}\cap\tilde{\Omega}_{\gamma_{v}}}=1,\;\left.g\right|_{\left(\Sigma_{\left[-8\tau,8\tau\right]}^{D}\cap\tilde{\varOmega}_{\gamma_{v}}\right)^{c}}=0\right\} 
\end{align*}
and set 
\[
T_{\left(\Omega,\varOmega\right)}^{\tau}\coloneqq\underset{1\leq v\leq M}{\min}T_{v}.
\]
We now have an analog of \prettyref{prop: Large Extension time}.
\begin{prop}
Let $\varOmega$ be a collection of Darboux-Reeb charts and $T>0$.
Then for each $\tau$ sufficiently small there exists an open sub-cover
$\Gamma\subset\Omega\subset\varOmega$ such that 
\begin{equation}
T_{\left(\Omega,\varOmega\right)}^{\tau}>T.\label{eq: second extension time}
\end{equation}
\end{prop}
\begin{proof}
The proof is similar to \prettyref{prop: Large Extension time} with
a some modifications that we precise. Let $0<\varepsilon\ll1$, be
sufficiently small such that for each Reeb orbit $\gamma_{v}$ the
set $A_{\varepsilon}\coloneqq S_{x_{0}}^{1}\times B_{\mathbb{R}^{2m}}\left(\varepsilon\right)\subset\varOmega_{\gamma_{v}}^{0}$
is contained inside the Darboux-Reeb chart \prettyref{prop:normal structure contact form}.
Next for $\left(x',\xi'\right)=\left(x_{1},\ldots,x_{m};\xi_{1},\ldots,\xi_{m}\right)$,
$\left(x'',\xi''\right)=\left(x_{m+1},\ldots,x_{2m};\xi_{m+1},\ldots,\xi_{2m}\right)$
set $\tilde{C}_{\varepsilon}\coloneqq\left\{ x''^{2}+\xi''^{2}<\varepsilon^{2}\right\} \subset T^{*}S_{x_{0}}^{1}\times\mathbb{R}_{x',x'',\xi',\xi''}^{4m}$.
Also set 
\[
U_{\varepsilon,\tau}\coloneqq\left\{ \frac{1}{L_{\gamma}^{2}}\left(\xi_{0}+\bar{\varphi}\right)^{2}+2\sum_{j=1}^{m}\mu_{j}\left(x_{j}^{2}+\xi_{j}^{2}\right)<\tau^{2},\:x''^{2}+\xi''^{2}<\varepsilon^{2}\right\} \subset\tilde{C}_{\varepsilon}
\]
 with $\bar{\varphi}=\bar{\varphi}\left(x'',\xi''\right)$ as in \prettyref{eq: effective hamiltonian}.
Also denote by $o_{N}'$, $o_{N}''$ functions which vanish to order
$N$ in $\left(\xi_{0}+\bar{\varphi},x',\xi'\right)$ and $\left(x'',\xi''\right)$
respectively. Then as in \prettyref{subsec:Birkhoff-normal-form}
(eqns \prettyref{eq:first conjugation normal form}, \prettyref{eq:symbol dirac d_1},
\prettyref{eq: d0 =000026 d1 are conjugate}, \prettyref{eq:dirac operator for formal birkhoff}
and \prettyref{eq: model Dirac symbol}) there exists $0<\tau\ll1$
sufficiently small of the following significance: for each $1\leq v\leq M$
there exists a neighborhood $M_{v}\subset\tilde{A}_{\varepsilon}$
of $\tilde{A}_{\frac{\varepsilon}{8}}\cap\Sigma_{0}^{D}$, a Hamiltonian
symplectomorphism
\begin{align*}
\kappa_{v}\coloneqq e^{H_{f_{1}}}\circ e^{H_{f_{0}}} & :U_{\varepsilon,\tau}\rightarrow M_{v}\\
\kappa_{v}\left(x_{0},0,x'';-\bar{\varphi},0,\xi''\right) & =\left(x_{0},-\frac{\xi''}{\sqrt{2}},\frac{x''}{\sqrt{2}};-\bar{\varphi},\frac{x''}{\sqrt{2}},\frac{\xi''}{\sqrt{2}}\right)
\end{align*}
a self-adjoint endomorphism $c_{A}\in C^{\infty}\left(U_{\varepsilon,\tau};i\mathfrak{u}\left(2^{m}\right)\right)$,
functions $r_{j,0}\in o_{1}'o_{1}''$,$r_{j,1}\in o_{2}'$, $r_{j,\infty}\in o_{\infty}''$,
$j=0,\ldots,2m$, such that 
\begin{align}
e^{ic_{A}}\left(\left(e^{H_{f_{1}}}\circ e^{H_{f_{0}}}\right)^{*}d\right)e^{-ic_{A}} & =H_{1}+\sigma^{j}r_{j,0}+\sigma^{j}r_{j,1}+\sigma^{j}r_{j,\infty},\label{eq: conjugate symbol-1}
\end{align}
with $H_{1}$ as in \prettyref{eq: model Dirac symbol}.

Also note that the terms $r_{0,0}+r_{0,1}$ maybe assumed to be $\left(x';\xi'\right)$
independent as observed after \prettyref{eq: model Dirac symbol}.
Now set 
\begin{align}
\left(\tilde{\theta}_{0},\tilde{\theta}_{1},\ldots,\tilde{\theta}_{2m}\right) & =\left(\frac{1}{L_{\gamma}}\left(\xi_{0}+\bar{\varphi}\right),\left(2\mu_{1}\right)^{\frac{1}{2}}x_{1},\left(2\mu_{1}\right)^{\frac{1}{2}}\xi_{1},\ldots,\left(2\mu_{m}\right)^{\frac{1}{2}}x_{m},\left(2\mu_{m}\right)^{\frac{1}{2}}\xi_{m}\right)\label{eq: components symbol-1}\\
 & \quad+\left(r_{0,0},r_{1,0},\ldots,r_{2m,0}\right)+\left(r_{0,1},r_{1,1},\ldots,r_{2m,1}\right)+\left(r_{0,\infty},r_{1,\infty},\ldots,r_{2m,\infty}\right)\nonumber \\
\tilde{\theta}' & =\left(\tilde{\theta}_{1},\ldots,\tilde{\theta}_{2m}\right)\nonumber 
\end{align}
 and note from \prettyref{eq: conjugate symbol} that the eigenvalues
of the symbol $d$ are $\pm\left|\tilde{\theta}\right|$. We clearly
have $U_{\varepsilon,\tau}\cap\Sigma_{0}^{D}=\left\{ \tilde{\theta}=0\right\} \cap\Sigma_{0}^{D}$and
we may set 
\begin{align}
\theta_{j} & =\frac{\tilde{\theta}_{j}}{\left|\tilde{\theta}\right|}\in C^{\infty}\left(U_{\varepsilon,\tau}\setminus\Sigma_{0}^{D};S^{n-1}\right).\label{eq:circle valued function-1}
\end{align}
We now compute
\begin{align}
\left\{ \tilde{\theta}_{0},x_{0}\right\} -\frac{1}{L_{\gamma}} & =o_{1}'+o_{1}''+\mathtt{o}_{\infty}''\nonumber \\
\left\{ \tilde{\theta}_{j},x_{0}\right\}  & =o_{1}'+o_{1}''+\mathtt{o}_{\infty}'',\quad j\geq1,\nonumber \\
\left\{ \tilde{\theta}_{j},x''\right\}  & =o_{1}'+\mathtt{o}_{\infty}'',\quad j\geq1,\nonumber \\
\left\{ \tilde{\theta}_{j},\xi''\right\}  & =o_{1}'+\mathtt{o}_{\infty}'',\quad j\geq1,\nonumber \\
\left\{ \tilde{\theta}_{0},\tilde{\theta}_{j}\right\}  & =o_{2}'+o_{1}'o_{1}''+\mathtt{o}_{\infty}'',\quad j\geq0,\nonumber \\
\left\{ \tilde{\theta}_{j},\tilde{\theta}_{k}\right\} \quad\textrm{or }\left\{ \tilde{\theta}_{j},\tilde{\theta}_{k}\right\} -1 & =o_{1}'+o_{1}''+\mathtt{o}_{\infty}''\quad k>j\geq0,\label{eq: orders of commmutators 2}
\end{align}
similar to \prettyref{eq: orders of commutators}. Note that the bracket
$\left\{ \tilde{\theta}_{0},\tilde{\theta}_{j}\right\} $ is still
$o_{2}+\mathtt{o}_{\infty}$ due to the $\left(x';\xi'\right)$-independence
of $r_{0}$ in $\tilde{\theta}_{0}$. In this case however, unlike
\prettyref{eq: orders of commutators} the brackets $\left\{ \tilde{\theta}_{0},x''\right\} ,\left\{ \tilde{\theta}_{0},\xi''\right\} $
may not be $o_{1}'+\mathtt{o}_{\infty}''$ due to the presence of
the $\bar{\varphi}\left(x'',\xi''\right)$ term in $\tilde{\theta}_{0}$
. However the quadratics
\begin{eqnarray}
\hat{Q}_{j}^{e} & = & \left(\xi_{j+m}^{2}+x_{j+m}^{2}\right)^{\frac{\varepsilon}{T}}\nonumber \\
\hat{Q}_{j}^{h} & = & \left(x_{N_{e}+j+m}^{2}+\xi_{N_{e}+j+m}^{2}\right)^{\frac{\varepsilon}{T}}\nonumber \\
\hat{Q}_{j}^{l,\textrm{Re}} & = & \left(x_{2m-2j+1}^{2}+x_{2m-2j+2}^{2}\right)^{\frac{\varepsilon}{T}}\nonumber \\
\hat{Q}_{j}^{l,\textrm{Im}} & = & \left(\xi_{2m-2j+1}^{2}+\xi_{2m-2j+2}^{2}\right)^{\frac{\varepsilon}{T}}\label{eq: escape fn quadratics}
\end{eqnarray}
are seen to satisfy
\begin{align}
\left\{ \tilde{\theta}_{0},\hat{Q}\right\} -\frac{1}{L}\left\{ \bar{\varphi},\hat{Q}\right\}  & =o_{1}'+\mathtt{o}_{\infty}''\label{eq: substitute x'' xi'' estimates 1}\\
\left|\left\{ \bar{\varphi},\hat{Q}\right\} \right| & \leq\frac{\varepsilon}{T}\underbrace{\left(m\sup_{\left(x'',\xi''\right)\leq\varepsilon}\left|\partial_{\bar{Q}}\bar{\varphi}\right|\right)}_{\eqqcolon c_{0}},\quad\hat{Q}\neq0,\label{eq: substitute x'' xi'' estimates 2}
\end{align}
where $\bar{\varphi}$ is considered as a function of the quadratics
$\bar{Q}$ as in \prettyref{eq: effective hamiltonian}. Hence for
$\varepsilon,\tau$ sufficiently small, the bracket relations \prettyref{eq: orders of commmutators 2},
\prettyref{eq: substitute x'' xi'' estimates 1} and \prettyref{eq: substitute x'' xi'' estimates 2}
again imply 
\begin{align}
\left|\left\{ \tilde{\theta}_{j},x_{0}\right\} \right| & \leq2,\qquad j\geq0,\nonumber \\
\left|\left\{ \left|\tilde{\theta}\right|,x_{0}\right\} \right| & \leq2,\qquad\left|\tilde{\theta}\right|\neq0,\nonumber \\
\left|\left\{ \frac{\tilde{\theta}_{j}}{\left|\tilde{\theta}\right|},x_{0}\right\} \right| & \leq\frac{4}{\left|\tilde{\theta}\right|},\qquad\left|\tilde{\theta}\right|\neq0,\,j\geq0,\nonumber \\
\left|\frac{1}{\varepsilon}\left\{ \tilde{\theta}_{j},\hat{Q}\right\} \right| & \leq\frac{2c_{0}}{LT},\qquad j\geq0,\nonumber \\
\left|\frac{1}{\varepsilon}\left\{ \left|\tilde{\theta}\right|,\hat{Q}\right\} \right| & \leq\frac{2c_{0}}{LT},\qquad\left|\tilde{\theta}\right|\neq0,\nonumber \\
\left|\frac{1}{\varepsilon}\left\{ \frac{\tilde{\theta}_{j}}{\left|\tilde{\theta}\right|},\hat{Q}\right\} \right| & \leq\frac{2c_{0}}{LT\left|\tilde{\theta}\right|},\qquad\left|\tilde{\theta}\right|\neq0,\,j\geq0,\nonumber \\
\left|\left\{ \tilde{\theta}_{0},\tilde{\theta}_{j}\right\} \right| & \leq\frac{\left|\tilde{\theta}\right|}{T},\qquad j\geq0,\nonumber \\
\left|\left\{ \tilde{\theta}_{0},\left|\tilde{\theta}\right|\right\} \right| & \leq\frac{\left|\tilde{\theta}\right|}{T},\qquad\left|\tilde{\theta}\right|\neq0,\ \ \nonumber \\
\left|\left\{ \tilde{\theta}_{0},\frac{\tilde{\theta}_{j}}{\left|\tilde{\theta}\right|}\right\} \right| & \leq\frac{1}{T},\qquad\left|\tilde{\theta}\right|\neq0,\,j\geq0,\nonumber \\
\frac{1}{4}\left[\left(\xi_{0}+\bar{\varphi}\right)^{2}+2\sum_{j=1}^{m}\mu_{j}\left(x_{j}^{2}+\xi_{j}^{2}\right)\right] & \leq\sum_{j=0}^{2m}\tilde{\theta}_{j}^{2}\leq4\left[\left(\xi_{0}+\bar{\varphi}\right)^{2}+2\sum_{j=1}^{m}\mu_{j}\left(x_{j}^{2}+\xi_{j}^{2}\right)\right]\label{eq: shrinking tau consequences-1}
\end{align}
on $U_{\varepsilon,\tau}$. Again define
\[
\tilde{U}_{\varepsilon,\tau}\coloneqq\left\{ \sum_{j=0}^{2m}\tilde{\theta}_{j}^{2}<\tau^{2},\:x''^{2}+\xi''^{2}<\varepsilon^{2}\right\} \subset U_{\varepsilon,\tau}.
\]

We now set 
\begin{equation}
\Omega_{\gamma_{v}}\coloneqq\left\{ \hat{Q}_{j}<\left(\frac{\varepsilon}{16m}\right)^{2}\right\} \label{eq: choice subcover}
\end{equation}

To verify \prettyref{eq: second extension time} again let $\chi\in C_{c}^{\infty}\left(\left[-4,4\right];\left[0,1\right]\right)$,
be a cutoff such that $\chi=1$ on $\left[-2,2\right]$ and $\left|\chi'\right|\leq1$.
Also for $\rho\in\left(0,\frac{1}{8}\right)$ fixed, define a function
$\varphi_{\rho}\in C^{\infty}\left(\left[-1,1\right]_{\theta_{0}};\left[0,1\right]\right)$
such that $\varphi_{\rho}\left(\theta_{0}\right)=\begin{cases}
1; & \textrm{for }\theta_{0}\in\left[1-\rho,1\right]\\
0; & \textrm{for }\theta_{0}\in\left[-1,1-2\rho\right]
\end{cases}$ and $\left|\varphi_{\rho}'\right|\leq\frac{2}{\rho}$ . The trapping
function in this case is now modified to 
\begin{align*}
g_{v}\coloneqq & \chi\left(\frac{\beta\left(\tilde{\theta}\right)}{\tau}\right)\left[\prod_{j}\chi\left(\frac{\hat{Q}_{j}}{\left(\varepsilon/16m\right)^{2}}\right)\right]\in C_{c}^{\infty}\left(\Sigma_{\left[-8\tau,8\tau\right]}^{D}\cap\tilde{\varOmega}_{\gamma_{v}}\right)\quad\textrm{ where}\\
\beta\left(\tilde{\theta}\right)\coloneqq & \sqrt{\left|\tilde{\theta}\right|^{2}-\varphi_{\rho}\left(\theta_{0}\right)\left|\tilde{\theta}'\right|^{2}}\\
= & \sqrt{\left|\tilde{\theta}_{0}\right|^{2}+\left(1-\varphi_{\rho}\right)\left|\tilde{\theta}'\right|^{2}}\quad\textrm{ satisfying}\\
\frac{\left|\tilde{\theta}\right|}{2}\leq & \beta\left(\tilde{\theta}\right)\leq\left|\tilde{\theta}\right|
\end{align*}
as before in terms of the relevant coordinates on $\tilde{U}_{\varepsilon,\tau}$.
With $\mathtt{v}_{u}$ now defined in a similar fashion to \prettyref{eq: actual conjugate symbol},
one may again estimate $\left|H_{g_{u},\mathtt{v}_{u}}\left(d\right)\right|=O\left(\frac{1}{T}\right)$
as in \prettyref{eq: comm comp 1}-\prettyref{eq: comm comp 6} using
\prettyref{eq: substitute x'' xi'' estimates 1} and \prettyref{eq: shrinking tau consequences-1}
to complete the proof.
\end{proof}
Next; we have a lemma reducing the trace asymptotics to $S^{1}\times\mathbb{R}^{2m}$.
First choose $T$ sufficiently large such that $\textrm{spt \ensuremath{\left(\theta\right)}}\subset\left[-T,T\right]$.
Then choose $\tau$ sufficiently small and an open sub-cover $\Gamma\subset\Omega\subset\varOmega$
with $T_{\left(\Omega,\varOmega\right)}^{\tau}>T$. Finally and as
observed before, by choosing $\tau$ even smaller if necessary, one
may also find an $\left(\Omega,\tau,\delta\right)$ partition to arrange
$T_{\left(\mathcal{P};\mathcal{V},\mathcal{W}\right)}>T$; reducing
us to study of the asymptotics of $\mathcal{T}_{B_{v},B_{v}}^{\theta}\left(D\right)$.
We now show that \prettyref{eq: second extension time} allows a further
reduction to $S^{1}\times\mathbb{R}^{2m}$. Below, the operator $D_{0}$
is as in \prettyref{eq:Dirac op weyl quantization without error}.
\begin{prop}
\label{prop:Reduction to S1 X R^2m}For each $1\leq v\leq M$, one
has 
\[
\mathcal{T}_{B_{v},B_{v}}^{\theta}\left(D\right)=\underbrace{\textrm{tr}\left[B_{v}^{0}f\left(\frac{D_{0}}{\sqrt{h}}\right)\check{\theta}\left(\frac{\lambda\sqrt{h}-D_{0}}{h}\right)B_{v}^{0}\right]}_{\coloneqq\mathcal{T}_{B_{v}^{0},B_{v}^{0}}^{\theta}\left(D_{0}\right)}\quad\textrm{mod }h^{\infty}
\]
for cutoffs $B_{v}^{0}\in\Psi_{\delta}^{0}\left(S^{1}\times\mathbb{R}^{2m}\right)$,
with $WF\left(B_{v}^{0}\right)\Subset\Sigma_{\left[-\tau_{\delta},\tau_{\delta}\right]}^{D_{0}}\cap\tilde{\Omega}_{\gamma_{v}}^{\delta}$
.
\end{prop}
\begin{proof}
The proof is again similar to \cite{Savale2017-Koszul} Prop. 4.1,
provided the smallness of $\textrm{spt \ensuremath{\left(\theta\right)}}$
is quantified. First one has an analog of \prettyref{lem:changing symbol near crit set}:
for $D'\in\Psi_{\textrm{cl}}^{1}\left(X;S\right)$ essentially self-adjoint,
with $D=D'$ microlocally on $\Sigma_{\left[-8\tau,8\tau\right]}^{D}\cap\tilde{\varOmega}_{\gamma_{v}}$
, and $\theta\in C_{c}^{\infty}\left(\left(T'h^{\epsilon},T_{v}\right);\left[0,1\right]\right)$
one has 
\[
\mathcal{T}_{B_{v},B_{v}}^{\theta}\left(D\right)=\mathcal{T}_{B_{v},B_{v}}^{\theta}\left(D'\right)\quad\textrm{mod }h^{\infty}
\]
(since $B_{v}$ has microsupport in $\Sigma_{\left[-\tau_{\delta},\tau_{\delta}\right]}^{D_{1}}\cap\tilde{\Omega}_{\gamma_{v}}^{\delta}$
and hence on $\Sigma_{\left[-\tau,\tau\right]}^{D_{1}}\cap\tilde{\Omega}_{\gamma_{v}}$).
Now as $D=D_{0}$ on $\varOmega_{\gamma_{v}}$ by construction \prettyref{eq:Dirac op weyl quantization without error}
and hence microlocally on $\Sigma_{\left[-8\tau,8\tau\right]}^{D}\cap\tilde{\varOmega}_{\gamma_{v}}$;
the proof in \cite{Savale2017-Koszul} is seen to carry through provided
$\textrm{spt}\left(\theta\right)$ is contained in each of $\left\{ \left(T'h^{\epsilon},T_{v}\right)\right\} $
, $1\leq v\leq M$. But this is guaranteed by our choice of an appropriate
subcover $\Gamma\subset\Omega\subset\varOmega$ satisfying \prettyref{eq: second extension time}and
$\textrm{spt \ensuremath{\left(\theta\right)}}\subset\left[-T,T\right]$.
\end{proof}
Next, we show how the Birkhoff normal form maybe used to perform a
further reduction on the trace. First note that we may similarly use
\prettyref{eq: clifford quantization} to define a self-adjoint Clifford-Weyl
quantization map 
\[
c_{0}^{W}\coloneqq\textrm{Op}\otimes c_{0}:S_{\textrm{cl}}^{0}\left(T^{*}S^{1}\times\mathbb{R}^{4m};\mathbb{C}\right)\otimes\Lambda^{\textrm{odd/even}}W\rightarrow\Psi_{\textrm{cl}}^{0}\left(S^{1}\times\mathbb{R}^{2m};\mathbb{C}^{2^{m}}\right)
\]
which maps real valued symbols $S_{\textrm{cl}}^{0}\left(T^{*}S^{1}\times\mathbb{R}^{4m};\mathbb{R}\right)\otimes\Lambda^{\textrm{odd/even}}W$
to self-adjoint operators in $\Psi_{\textrm{cl}}^{0}\left(S^{1}\times\mathbb{R}^{2m};\mathbb{C}^{2^{m}}\right)$.
Similarly we define a space of real-valued, twisted $\tilde{\Delta}^{0}$-harmonic,
$\bar{\varphi}$-commuting, $x_{0}$- independent symbols 
\[
\mathcal{H}^{k}S_{\textrm{cl}}^{0}\coloneqq\left\{ \omega\in S_{\textrm{cl}}^{0}\left(T^{*}S^{1}\times\mathbb{R}^{4m};\mathbb{R}\right)\otimes\Lambda^{k}W|\,\tilde{\Delta}^{0}\omega=0,\,H_{\bar{\varphi}}\omega=0,\,\partial_{x_{0}}\omega=0\right\} .
\]
Next, an application of Borel's lemma by virtue of \prettyref{eq:first conjugation normal form},
\prettyref{eq: d0 =000026 d1 are conjugate} and \prettyref{eq: normal form d1}
gives the existence of
\begin{align*}
\bar{a}\sim & \sum_{j=0}^{\infty}h^{j}\bar{a}_{j}\in S_{\textrm{cl}}^{0}\left(T^{*}S^{1}\times\mathbb{R}^{4m};\mathbb{R}\right)\otimes\Lambda^{\textrm{odd}}W\\
\bar{r}\sim & \sum_{j=0}^{\infty}h^{j}\bar{r}_{j}\in S_{\textrm{cl}}^{0}\left(T^{*}S^{1}\times\mathbb{R}^{4m};\mathbb{R}\right)\otimes\Lambda^{\textrm{odd}}W\\
\bar{f}\sim & \sum_{j=0}^{\infty}h^{j}\bar{f}_{j}\in S_{\textrm{cl}}^{0}\left(T^{*}S^{1}\times\mathbb{R}^{4m};\mathbb{R}\right)\\
\bar{\omega}\sim & \sum_{j=0}^{\infty}h^{j}\bar{\omega}_{j}\in\mathcal{H}^{\textrm{odd}}S_{\textrm{cl}}^{0}
\end{align*}
such that 
\begin{equation}
e^{ic_{0}^{W}\left(\bar{a}\right)}e^{\frac{i}{h}\bar{f}^{W}}d_{0}^{W}e^{-\frac{i}{h}\bar{f}^{W}}e^{-ic_{0}^{W}\left(\bar{a}\right)}=\underbrace{H_{1}^{W}+c_{0}^{W}\left(\bar{\omega}\right)}_{\coloneqq\bar{D}}+c_{0}^{W}\left(\bar{r}\right)\label{eq:normal form conjugation}
\end{equation}
on $S^{1}\times\mathbb{R}^{2m}$. Here $\left\{ \bar{r}_{j}\right\} _{j\in\mathbb{N}_{0}},$$\bar{f}_{0}$,
$\bar{\omega}_{0}$ vanish to infinite, second and second order respectively
along 
\[
\varGamma=\left\{ \xi_{0}+\bar{\varphi}=x'=\xi'=x''=\xi''=0\right\} .
\]
Moreover $\bar{f}_{0}$, $\bar{\omega}_{0}$ vanish to first order
along 
\[
\varGamma'=\left\{ \xi_{0}+\bar{\varphi}=x'=\xi'\right\} .
\]
We may hence choose $\bar{\omega}_{0}$ having an expansion 
\begin{equation}
\bar{\omega}_{0}=\left(\xi_{0}+\bar{\varphi}\right)\omega_{00}+\sum_{j=1}^{m}\left(\omega_{0j}z_{j}+\bar{\omega}_{0j}\bar{z}_{j}\right)\label{eq: omega0 expansion}
\end{equation}
 in terms of the complex coordinates $z_{j}=x'_{j}+iy_{j}'$ with
\[
\left\Vert \bar{\omega}_{0j}\right\Vert _{C^{0}}\leq\varepsilon
\]
arbitrarily small. 

Next we show that one may pass from the trace asymptotics of $D_{0}$
to $\bar{D}$\prettyref{eq:Dirac op weyl quantization without error}.
Below we set $\bar{B}_{v}=e^{ic_{0}^{W}\left(\bar{a}\right)}e^{\frac{i}{h}\bar{f}^{W}}B_{v}^{0}e^{-\frac{i}{h}\bar{f}^{W}}e^{-ic_{0}^{W}\left(\bar{a}\right)}$.
Note that $\bar{B}_{v}=1$ on an $h^{\delta}$ size neighborhood of
$\varGamma$ by construction.
\begin{prop}
\label{prop: last reduction of trace}For each $1\leq v\leq M$, we
have 
\[
\mathcal{T}_{B_{v}^{0},B_{v}^{0}}^{\theta}\left(D_{0}\right)=\mathcal{T}_{\bar{B}_{v},\bar{B}_{v}}^{\theta}\left(\bar{D}\right)\quad\textrm{mod }h^{\infty}.
\]
\end{prop}
\begin{proof}
By choosing an appropriately small $\Omega$ in terms of Reeb Darboux
coordinates as in \prettyref{eq: choice subcover}, we may find a
cutoff of the form $A=\chi\left(\frac{\bar{D}^{2}+\left(x''^{2}+\xi''^{2}\right)^{W}}{h^{2\delta}}\right)$,
$\chi\in C_{c}^{\infty}\left(\mathbb{R}\right)$, that is microlocally
$1$ on $WF\left(\bar{B}_{v}\right)$. We then have by the Helffer-Sjöstrand
formula
\begin{equation}
\mathcal{T}_{B_{v}^{0},B_{v}^{0}}^{\theta}\left(D_{0}\right)-\mathcal{T}_{\bar{B}_{v},\bar{B}_{v}}^{\theta}\left(\bar{D}\right)=\frac{1}{\pi}\int_{\mathbb{C}}\bar{\partial}\tilde{f}\left(z\right)\check{\theta}\left(\frac{\lambda-z}{\sqrt{h}}\right)\textrm{tr }\left[\bar{B}_{v}\varDelta_{z}A\bar{B}_{v}\right]dzd\bar{z}\quad\textrm{mod }h^{\infty},\label{eq: difference traces}
\end{equation}
with 
\[
\varDelta_{z}=\left(\frac{1}{\sqrt{h}}\left(\bar{D}+c_{0}^{W}\left(\bar{r}\right)\right)-z\right)^{-1}c_{0}^{W}\left(\bar{r}\right)\left(\frac{1}{\sqrt{h}}\left(\bar{D}\right)-z\right)^{-1}.
\]
Since $\bar{r}$ vanishes to infinite order along $\varGamma$, symbolic
calculus gives 
\[
c_{0}^{W}\left(\bar{r}\right)=R_{N}\left[\bar{D}^{N}+\left(\bar{\varphi}^{W}\right)^{N}\right]\quad\forall N,
\]
for some $R_{N}\in\Psi_{\textrm{cl}}^{0}\left(S^{1}\times\mathbb{R}^{2m};\mathbb{C}^{2^{m}}\right)$.
From which the commutation $\left[\bar{D},\bar{\varphi}^{W}\right]=0$
gives 
\[
\varDelta_{z}=\left(\frac{1}{\sqrt{h}}\left(\bar{D}+c_{0}^{W}\left(\bar{r}\right)\right)-z\right)^{-1}S_{N}\left(\frac{1}{\sqrt{h}}\left(\bar{D}\right)-z\right)^{-1}\left[\bar{D}^{2}+\left(x''^{2}+\xi''^{2}\right)^{W}\right]^{N}\quad\forall N,
\]
for some $S_{N}\in\Psi_{\textrm{cl}}^{0}\left(S^{1}\times\mathbb{R}^{2m};\mathbb{C}^{2^{m}}\right)$.
Now 
\[
\varDelta_{z}A=\left(\frac{1}{\sqrt{h}}\left(\bar{D}+c_{0}^{W}\left(\bar{r}\right)\right)-z\right)^{-1}S_{N}\left(\frac{1}{\sqrt{h}}\left(\bar{D}\right)-z\right)^{-1}h^{2\delta N}\chi_{N}\left(\frac{\bar{D}^{2}+\left(x''^{2}+\xi''^{2}\right)^{W}}{h^{2\delta}}\right)\quad\forall N,
\]
for $\chi_{N}\left(x\right)=x^{N}\chi\left(x\right)\in C_{c}^{\infty}\left(\mathbb{R}\right)$.
Plugging this last equation into \prettyref{eq: difference traces}
gives the result.
\end{proof}

\section{\label{sec: Trace Formula}Trace Asymptotics}

In this section we finish the proof of \prettyref{lem: O(h infty) LEMMA}
and hence \prettyref{thm:main trace expansion-1}. By the reductions
\prettyref{prop:Reduction to S1 X R^2m} and \prettyref{prop: last reduction of trace}
of the last section it suffices to consider the trace $\mathcal{T}_{\bar{B}_{v},\bar{B}_{v}}^{\theta}\left(\bar{D}\right)$. 
\begin{proof}[Proof of \prettyref{lem: O(h infty) LEMMA}]
 We begin with the orthogonal Landau decomposition \prettyref{eq: Landau Levels}
\begin{align}
L^{2}\left(S^{1}\times\mathbb{R}^{2m};\mathbb{C}^{2^{m}}\right) & =L^{2}\left(S_{x_{0}}^{1}\times\mathbb{R}_{x''}^{m}\right)\otimes\underbrace{\left(\mathbb{C}\left[\psi_{0,0}\right]\oplus\bigoplus_{\varLambda\in\mu.\left(\mathbb{N}_{0}^{m}\setminus0\right)}\left[E_{\varLambda}^{\textrm{even}}\oplus E_{\varLambda}^{\textrm{odd}}\right]\right)}_{=L^{2}\left(\mathbb{R}_{x'}^{m};\mathbb{C}^{2^{m}}\right)}\;\textrm{where}\label{eq: Landau Decomposition}\\
E_{\varLambda}^{\textrm{even}} & \coloneqq\bigoplus_{\begin{subarray}{l}
\tau\in\mathbb{N}_{0}^{m}\setminus0\\
\varLambda=\mu.\tau
\end{subarray}}E_{\tau}^{\textrm{even}}\nonumber \\
E_{\varLambda}^{\textrm{odd}} & \coloneqq\bigoplus_{\begin{subarray}{l}
\tau\in\mathbb{N}_{0}^{m}\setminus0\\
\varLambda=\mu.\tau
\end{subarray}}E_{\tau}^{\textrm{odd}}\nonumber 
\end{align}
according to the eigenspaces of the squared magnetic Dirac operator
$D_{\mathbb{R}^{m}}^{2}$ \prettyref{eq: magnetic Dirac Rm} on $\mathbb{R}^{m}$.
It is clear from \prettyref{eq: model Dirac symbol} that 
\[
H_{1}^{W}=\frac{1}{L_{\gamma}}\left(\xi_{0}+\bar{\varphi}\right)^{W}\sigma_{0}+D_{\mathbb{R}^{m}}
\]
in terms of the above decomposition. Furthermore one has the commutation
relations 
\begin{eqnarray*}
\left[\sigma_{0},D_{\mathbb{R}^{m}}^{2}\right] & = & 0\\
\left[c_{0}^{W}\left(\bar{\omega}\right),D_{\mathbb{R}^{m}}^{2}\right] & = & ihc_{0}^{W}\left(\tilde{\Delta}^{0}\bar{\omega}\right)=0
\end{eqnarray*}
since $\bar{\omega}$ in \prettyref{eq:normal form conjugation} is
$\tilde{\Delta}^{0}$-harmonic. Hence $\bar{D}$ preserves the decomposition
\prettyref{eq: Landau Decomposition} and we may consider the restriction
of its traces to the eigenspaces of $D_{\mathbb{R}^{m}}^{2}$. Namely,
let $E_{0}\coloneqq\mathbb{C}\left[\psi_{0,0}\right],\;E_{\varLambda}\coloneqq E_{\varLambda}^{\textrm{even}}\oplus E_{\varLambda}^{\textrm{odd}},\;E_{>0}\coloneqq\bigoplus_{\varLambda\in\mu.\left(\mathbb{N}_{0}^{m}\setminus0\right)}E_{\varLambda}$
and $\mathtt{P}_{0},\:\mathtt{P}_{\varLambda},\:\mathtt{P}_{>0}\coloneqq\bigoplus_{\varLambda\in\mu.\left(\mathbb{N}_{0}^{m}\setminus0\right)}\mathtt{P}_{\varLambda}$
denote the summands and the corresponding projections of \prettyref{eq: Landau Decomposition}.
It is then clear that $\mathcal{T}_{\bar{B}_{v},\bar{B}_{v}}^{\theta}\left(\bar{D}\right)=\mathcal{T}_{\bar{B}_{v},\bar{B}_{v}}^{\theta}\left(\mathtt{P}_{0}\bar{D}\mathtt{P}_{0}\right)+\mathcal{T}_{\bar{B}_{v},\bar{B}_{v}}^{\theta}\left(\mathtt{P}_{>0}\bar{D}\mathtt{P}_{>0}\right)$.

Set 
\begin{eqnarray*}
\bar{D}_{0}\coloneqq\mathtt{P}_{0}\bar{D}\mathtt{P}_{0} & : & L^{2}\left(S_{x_{0}}^{1}\times\mathbb{R}_{x''}^{m}\right)\rightarrow L^{2}\left(S_{x_{0}}^{1}\times\mathbb{R}_{x''}^{m}\right)\\
\bar{D}_{\varLambda}\coloneqq\mathtt{P}_{\varLambda}\bar{D}\mathtt{P}_{\varLambda} & : & L^{2}\left(S_{x_{0}}^{1}\times\mathbb{R}_{x''}^{m};E_{\varLambda}^{\textrm{even}}\oplus E_{\varLambda}^{\textrm{odd}}\right)\rightarrow L^{2}\left(S_{x_{0}}^{1}\times\mathbb{R}_{x''}^{m};E_{\varLambda}^{\textrm{even}}\oplus E_{\varLambda}^{\textrm{odd}}\right),\;\varLambda>0.
\end{eqnarray*}
 The restrictions of the $c_{0}^{W}\left(\bar{\omega}\right)$ term
in $\bar{D}$ are
\begin{eqnarray*}
\varOmega_{0}\coloneqq\mathtt{P}_{0}c_{0}^{W}\left(\bar{\omega}\right)\mathtt{P}_{0} & : & L^{2}\left(S_{x_{0}}^{1}\times\mathbb{R}_{x''}^{m}\right)\rightarrow L^{2}\left(S_{x_{0}}^{1}\times\mathbb{R}_{x''}^{m}\right)\\
\varOmega_{\varLambda}\coloneqq\mathtt{P}_{\varLambda}c_{0}^{W}\left(\bar{\omega}\right)\mathtt{P}_{\varLambda} & : & L^{2}\left(S_{x_{0}}^{1}\times\mathbb{R}_{x''}^{m};E_{\varLambda}^{\textrm{even}}\oplus E_{\varLambda}^{\textrm{odd}}\right)\rightarrow L^{2}\left(S_{x_{0}}^{1}\times\mathbb{R}_{x''}^{m};E_{\varLambda}^{\textrm{even}}\oplus E_{\varLambda}^{\textrm{odd}}\right),\;\varLambda>0.
\end{eqnarray*}
The operator $\varOmega_{0}=\alpha_{0}^{W}\in\Psi_{\textrm{cl}}^{0}\left(S_{x_{0}}^{1}\times\mathbb{R}_{x''}^{m}\right)$
is pseudo-differential operator with symbol vanishing to second order
along $\varGamma''=\left\{ \xi_{0}+\bar{\varphi}=x''=\xi''=0\right\} $.
Also, quantizing the expansion \prettyref{eq: omega0 expansion} gives
\[
c_{0}^{W}\left(\bar{\omega}\right)=\left(\xi_{0}+\bar{\varphi}\right)^{W}\underbrace{c_{0}^{W}\left(\omega_{00}\right)}_{=O_{L^{2}\rightarrow L^{2}}\left(\varepsilon\right)}+\sum_{j=1}^{m}\left[\underbrace{c_{0}^{W}\left(\omega_{0j}\right)}_{=O_{L^{2}\rightarrow L^{2}}\left(\varepsilon\right)}A_{j}+A_{j}^{*}\underbrace{c_{0}^{W}\left(\bar{\omega}_{0j}\right)}_{=O_{L^{2}\rightarrow L^{2}}\left(\varepsilon\right)}\right]+O\left(h\right)
\]
Knowing the action of the lowering and raising operators $A_{j}$,
$A_{j}^{*}$ on each eigenstate \prettyref{eq: Hermite functions}
of $D_{\mathbb{R}^{m}}^{2}$ then gives the estimate 
\begin{equation}
\varOmega_{\varLambda}=\left(\xi_{0}+\bar{\varphi}\right)^{W}O_{L^{2}\rightarrow L^{2}}\left(\varepsilon\right)+O_{L^{2}\rightarrow L^{2}}\left(\varepsilon\sqrt{\varLambda h}\right)+O_{L^{2}\rightarrow L^{2}}\left(h\right)\label{eq: bound Omega N}
\end{equation}
with all constants above being uniform in $\varLambda>0$.

Next, we consider $\mathcal{T}_{\bar{B}_{v},\bar{B}_{v}}^{\theta}\left(\mathtt{P}_{>0}\bar{D}\mathtt{P}_{>0}\right)$
by computing the restriction of $\left(\frac{1}{\sqrt{h}}\bar{D}-z\right)$,
$z\in\mathbb{C}$, to each $E_{\varLambda}$, $\varLambda>0$, eigenspace
in \prettyref{eq: Landau Decomposition}. Using \prettyref{eq: Dirac operator 2 by 2 block}
this has the form 
\begin{eqnarray*}
\bar{D}_{\varLambda}\left(z\right) & \coloneqq & \mathtt{P}_{\varLambda}\left(\frac{1}{\sqrt{h}}\bar{D}-z\right)\mathtt{P}_{\varLambda}\\
 & = & \frac{1}{\sqrt{h}}\begin{bmatrix}-\left(\xi_{0}+\bar{\varphi}\right)-z\sqrt{h} & \left(\sqrt{2\varLambda h}\right)^{W}\\
\left(\sqrt{2\varLambda h}\right)^{W} & \xi_{0}+\bar{\varphi}-z\sqrt{h}
\end{bmatrix}+\frac{1}{\sqrt{h}}\varOmega_{\varLambda}
\end{eqnarray*}
with respect to the $\mathbb{Z}_{2}$- grading $E_{\varLambda}=E_{\varLambda}^{\textrm{even}}\oplus E_{\varLambda}^{\textrm{odd}}$.
Here we leave the identification $\mathtt{i}_{\tau}$ in \prettyref{eq: Dirac operator 2 by 2 block}
between the odd and even parts as being understood. Let $\varepsilon_{0}>0$
be such that $f\in C_{c}^{\infty}\left(-\sqrt{2\mu_{1}}+\varepsilon_{0},\sqrt{2\mu_{1}}-\varepsilon_{0}\right)$.
Set $R_{\varLambda}\left(z\right)=\left[r_{\varLambda}\left(z\right)\right]^{W}$
\[
r_{\varLambda}\left(z\right)\coloneqq\frac{\sqrt{h}\begin{bmatrix}-\left(\xi_{0}+\bar{\varphi}\right)-z\sqrt{h} & \left(\sqrt{2\varLambda h}\right)\\
\left(\sqrt{2\varLambda h}\right) & \xi_{0}+\bar{\varphi}-z\sqrt{h}
\end{bmatrix}}{z^{2}h-\left(\xi_{0}+\bar{\varphi}\right)^{2}-2\varLambda h}
\]
which is well defined for $\left|\textrm{Re}z\right|\leq\sqrt{2\mu_{1}}-\varepsilon_{0}<\inf_{\mathbb{R}^{n}}\sqrt{2\varLambda}$,
and $h$ sufficiently small. We now compute 
\begin{eqnarray*}
\left\Vert R_{\varLambda}\left(z\right)\bar{D}_{\varLambda}\left(z\right)-I\right\Vert  & \leq & C\varepsilon+O\left(h\right)\\
\left\Vert \bar{D}_{\varLambda}\left(z\right)R_{\varLambda}\left(z\right)-I\right\Vert  & \leq & C\varepsilon+O\left(h\right)
\end{eqnarray*}
using \prettyref{eq: bound Omega N} with the constants above being
uniform in $\varLambda$. Choosing $\varepsilon$ sufficiently small
in \prettyref{eq: bound Omega N} shows that the inverse $\bar{D}_{\varLambda}\left(z\right)^{-1}$
exists and is $O\left(R_{\varLambda}\left(z\right)\right)=O\left(h^{-\frac{1}{2}}\right)$
uniformly. Thus the resolvent $\left(\mathtt{P}_{>0}\bar{D}\mathtt{P}_{>0}-z\right)^{-1}$
extends holomorphically to the strip $\left|\textrm{Re}z\right|\leq\sqrt{2\mu_{1}}-\varepsilon_{0}$
and picking an almost analytic continuation for $f$ in the Helffer-Sjöstrand
formula \prettyref{eq: four traces} supported in this strip gives
$\mathcal{T}_{\bar{B}_{v},\bar{B}_{v}}^{\theta}\left(\mathtt{P}_{>0}\bar{D}\mathtt{P}_{>0}\right)=0$. 

We now consider $\mathcal{T}_{\bar{B}_{v},\bar{B}_{v}}^{\theta}\left(\mathtt{P}_{0}\bar{D}\mathtt{P}_{0}\right)$.
The cutoffs maybe taken to be of the form $\bar{B}_{v}=\chi\left(\frac{\left(x''^{2}+\xi''^{2}\right)^{W}}{h^{2\delta}}\right)\chi\left(\frac{\mathtt{H}_{2}+\left(\left(\xi_{0}+\bar{\varphi}\right)^{W}\right)^{2}}{h^{2\delta}}\right)$,
with $\mathtt{H}_{2}$ being the harmonic oscillator \prettyref{eq:Harmonic oscillator},
to compute 
\begin{align}
\mathcal{T}_{\bar{B}_{v},\bar{B}_{v}}^{\theta}\left(\mathtt{P}_{0}\bar{D}\mathtt{P}_{0}\right) & =\frac{1}{\pi}\int_{\mathbb{C}}\bar{\partial}\tilde{f}\left(z\right)\check{\theta}\left(\frac{\lambda-z}{\sqrt{h}}\right)\textrm{tr }\left[\bar{B}_{v}^{0}\left(\frac{1}{\sqrt{h}}\bar{D}_{0}-z\right)^{-1}\bar{B}_{v}^{0}\right]dzd\bar{z}\label{eq: reduction scalar trace}
\end{align}
where $\bar{B}_{v}^{0}=\chi\left(\frac{\left(x''^{2}+\xi''^{2}\right)^{W}}{h^{2\delta}}\right)\chi\left(\frac{\left(\left(\xi_{0}+\bar{\varphi}\right)^{W}\right)^{2}}{h^{2\delta}}\right)$
and
\[
\bar{D}_{0}=-\frac{1}{L_{\gamma}}\left(\xi_{0}+\bar{\varphi}\right)^{W}+\alpha_{0}^{W}
\]
being the effective Hamiltonian. The above being a scalar operator,
\prettyref{eq: reduction scalar trace} now reduces to the usual trace
formula microlocalized near the Hamiltonian trajectory $\varGamma''=\left\{ \xi_{0}+\bar{\varphi}=x''=\xi''=0\right\} $
of $\frac{1}{L_{\gamma}}\left(\xi_{0}+\bar{\varphi}\right)$. The
formula \prettyref{eq: non-local dynamical trace} now follows on
identifying the period, symplectic action and return map of this trajectory
to be $L_{\gamma}$, $T_{\gamma}$ and $P_{\gamma}^{+}$ respectively
(cf. \cite{DeGosson-90-maslovindex,degosson-book} Ch 7. for an identification
of the Maslov index in terms of the metaplectic group). 
\end{proof}

\section{\label{sec:Local trace expansion second term}Local trace expansion:
computation of the second coefficient}

In this section we study the trace expansion of a function of the
operator $\frac{D}{\sqrt{h}}$. We first recall the following which
appears as Proposition 7.1 of \cite{Savale2017-Koszul}. 
\begin{prop}
\label{prop: local trace expansion}There exist tempered distributions
$u_{j}\in\mathcal{S}'\left(\mathbb{R}_{s}\right)$, $j=0,1,2,\ldots$,
such that one has a trace expansion
\begin{equation}
\textrm{tr }\phi\left(\frac{D}{\sqrt{h}}\right)=h^{-n/2}\left(\sum_{j=0}^{N}u_{j}\left(\phi\right)h^{j/2}\right)+h^{\left(N+1-n\right)/2}O\left(\sum_{k=0}^{n+1}\left\Vert \left\langle \xi\right\rangle ^{N}\hat{\phi}^{\left(k\right)}\right\Vert _{L^{1}}\right)\label{eq: local trace expansion}
\end{equation}
for each $N\in\mathbb{N}$, $\phi\in\mathcal{S}\left(\mathbb{R}_{s}\right)$. 
\end{prop}
The coefficient $u_{0}$ in \prettyref{eq: local trace expansion}
was computed in Proposition 7.4 of \cite{Savale2017-Koszul}. Our
main task here is the computation of the next coefficient $u_{1}$.
The calculation here is similar to that of the second coefficient
of the symplectic Bergman kernel (cf. \cite{Ma-Marinescu} Ch. 8)
using the local index theory method. 

To this end we first briefly recall the construction of the distributions
$u_{j}.$ Fixing the point $p\in X$ there is an orthonormal basis
$e_{0,p}=\frac{R}{\left|R\right|}$,$\left\{ e_{j,p},\,e_{j+m,p}\right\} _{j=1}^{m}\in R^{\perp}$,
of the tangent space at $p$ consisting of eigenvectors of $\mathfrak{J}_{p}$
with eigenvalues $0,\pm i\mu_{j}$, $j=1,\ldots,m$, such that 
\begin{equation}
da\left(p\right)=\sum_{j=1}^{m}\mu_{j}e_{j,p}^{*}\wedge e_{j+m,p}^{*}.\label{eq: da diagonal form}
\end{equation}
Using the parallel transport from this basis fix a geodesic coordinate
system $\left(x_{0},\ldots,x_{2m}\right)$ on an open neighborhood
of $p\in\Omega$. Let $e_{j}=w_{j}^{k}\partial_{x_{k}}$, $0\leq j\leq2m$,
be the local orthonormal frame of $TX$ obtained by parallel transport
of $e_{j,p}=\left.\partial_{x_{j}}\right|_{p}$,$0\leq j\leq2m$,
along geodesics. Hence we again have $w_{j}^{k}g_{kl}w_{r}^{l}=\delta_{jr}$;
$\left.w_{j}^{k}\right|_{p}=\delta_{j}^{k}$ with $g_{kl}$ being
the components of the metric in these coordinates. Choose an orthonormal
basis $\left\{ s_{j,p}\right\} _{j=1}^{2^{m}}$for $S_{p}$ in which
Clifford multiplication
\begin{equation}
\left.c\left(e_{j}\right)\right|_{p}=\gamma_{j}\label{eq: Clifford multiplication standard}
\end{equation}
is standard. Choose an orthonormal basis $\mathtt{l}_{p}$ for $L_{p}$.
Parallel transport the bases $\left\{ s_{j,p}\right\} _{j=1}^{2^{m}}$,
$\mathtt{l}_{p}$ along geodesics using the spin connection $\nabla^{S}$
and unitary family of connections $\nabla^{h}=A_{0}+\frac{i}{h}a$
to obtain trivializations $\left\{ s_{j}\right\} _{j=1}^{2^{m}}$,
$\mathtt{l}$ of $S$, $L$ on $\Omega$. Since Clifford multiplication
is parallel, the relation \prettyref{eq: Clifford multiplication standard}
now holds on $\Omega$. The connection $\nabla^{S\otimes L}=\nabla^{S}\otimes1+1\otimes\nabla^{h}$
can be expressed in this frame and these coordinates as
\begin{equation}
\nabla^{S\otimes L}=d+A_{j}^{h}dx^{j}+\Gamma_{j}dx^{j},
\end{equation}
 where each $A_{j}^{h}$ is a Christoffel symbol of $\nabla^{h}$
and each $\Gamma_{j}$ is a Christoffel symbol of the spin connection
$\nabla^{S}$. Since the section $\mathtt{l}$ is obtained via parallel
transport along geodesics, the connection coefficient $A_{j}^{h}$
maybe written in terms of the curvature $F_{jk}^{h}dx^{j}\wedge dx^{k}$
of $\nabla^{h}$ via
\begin{equation}
A_{j}^{h}(x)=\int_{0}^{1}d\rho\left(\rho x^{k}F_{jk}^{h}\left(\rho x\right)\right).
\end{equation}
The dependence of the curvature coefficients $F_{jk}^{h}$ on the
parameter $h$ is seen to be linear in $\frac{1}{h}$ via 
\begin{equation}
F_{jk}^{h}=F_{jk}^{0}+\frac{i}{h}\left(da\right){}_{jk}\label{eq:curvature linear in 1/h}
\end{equation}
 despite the fact that they are expressed in the $h$ dependent frame
$\mathtt{l}$. This is because a gauge transformation from an $h$
independent frame into $\mathtt{l}$ changes the curvature coefficient
by conjugation. Since $L$ is a line bundle this is conjugation by
a function and hence does not change the coefficient. Furthermore,
the coefficients in the Taylor expansion of \prettyref{eq:curvature linear in 1/h}
at $0$ maybe expressed in terms of the covariant derivatives $\left(\nabla^{A_{0}}\right)^{l}F_{jk}^{0},$
$\left(\nabla^{A_{0}}\right)^{l}\left(da\right){}_{jk}$ evaluated
at $p$. Next, using the Taylor expansion
\begin{equation}
\left(da\right){}_{jk}=\left(da\right){}_{jk}\left(0\right)+x^{l}a_{jkl},\label{eq: Taylor expansion da}
\end{equation}
we see that the connection $\nabla^{S\otimes L}$ has the form
\begin{equation}
\nabla^{S\otimes L}=d+\left[\frac{i}{h}\left(\frac{x^{k}}{2}\left(da\right){}_{jk}\left(0\right)+x^{k}x^{l}A_{jkl}\right)+x^{k}A_{jk}^{0}+\Gamma_{j}\right]dx^{j}\label{eq: connection in geodesic}
\end{equation}
where 
\begin{eqnarray*}
A_{jk}^{0} & = & \int_{0}^{1}d\rho\left(\rho F_{jk}^{0}\left(\rho x\right)\right)\\
A_{jkl} & = & \int_{0}^{1}d\rho\left(\rho a_{jkl}\left(\rho x\right)\right)
\end{eqnarray*}
and $\Gamma_{j}$ are all independent of $h$. Finally from \prettyref{eq: Clifford multiplication standard}
and \prettyref{eq: connection in geodesic} may write down the expression
for the Dirac operator \prettyref{eq:Semiclassical Magnetic Dirac}
also given as $D=hc\circ\left(\nabla^{S\otimes L}\right)$ in terms
of the chosen frame and coordinates to be 
\begin{align}
D & =\gamma^{r}w_{r}^{j}\left[h\partial_{x_{j}}+i\frac{x^{k}}{2}\left(da\right){}_{jk}\left(0\right)+ix^{k}x^{l}A_{jkl}+h\left(x^{k}A_{jk}^{0}+\Gamma_{j}\right)\right]\label{eq: Dirac operator geodesic coordinates}\\
 & =\gamma^{r}\left[w_{r}^{j}h\partial_{x_{j}}+iw_{r}^{j}\frac{x^{k}}{2}\left(da\right){}_{jk}\left(0\right)+\frac{1}{2}hg^{-\frac{1}{2}}\partial_{x_{j}}\left(g^{\frac{1}{2}}w_{r}^{j}\right)\right]+\\
 & \gamma^{r}\left[iw_{r}^{j}x^{k}x^{l}A_{jkl}+hw_{r}^{j}\left(x^{k}A_{jk}^{0}+\Gamma_{j}\right)-\frac{1}{2}hg^{-\frac{1}{2}}\partial_{x_{j}}\left(g^{\frac{1}{2}}w_{r}^{j}\right)\right]\in\Psi_{\textrm{cl}}^{1}\left(\Omega_{s}^{0};\mathbb{C}^{2^{m}}\right)\nonumber 
\end{align}
In the second expression above both square brackets are self-adjoint
with respect to the Riemannian density $e^{1}\wedge\ldots\wedge e^{n}=\sqrt{g}dx\coloneqq\sqrt{g}dx^{1}\wedge\ldots\wedge dx^{n}$
with $g=\det\left(g_{ij}\right)$. Again one may obtain an expression
self-adjoint with respect to the Euclidean density $dx$ in the framing
$g^{\frac{1}{4}}u_{j}\otimes\mathtt{l},1\leq j\leq2^{m}$, with the
result being an addition of the term $h\gamma^{j}w_{j}^{k}g^{-\frac{1}{4}}\left(\partial_{x_{k}}g^{\frac{1}{4}}\right)$. 

Let $i_{g}$ be the injectivity radius of $g^{TX}$ . Define the cutoff
$\chi\in C_{c}^{\infty}\left(-1,1\right)$ such that $\chi=1$ on
$\left(-\frac{1}{2},\frac{1}{2}\right)$. We now modify the functions
$w_{j}^{k}$, outside the ball $B_{i_{g}/2}\left(p\right)$, such
that $w_{j}^{k}=\delta_{j}^{k}$ (and hence $g_{jk}=\delta_{jk}$)
are standard outside the ball $B_{i_{g}}\left(p\right)$ of radius
$i_{g}$ centered at $p$. This again gives 
\begin{align}
\mathbb{D} & =\gamma^{r}\left[w_{r}^{j}h\partial_{x_{j}}+iw_{r}^{j}\frac{x^{k}}{2}\left(da\right){}_{jk}\left(0\right)+\frac{1}{2}hg^{-\frac{1}{2}}\partial_{x_{j}}\left(g^{\frac{1}{2}}w_{r}^{j}\right)\right]+\label{eq: Localized Dirac}\\
 & \chi\left(\left|x\right|/i_{g}\right)\gamma^{r}\left[iw_{r}^{j}x^{k}x^{l}A_{jkl}+hw_{r}^{j}\left(x^{k}A_{jk}^{0}+\Gamma_{j}\right)-\frac{1}{2}hg^{-\frac{1}{2}}\partial_{x_{j}}\left(g^{\frac{1}{2}}w_{r}^{j}\right)\right]\nonumber \\
 & \qquad\qquad\qquad\qquad\qquad\qquad\qquad\qquad\qquad\in\Psi_{\textrm{cl}}^{1}\left(\mathbb{R}^{n};\mathbb{C}^{2^{m}}\right)\nonumber 
\end{align}
as a well defined operator on $\mathbb{R}^{n}$ formally self adjoint
with respect to $\sqrt{g}dx$. Again $\mathbb{D}+i$ being elliptic
in the class $S^{0}\left(m\right)$ for the order function 
\[
m=\sqrt{1+g^{jl}\left(\xi_{j}+\frac{x^{k}}{2}\left(da\right)_{jk}\left(0\right)\right)\left(\xi_{l}+\frac{x^{r}}{2}\left(da\right)_{lr}\left(0\right)\right)},
\]
the operator $\mathbb{D}$ is essentially self adjoint. Also as observed
in \cite{Savale2017-Koszul} Section 7 
\begin{equation}
\textrm{tr }\phi\left(\frac{D}{\sqrt{h}}\right)\left(p,\cdot\right)=\textrm{tr }\phi\left(\frac{\mathbb{D}}{\sqrt{h}}\right)\left(0,\cdot\right)\label{eq: finite propagation}
\end{equation}
mod $h^{\infty}$. 

We now introduce the rescaling operator $\mathscr{R}:C^{\infty}\left(\mathbb{R}^{n};\mathbb{C}^{2^{m}}\right)\rightarrow C^{\infty}\left(\mathbb{R}^{n};\mathbb{C}^{2^{m}}\right)$;
$\left(\mathscr{R}s\right)\left(x\right)\coloneqq s\left(\frac{x}{\sqrt{h}}\right)$.
Conjugation by $\mathscr{R}$ amounts to the rescaling of coordinates
$x\rightarrow x\sqrt{h}$. A Taylor expansion in \prettyref{eq: Localized Dirac}
now gives the existence of classical ($h$-independent) self-adjoint,
first-order differential operators $\mathrm{\mathtt{D}}_{j}=a_{j}^{k}\left(x\right)\partial_{x_{k}}+b_{j}\left(x\right)$,
$j=0,1\ldots$, with polynomial coefficients (of degree at most $j+1$)
as well as $h$-dependent self-adjoint, first-order differential operators
$\mathrm{E}_{N+1}=\sum_{\left|\alpha\right|=N+1}x^{\alpha}\left[c_{\alpha}^{k}\left(x;h\right)\partial_{x_{k}}+d_{\alpha}\left(x;h\right)\right]$,
$N\in\mathbb{N}$, with uniformly $C^{\infty}$ bounded coefficients
$c_{j,\alpha}^{k},\,d_{j,\alpha}$ such that 
\begin{eqnarray}
\mathscr{R}\mathbb{D}\mathscr{R}^{-1} & = & \sqrt{h}\mathrm{\mathtt{D}}\quad\textrm{ with}\label{eq: rescaled Dirac-1}\\
\mathrm{\mathtt{D}} & = & \left(\sum_{j=0}^{N}h^{j/2}\mathrm{\mathtt{D}}_{j}\right)+h^{\left(N+1\right)/2}\mathrm{E}_{N+1},\;\forall N.\label{eq: Taylor expansion Dirac-1}
\end{eqnarray}
The coefficients of the polynomials $a_{j}^{k}\left(x\right),\,b_{j}\left(x\right)$
again involve the covariant derivatives of the curvatures $F^{TX},F^{A_{0}}$
and $da$ evaluated at $p$. It is now clear from \prettyref{eq: rescaled Dirac-1}
that 
\begin{equation}
\phi\left(\frac{\mathbb{D}}{\sqrt{h}}\right)\left(x,x'\right)=h^{-n/2}\phi\left(\mathrm{\mathtt{D}}\right)\left(\frac{x}{\sqrt{h}},\frac{x'}{\sqrt{h}}\right).\label{eq: rescaling Schw kernel}
\end{equation}
Next, let $I_{j}=\left\{ k=\left(k_{0},k_{1},\ldots\right)|k_{\alpha}\in\mathbb{N},\:\sum k_{\alpha}=j\right\} $
denote the set of partitions of the integer $j$ and set 
\begin{equation}
\mathtt{C}_{j}^{z}=\sum_{k\in I_{j}}\left(z-\mathrm{\mathtt{D}}_{0}\right)^{-1}\left[\Pi_{\alpha}\mathrm{\mathtt{D}}_{k_{\alpha}}\left(z-\mathrm{\mathtt{D}}_{0}\right)^{-1}\right].\label{eq: jth term kernel expansion}
\end{equation}
The coefficient $u_{j}$ in the expansion \prettyref{eq: local trace expansion}
is now the total integral over $X$ of a smooth family of distributions
$u_{j,p}\in C^{\infty}\left(X;\mathcal{S}'\left(\mathbb{R}_{s}\right)\right)$
parametrized by $X$ 
\begin{eqnarray*}
u_{j} & = & \int_{X}u_{j,p},\quad\textrm{ where}\\
u_{j,p} & = & \textrm{tr }U_{j,p}\:\textrm{ and }\\
U_{j,p}\left(\phi\right) & = & -\frac{1}{\pi}\int_{\mathbb{C}}\bar{\partial}\tilde{\phi}\left(z\right)\mathtt{C}_{j}^{z}\left(0,0\right)dzd\bar{z}\in\textrm{End}S_{p}^{TX}.
\end{eqnarray*}

It was further shown in \cite{Savale2017-Koszul} that each $u_{j,p}$
is point-wise given by a linear combination of the following elementary
distributions
\begin{align}
v_{a}\left(s\right) & \coloneqq s^{a},\;a\in\mathbb{N}_{0}\label{eq: elementary distribution 1}\\
v_{a,b,c,\varLambda}\left(s\right) & \coloneqq\partial_{s}^{a}\left[\left|s\right|s^{b}\left(s^{2}-2\varLambda\right)^{c-\frac{1}{2}}H\left(s^{2}-2\varLambda\right)\right],\label{eq: elementary distribution 2}\\
 & \;\qquad\qquad\qquad\qquad\left(a,b,c;\varLambda\right)\in\mathbb{N}_{0}\times\mathbb{Z}\times\mathbb{N}_{0}\times\mu.\left(\mathbb{N}_{0}^{m}\setminus0\right).\nonumber 
\end{align}
To now state the computation of $u_{1}$; first define $\mathscr{P}_{j}^{\pm}:T_{p}X\rightarrow\textrm{ker }\left(\pm i\mu_{j}-\mathfrak{J}\right)$,
$1\leq j\leq m,$ the projections onto the eigenspaces of $\mathfrak{J}$
with eigenvalue $\pm i\mu_{j}$ respectively in \prettyref{eq:Diagonalizability assumption-1}.
Also set $\frac{d_{j}}{2}=d_{j}^{+}=d_{j}^{-}=\textrm{dim ker }\left(\pm i\mu_{j}-\mathfrak{J}\right)$
and $\mathscr{P}_{j}\coloneqq\mathscr{P}_{j}^{+}+\mathscr{P}_{j}^{-}$.
Next, define the endomorphism 
\begin{eqnarray*}
\left(\nabla^{TX}\mathfrak{J}\right)^{0}:T_{p}X & \rightarrow & T_{p}X\\
\left(\nabla^{TX}\mathfrak{J}\right)^{0}v & \coloneqq & \left(\nabla_{v}^{TX}\mathfrak{J}\right)R,\quad v\in T_{p}X,
\end{eqnarray*}
agreeing with \prettyref{eq: abs. ctct. end.} on $R^{\perp}$, and
set $\left(\nabla^{TX}\mathfrak{J}\right)_{j}\coloneqq\mathscr{P}_{j}\left(\nabla^{TX}\mathfrak{J}\right)^{0}\mathscr{P}_{j}$
, $1\leq j\leq m$. 

We then have the following.
\begin{prop}
The second coefficient $u_{1}$ of \prettyref{eq: local trace expansion}
is given by 
\begin{align}
u_{1,p}\left(s\right) & =c_{1;1}v_{1}+\sum_{\begin{subarray}{l}
\varLambda\in\mu.\left(\mathbb{N}_{0}^{m}\setminus0\right)\end{subarray}}c_{1;1,-2,0,\varLambda}\left(p\right)v_{1,-2,0,\varLambda}\left(s\right)\nonumber \\
 & +\sum_{\begin{subarray}{l}
\varLambda\in\mu.\left(\mathbb{N}_{0}^{m}\setminus0\right)\end{subarray}}c_{1;0,-3,0,\varLambda}\left(p\right)v_{0,-3,0,\varLambda}\left(s\right),\;\textrm{ where}\label{eq: u1 formula}\\
c_{1;1} & =-\frac{\left(\Pi_{j=1}^{m}\mu_{j}\right)}{\left(2\pi\right)^{m+1}}\left[\textrm{tr }\mathfrak{J}^{-2}\left(\nabla^{TX}\mathfrak{J}\right)^{0}\right]\;\textrm{ and}\label{eq: first coeff u1}\\
c_{1;1,-2,0,\varLambda}\left(p\right)=c_{1;0,-3,0,\varLambda}\left(p\right) & =\begin{cases}
-\frac{\left(\Pi_{j=1}^{m}\mu_{j}\right)}{\left(2\pi\right)^{m+1}}\tau\left[\frac{1}{d_{j}}\textrm{tr }\left(\nabla^{TX}\mathfrak{J}\right)_{j}\right]; & \textrm{ if }\varLambda=\mu_{j}\tau\textrm{ for some }j,\\
0; & \textrm{ otherwise}.
\end{cases}\label{eq: 2nd 3rd coeff u1}
\end{align}
\end{prop}
\begin{proof}
We begin by noting the first two terms in \prettyref{eq: Taylor expansion Dirac-1}
\begin{align}
\mathrm{\mathtt{D}}_{0} & =\gamma^{j}\left[\partial_{x_{j}}+i\frac{x^{k}}{2}\left(da\right){}_{jk}\left(0\right)\right]\label{eq: leading term rescaled Dirac}\\
 & =\gamma^{0}\partial_{x_{0}}+\underbrace{\gamma^{j}\left[\partial_{x_{j}}+\frac{i\mu_{j}\left(p\right)}{2}x_{j+m}\right]+\gamma^{j+m}\left[\partial_{x_{j+m}}-\frac{i\mu_{j}\left(p\right)}{2}x_{j}\right]}_{\coloneqq\mathrm{\mathtt{D}}_{00}}\label{eq: leading term rescaled Dirac-1}\\
\mathrm{\mathtt{D}}_{1} & =\frac{i}{3}\gamma^{j}x^{k}x^{l}\underbrace{\left(\nabla_{e_{l}}da\right)_{jk}\left(0\right)}_{\eqqcolon\mathtt{A}_{jkl}}\label{eq: second term Dirac Taylor}\\
 & =\frac{i}{3}\gamma^{j}x^{k}x^{l}\underbrace{g^{TX}\left(e_{j},\left(\nabla_{e_{l}}\mathfrak{J}\right)e_{k}\right)}_{\eqqcolon\mathtt{A}_{jkl}}
\end{align}
using \prettyref{eq: da diagonal form}, \prettyref{eq: Taylor expansion da}.
For future reference we also note that 
\begin{eqnarray*}
\mathrm{\mathtt{D}}_{0}^{2} & = & -\partial_{x_{0}}^{2}+\underbrace{\sum_{j=1}^{m}\left[-\partial_{x_{j}}^{2}-\partial_{x_{j+m}}^{2}+i\mu_{j}\left(x_{j+m}\partial_{x_{j}}-x_{j}\partial_{x_{j+m}}\right)+\frac{1}{4}\left(x_{j}^{2}+x_{j+m}^{2}\right)\right]-i\mathtt{F}_{m}}_{\eqqcolon\mathrm{\mathtt{D}}_{00}^{2}}\\
\mathtt{F}_{m} & = & \mu_{j}\left[\sum_{j=1}^{m}\gamma^{j}\gamma^{j+m}\right]
\end{eqnarray*}
gives the complex harmonic oscillator.

As in the computation for $u_{0}$ in \cite{Savale2017-Koszul}, we
compute $u_{1}$ by computing the expansions of the heat traces $\textrm{ tr }e^{-t\mathrm{\mathtt{D}}^{2}},\textrm{ tr }\mathrm{\mathtt{D}}e^{-t\mathrm{\mathtt{D}}^{2}}$.
First note that following \prettyref{eq: rescaled Dirac-1}, \prettyref{eq: Taylor expansion Dirac-1}
we may compute 
\[
\mathrm{\mathtt{D}}^{2}=\mathrm{\mathtt{D}}_{0}^{2}+\sqrt{h}\left\{ \mathrm{\mathtt{D}}_{0},\mathrm{\mathtt{D}}_{1}\right\} +O\left(h\right).
\]
 An application of Duhamel's principle then yields
\begin{equation}
e^{-t\mathrm{\mathtt{D}}^{2}}=e^{-t\mathrm{\mathtt{D}}_{0}^{2}}-\sqrt{h}\left(\underbrace{\int_{0}^{t}e^{-\left(t-s\right)\mathrm{\mathtt{D}}_{0}^{2}}\left\{ \mathrm{\mathtt{D}}_{0},\mathrm{\mathtt{D}}_{1}\right\} e^{-s\mathrm{\mathtt{D}}_{0}^{2}}ds}_{\eqqcolon\mathtt{U_{10}}}\right)+O\left(h\right).\label{eq: expansion e^-tD2}
\end{equation}
We compute
\begin{eqnarray}
\left\{ \mathrm{\mathtt{D}}_{0},\mathrm{\mathtt{D}}_{1}\right\}  & = & \frac{i}{3}\mathtt{A}_{jkl}\left\{ -2x^{k}x^{l}\partial_{x_{j}}+\gamma^{k}\gamma^{j}x^{l}+\gamma^{l}\gamma^{j}x^{k}-2\left(ia_{j}\right)x^{k}x^{l}\right\} .\label{eq: commutator D0 D1}
\end{eqnarray}
Next set $\mu_{j+m}=\mu_{j},\,1\leq j\leq m$, and note Mehler's formula
\begin{align}
e^{-t\mathrm{\mathtt{D}}_{0}^{2}}\left(x,y\right) & =e^{t\partial_{x_{0}}^{2}}e^{-t\mathrm{\mathtt{D}}_{00}^{2}}\label{eq: Mehler's formula}\\
\nonumber \\
 & =\frac{e^{-\frac{\left(x_{0}-y_{0}\right)^{2}}{4t}}}{\sqrt{4\pi t}}\left(\prod_{j=1}^{m}\frac{\mu_{j}}{4\pi\sinh\mu_{j}t}\right)m_{t}\left(x',y'\right)e^{it\mathtt{F}_{m}},\qquad\nonumber \\
m_{t}\left(x',y'\right) & =\exp\left\{ -\frac{\mu_{j}}{4\tanh\mu_{j}t}\left(\left(x_{j}-y_{j}\right)^{2}+\left(x_{j+m}-y{}_{j+m}\right)^{2}\right)\right.\nonumber \\
 & \qquad\qquad\left.+\frac{\mu_{j}}{2}\tanh\left(\frac{\mu_{j}t}{2}\right)\left(x_{j}y_{j}+x_{j+m}y_{j+m}\right)\right\} \\
 & =\exp\left\{ -\frac{\mu_{j}}{4\tanh\mu_{j}t}\left(x_{j}{}^{2}+x_{j+m}^{2}+y{}_{j}^{2}+y{}_{j+m}^{2}\right)+\frac{\mu_{j}}{2\sinh\mu_{j}t}\left(x_{j}y_{j}+x_{j+m}y_{j+m}\right)\right\} ,\nonumber 
\end{align}
where $\left(x';y'\right)=\left(x_{1},\ldots,x_{2m};y_{1},\ldots,y_{2m}\right)$.
We may now substitute \prettyref{eq: commutator D0 D1} and\prettyref{eq: Mehler's formula}
into \prettyref{eq: expansion e^-tD2}. This gives a formula for $\mathtt{U_{10}}\left(0,0\right)$
as an integral over $s$ and $x$. Furthermore one observes that the
$x$ -integral is an odd integral which must evaluate to $0$. Hence
we have 
\begin{equation}
u_{1}\left(e^{-ts^{2}}\right)=-\textrm{tr }\mathtt{U_{10}}\left(0,0\right)=0.\label{eq: tr U10}
\end{equation}

We now compute the second term in $\textrm{ tr }\mathrm{\mathtt{D}}e^{-t\mathrm{\mathtt{D}}^{2}}$.
First differentiate \prettyref{eq: expansion e^-tD2} using \prettyref{eq: Taylor expansion Dirac-1}
to obtain 
\begin{eqnarray*}
\mathtt{D}e^{-t\mathrm{\mathtt{D}}^{2}} & = & \mathrm{\mathtt{D}}_{0}e^{-t\mathrm{\mathtt{D}}_{0}^{2}}\\
 &  & -\sqrt{h}\left(\mathrm{\mathtt{D}}_{0}\int_{0}^{t}e^{-\left(t-s\right)\mathrm{\mathtt{D}}_{0}^{2}}\left\{ \mathrm{\mathtt{D}}_{0},\mathrm{\mathtt{D}}_{1}\right\} e^{-s\mathrm{\mathtt{D}}_{0}^{2}}ds-\mathrm{\mathtt{D}}_{1}e^{-t\mathrm{\mathtt{D}}_{0}^{2}}\right)+O\left(h\right).
\end{eqnarray*}
The $O\left(\sqrt{h}\right)$ term above maybe rewritten symmetrically
\begin{eqnarray}
 &  & \mathtt{U}_{11}\nonumber \\
 & \coloneqq & \mathrm{\mathtt{D}}_{0}\int_{0}^{t}e^{-\left(t-s\right)\mathrm{\mathtt{D}}_{0}^{2}}\left\{ \mathrm{\mathtt{D}}_{0},\mathrm{\mathtt{D}}_{1}\right\} e^{-s\mathrm{\mathtt{D}}_{0}^{2}}ds-\mathrm{\mathtt{D}}_{1}e^{-t\mathrm{\mathtt{D}}_{0}^{2}}\\
 & = & \underbrace{\int_{0}^{t}\left(\mathrm{\mathtt{D}}_{0}e^{-\left(t-s\right)\mathrm{\mathtt{D}}_{0}^{2}}\right)\mathrm{\mathtt{D}}_{1}\left(\mathrm{\mathtt{D}}_{0}e^{-s\mathrm{\mathtt{D}}_{0}^{2}}\right)ds}_{\eqqcolon\mathtt{K}_{1}}\nonumber \\
 &  & +\underbrace{\frac{1}{2}\int_{0}^{t}e^{-\left(t-s\right)\mathrm{\mathtt{D}}_{0}^{2}}\left\{ \mathrm{\mathtt{D}}_{0}^{2},\mathrm{\mathtt{D}}_{1}\right\} e^{-s\mathrm{\mathtt{D}}_{0}^{2}}ds}_{\eqqcolon\mathtt{K}_{2}}\nonumber \\
 &  & -\frac{1}{2}\left(\mathrm{\mathtt{D}}_{1}e^{-t\mathrm{\mathtt{D}}_{0}^{2}}+e^{-t\mathrm{\mathtt{D}}_{0}^{2}}\mathrm{\mathtt{D}}_{1}\right)\label{eq:computation O h1/2 term}
\end{eqnarray}
using an integration by parts argument. It is clear from \prettyref{eq: second term Dirac Taylor}
that 
\[
\mathrm{\mathtt{D}}_{1}e^{-t\mathrm{\mathtt{D}}_{0}^{2}}\left(0,0\right)=0
\]
with the same being true of its adjoint 
\[
e^{-t\mathrm{\mathtt{D}}_{0}^{2}}\mathrm{\mathtt{D}}_{1}\left(0,0\right)=0.
\]
Similarly the adjointness property for
\begin{eqnarray*}
\int_{0}^{t}e^{-\left(t-s\right)\mathrm{\mathtt{D}}_{0}^{2}}\left(\mathrm{\mathtt{D}}_{0}^{2}\mathrm{\mathtt{D}}_{1}\right)e^{-s\mathrm{\mathtt{D}}_{0}^{2}}ds &  & \textrm{and}\\
\int_{0}^{t}e^{-\left(t-s\right)\mathrm{\mathtt{D}}_{0}^{2}}\left(\mathrm{\mathtt{D}}_{1}\mathrm{\mathtt{D}}_{0}^{2}\right)e^{-s\mathrm{\mathtt{D}}_{0}^{2}}ds
\end{eqnarray*}
gives 
\[
\mathtt{K}_{2}\left(0,0\right)=\left[\int_{0}^{t}e^{-\left(t-s\right)\mathrm{\mathtt{D}}_{0}^{2}}\left(\mathrm{\mathtt{D}}_{1}\mathrm{\mathtt{D}}_{0}^{2}\right)e^{-s\mathrm{\mathtt{D}}_{0}^{2}}ds\right]\left(0,0\right).
\]
We now compute 
\begin{eqnarray*}
 &  & \mathtt{K}_{1}\\
 & = & \int_{0}^{t}ds\left(\mathrm{\mathtt{D}}_{0}e^{-\left(t-s\right)\mathrm{\mathtt{D}}_{0}^{2}}\right)\mathrm{\mathtt{D}}_{1}\left(\mathrm{\mathtt{D}}_{0}e^{-s\mathrm{\mathtt{D}}_{0}^{2}}\right)\\
 & = & \int_{0}^{t}ds\,e^{-\left(t-s\right)\mathrm{\mathtt{D}}_{0}^{2}}\left[\gamma^{\mu}\left(\partial_{x_{\mu}}+ia_{\mu}\right)\right]\left(\frac{i}{3}\gamma^{j}x^{k}x^{l}\mathtt{A}_{jkl}\right)\left(\mathrm{\mathtt{D}}_{0}e^{-s\mathrm{\mathtt{D}}_{0}^{2}}\right)\\
 & = & \int_{0}^{t}ds\,e^{-\left(t-s\right)\mathrm{\mathtt{D}}_{0}^{2}}\left(\frac{i}{3}\gamma^{k}\gamma^{j}x^{l}\mathtt{A}_{jkl}\right)\left(\mathrm{\mathtt{D}}_{0}e^{-s\mathrm{\mathtt{D}}_{0}^{2}}\right)\\
 &  & +\int_{0}^{t}ds\,e^{-\left(t-s\right)\mathrm{\mathtt{D}}_{0}^{2}}\left(\frac{i}{3}\gamma^{l}\gamma^{j}x^{k}\mathtt{A}_{jkl}\right)\left(\mathrm{\mathtt{D}}_{0}e^{-s\mathrm{\mathtt{D}}_{0}^{2}}\right)\\
 &  & -\int_{0}^{t}ds\,2e^{-\left(t-s\right)\mathrm{\mathtt{D}}_{0}^{2}}\left(\frac{i}{3}x^{k}x^{l}\mathtt{A}_{jkl}\right)\left(\partial_{x_{j}}+ia_{j}\right)\left(\mathrm{\mathtt{D}}_{0}e^{-s\mathrm{\mathtt{D}}_{0}^{2}}\right)\\
 &  & -\underbrace{\int_{0}^{t}ds\,e^{-\left(t-s\right)\mathrm{\mathtt{D}}_{0}^{2}}\left(\frac{i}{3}\gamma^{j}x^{k}x^{l}\mathtt{A}_{jkl}\right)\left[\gamma^{\mu}\left(\partial_{x_{\mu}}+ia_{\mu}\right)\right]\left(\mathrm{\mathtt{D}}_{0}e^{-s\mathrm{\mathtt{D}}_{0}^{2}}\right)}_{=\mathtt{K}_{2}}.
\end{eqnarray*}
Hence we now simplify \prettyref{eq:computation O h1/2 term} to 
\begin{eqnarray}
 &  & \mathtt{U}_{11}\nonumber \\
 & = & \underbrace{\int_{0}^{t}ds\,e^{-\left(t-s\right)\mathrm{\mathtt{D}}_{0}^{2}}\left(\frac{i}{3}\gamma^{k}\gamma^{j}x^{l}\mathtt{A}_{jkl}\right)\left(\mathrm{\mathtt{D}}_{0}e^{-s\mathrm{\mathtt{D}}_{0}^{2}}\right)}_{\eqqcolon\mathtt{L}_{1}}\nonumber \\
 &  & +\underbrace{\int_{0}^{t}ds\,e^{-\left(t-s\right)\mathrm{\mathtt{D}}_{0}^{2}}\left(\frac{i}{3}\gamma^{l}\gamma^{j}x^{k}\mathtt{A}_{jkl}\right)\left(\mathrm{\mathtt{D}}_{0}e^{-s\mathrm{\mathtt{D}}_{0}^{2}}\right)}_{\eqqcolon\mathtt{L}_{2}}\nonumber \\
 &  & -\underbrace{2\int_{0}^{t}ds\,e^{-\left(t-s\right)\mathrm{\mathtt{D}}_{0}^{2}}\left(\frac{i}{3}x^{k}x^{l}\mathtt{A}_{jkl}\right)\left(\partial_{x_{j}}+ia_{j}\right)\left(\mathrm{\mathtt{D}}_{0}e^{-s\mathrm{\mathtt{D}}_{0}^{2}}\right)}_{\eqqcolon\mathtt{L}_{3}}\label{eq: breakup L1 L2 L3}
\end{eqnarray}
We now evaluate traces of each of the kernels $\mathtt{L}_{1},\mathtt{L}_{2}$
and $\mathtt{L}_{3}$. 

First compute 
\begin{align}
\mathrm{\mathtt{D}}_{0}e^{-t\mathrm{\mathtt{D}}_{0}^{2}}\left(x,0\right) & =-\left[\frac{\gamma^{0}x_{0}}{2t}+\sum_{\mu=1}^{m}\frac{\mu_{\mu}}{2\tanh\mu_{\mu}t}\left(\gamma^{\mu}x_{\mu}+\gamma^{\mu+m}x_{\mu+m}\right)\right]\frac{e^{-\frac{x_{0}^{2}}{4t}}}{\sqrt{4\pi t}}m_{t}\left(x',0\right)e^{it\mathtt{F}_{m}}\nonumber \\
 & +\left[\sum_{\mu=1}^{m}\frac{i\mu_{\mu}}{2}\left(\gamma^{\mu}x_{\mu+m}-\gamma^{\mu+m}x_{\mu}\right)\right]\frac{e^{-\frac{x_{0}^{2}}{4t}}}{\sqrt{4\pi t}}m_{t}\left(x',0\right)e^{it\mathtt{F}_{m}}.\label{eq: D0 etD0^2}
\end{align}
and set 
\begin{eqnarray*}
\tilde{m}_{t}\left(x,y\right) & \coloneqq & \frac{e^{-\frac{\left(x_{0}-y_{0}\right)^{2}}{4t}}}{\sqrt{4\pi t}}m_{t}\left(x',y'\right)\\
E\left(x';s,t\right) & \coloneqq & m_{t-s}\left(0,x'\right)m_{s}\left(x',0\right)\\
\tilde{E}\left(x;s,t\right) & \coloneqq & \tilde{m}_{t-s}\left(0,x\right)\tilde{m}_{s}\left(x,0\right)\\
\frac{1}{\rho_{\mu}\left(t\right)} & \coloneqq & \begin{cases}
\frac{1}{2t}; & \mu=0,\\
\frac{\mu_{\mu}}{2\tanh\mu_{\mu}t}; & 1\leq\mu\leq2m.
\end{cases}
\end{eqnarray*}

Plugging \prettyref{eq: Mehler's formula} and \prettyref{eq: D0 etD0^2}
into \prettyref{eq: breakup L1 L2 L3} gives 
\begin{eqnarray}
\textrm{tr }\mathtt{L}_{1}\left(0,0\right) & = & \mathtt{A}_{jkl}\left\{ -\underbrace{\int_{0}^{t}ds\int dxE\left(x;s,t\right)\frac{x_{\mu}x_{l}}{\rho_{\mu}\left(s\right)}\textrm{tr}\left[\frac{i}{3}\gamma^{k}\gamma^{j}\gamma^{\mu}e^{it\mathtt{F}_{m}}\right]}_{\eqqcolon\mathtt{l}_{10}^{jkl}}\right.\nonumber \\
 &  & \left.+\underbrace{\int_{0}^{t}ds\int dxE\left(x;s,t\right)\left(ia_{\mu}x_{l}\right)\textrm{tr}\left[\frac{i}{3}\gamma^{k}\gamma^{j}\gamma^{\mu}e^{it\mathtt{F}_{m}}\right]}_{\eqqcolon\mathtt{l}_{11}^{jkl}}\right\} \label{eq: L1}\\
\textrm{tr }\mathtt{L}_{2}\left(0,0\right) & = & \mathtt{A}_{jkl}\left\{ -\underbrace{\int_{0}^{t}ds\int dxE\left(x;s,t\right)\frac{x_{\mu}x_{k}}{\rho_{\mu}\left(s\right)}\textrm{tr}\left[\frac{i}{3}\gamma^{l}\gamma^{j}\gamma^{\mu}e^{it\mathtt{F}_{m}}\right]}_{\eqqcolon\mathtt{l}_{20}^{jkl}}\right.\nonumber \\
 &  & \left.+\underbrace{\int_{0}^{t}ds\int dxE\left(x;s,t\right)\left(ia_{\mu}x_{k}\right)\textrm{tr}\left[\frac{i}{3}\gamma^{l}\gamma^{j}\gamma^{\mu}e^{it\mathtt{F}_{m}}\right]}_{\eqqcolon\mathtt{l}_{21}^{jkl}}\right\} .\label{eq: L2}
\end{eqnarray}
Since the function $E\left(x;s,t\right)$ is an even function in $x$,
we must have $\mu=l$ for the $x$ integral in $\mathtt{L}_{10}^{jkl}$
to be non-zero. Similarly, we must have $\mu,l>0$ with $\left|\mu-l\right|=m$
for the $x$ integral in $\mathtt{l}_{11}^{jkl}$ to be non-zero.
We now note that for indices $p<q<r$; 
\begin{eqnarray*}
\textrm{tr}\left[i\gamma^{p}e^{it\mathtt{F}_{m}}\right] & = & \begin{cases}
2^{m}\left(\Pi_{j=1}^{m}\sinh\mu_{j}t\right); & p=0\\
0 & \textrm{ otherwise}.
\end{cases}\\
\textrm{tr}\left[i\gamma^{p}\gamma^{q}\gamma^{r}e^{it\mathtt{F}_{m}}\right] & = & \begin{cases}
-i2^{m}\frac{\left(\Pi_{j=1}^{m}\sinh\mu_{j}t\right)}{\tanh\left(\mu_{q}t\right)}; & p=0\textrm{ and }r-q=m\\
0 & \textrm{ otherwise}.
\end{cases}
\end{eqnarray*}
This now implies that the coefficient of $\mathtt{A}_{jkl}$ in $\mathtt{L}_{1}$
is zero unless exactly one of the indices$\left(i,j,k\right)$ is
zero; and the other two are either equal or differ by $m$. A similar
analysis also give that the coefficient of $\mathtt{A}_{jkl}$ in
$\mathtt{L}_{2}$, $\mathtt{L}_{3}$ is zero unless exactly one of
$\left(i,j,k\right)$ is zero. Furthermore 
\begin{eqnarray}
\textrm{tr }\mathtt{L}_{3}\left(0,0\right) & = & 2\mathtt{A}_{jkl}\left\{ \underbrace{\int_{0}^{t}ds\int dxE\left(x;s,t\right)x_{k}x_{l}\left(\frac{\delta^{0j}}{2s}-\frac{x_{0}x_{j}}{2s\rho_{j}\left(s\right)}\right)\textrm{tr}\left[\frac{i}{3}\gamma^{0}e^{it\mathtt{F}_{m}}\right]}_{\eqqcolon\mathtt{l}_{30}^{jkl}}\right.\nonumber \\
 &  & \left.+\underbrace{\int_{0}^{t}ds\int dxE\left(x;s,t\right)\left(\frac{ia_{j}x_{0}x_{k}x_{l}}{2s}\right)\textrm{tr}\left[\frac{i}{3}\gamma^{0}e^{it\mathtt{F}_{m}}\right]}_{\eqqcolon\mathtt{l}_{31}^{jkl}}\right\} .\label{eq: L3}
\end{eqnarray}
For future reference we define $\mathtt{l}_{1}^{jkl}=\mathtt{l}_{10}^{jkl}+\mathtt{l}_{11}^{jkl}$,
$\mathtt{l}_{2}^{jkl}=\mathtt{l}_{20}^{jkl}+\mathtt{l}_{21}^{jkl}$,
$\mathtt{l}_{3}^{jkl}=\mathtt{l}_{30}^{jkl}+\mathtt{l}_{31}^{jkl}$
and $\mathtt{u}_{11}^{jkl}\coloneqq\mathtt{l}_{1}^{jkl}+\mathtt{l}_{2}^{jkl}+\mathtt{l}_{3}^{jkl}$. 

We may now make the three cases. 

\textbf{Case (i) }$j=0$\\
Again as observed before we must have either $k=l$ or $\left|k-l\right|=m$.
If $1\leq k=l\leq m$, we compute 
\begin{align}
 & \mathtt{l}_{1}^{0kk}+\mathtt{l}_{2}^{0kk}\nonumber \\
 & \coloneqq\mathtt{l}_{10}^{0kk}+\mathtt{l}_{11}^{0kk}+\mathtt{l}_{20}^{0kk}+\mathtt{l}_{21}^{0kk}\nonumber \\
 & =-\int_{0}^{t}ds\int dx\tilde{E}\left(x;s,t\right)\frac{\mu_{k}x_{k}^{2}}{\tanh\left(\mu_{k}s\right)}\textrm{tr}\left[\frac{i}{3}\gamma^{k}\gamma^{0}\gamma^{k}e^{it\mathtt{F}_{m}}\right]\nonumber \\
 & -\int_{0}^{t}ds\int dx\tilde{E}\left(x;s,t\right)\left(i\mu_{k}x_{k}^{2}\right)\textrm{tr}\left[\frac{i}{3}\gamma^{k}\gamma^{0}\gamma^{k+m}e^{it\mathtt{F}_{m}}\right]\nonumber \\
 & =\frac{1}{3}2^{m}\left(\Pi_{j=1}^{m}\sinh\mu_{j}t\right)\int_{0}^{t}ds\left[\frac{\mu_{k}}{\tanh\left(\mu_{k}t\right)}-\frac{\mu_{k}}{\tanh\left(\mu_{k}s\right)}\right]\left(\int dx\,x_{k}^{2}\tilde{E}\left(x;s,t\right)\right)\nonumber \\
 & =\frac{1}{3}\frac{1}{\sqrt{4\pi t}}2^{m}\left(\Pi_{j=1}^{m}\sinh\mu_{j}t\right)\int_{0}^{t}ds\left[\frac{\mu_{k}}{\tanh\left(\mu_{k}t\right)}-\frac{\mu_{k}}{\tanh\left(\mu_{k}s\right)}\right]\left(\int dx'\,x_{k}^{2}E\left(x';s,t\right)\right)\label{eq: x integral 1}\\
 & =\frac{1}{3}\frac{2\left(\Pi_{j=1}^{m}\mu_{j}\right)}{\left(2\pi\right)^{m}}\frac{1}{\sqrt{4\pi t}}\int_{0}^{t}ds\left[\frac{\mu_{k}}{\tanh\left(\mu_{k}t\right)}-\frac{\mu_{k}}{\tanh\left(\mu_{k}s\right)}\right]\left[\frac{\sinh\mu_{k}s\,\sinh\mu_{k}\left(t-s\right)}{\mu_{k}\sinh\mu_{k}t}\right]\label{eq: s-integral 1}\\
 & =\frac{1}{3}\frac{\left(\Pi_{j=1}^{m}\mu_{j}\right)}{\left(2\pi\right)^{m}}\frac{1}{\mu_{k}}\frac{1}{\sqrt{4\pi t}}\frac{1}{\sinh\mu_{k}t}\left[\frac{\left(\mu_{k}t\right)\cosh\left(\mu_{k}t\right)-\sinh\left(\mu_{k}t\right)}{\tanh\left(\mu_{k}t\right)}-\left(\mu_{k}t\right)\sinh\left(\mu_{k}t\right)\right]\label{eq: L1+L2 (0kk)}
\end{align}
and 
\begin{align}
 & \mathtt{l}_{3}^{0kk}\nonumber \\
 & =\frac{1}{t}\frac{1}{\sqrt{4\pi t}}\int E\left(x';s,t\right)x_{k}^{2}\textrm{tr}\left[\frac{i}{3}\gamma^{0}e^{it\mathtt{F}_{m}}\right]\nonumber \\
 & =\frac{1}{3}\frac{1}{t}\frac{1}{\sqrt{4\pi t}}2^{m}\left(\Pi_{j=1}^{m}\sinh\mu_{j}t\right)\int ds\int dx'E\left(x';s,t\right)x_{k}^{2}\label{eq: x integral 2}\\
 & =\frac{1}{3}\frac{2\left(\Pi_{j=1}^{m}\mu_{j}\right)}{\left(2\pi\right)^{m}}\frac{1}{t}\frac{1}{\sqrt{4\pi t}}\int_{0}^{t}ds\left[\frac{\sinh\mu_{k}s\,\sinh\mu_{k}\left(t-s\right)}{\mu_{k}\sinh\mu_{k}t}\right]\label{eq: s integral 2}\\
 & =\frac{1}{3}\frac{\left(\Pi_{j=1}^{m}\mu_{j}\right)}{\left(2\pi\right)^{m}}\frac{1}{\mu_{k}}\frac{1}{\sqrt{4\pi t}}\frac{1}{\left(\mu_{k}t\right)\sinh\mu_{k}t}\left[\left(\mu_{k}t\right)\cosh\left(\mu_{k}t\right)-\sinh\left(\mu_{k}t\right)\right].\label{eq: L3 0kk}
\end{align}
Here we have used one of the integrals 
\begin{eqnarray*}
\int dx'\,E\left(x';s,t\right) & = & \prod_{j=1}^{m}\frac{\mu_{j}}{4\pi\sinh\mu_{j}t}\\
\int dx'\,x_{k}^{2}E\left(x';s,t\right) & = & 2\left(\prod_{j=1}^{m}\frac{\mu_{j}}{4\pi\sinh\mu_{j}t}\right)\left[\frac{\sinh\mu_{k}s\,\sinh\mu_{k}\left(t-s\right)}{\mu_{k}\sinh\mu_{k}t}\right]\\
\int dx'\,x_{k}^{2}x_{l}^{2}E\left(x';s,t\right) & = & 4\left(\prod_{j=1}^{m}\frac{\mu_{j}}{4\pi\sinh\mu_{j}t}\right)\left[\frac{\sinh\mu_{k}s\,\sinh\mu_{k}\left(t-s\right)}{\mu_{k}\sinh\mu_{k}t}\right]\left[\frac{\sinh\mu_{l}s\,\sinh\mu_{l}\left(t-s\right)}{\mu_{l}\sinh\mu_{l}t}\right]\\
\int dx'\,x_{k}^{4}E\left(x';s,t\right) & = & 12\left(\prod_{j=1}^{m}\frac{\mu_{j}}{4\pi\sinh\mu_{j}t}\right)\left[\frac{\sinh\mu_{k}s\,\sinh\mu_{k}\left(t-s\right)}{\mu_{k}\sinh\mu_{k}t}\right]^{2}
\end{eqnarray*}
in \prettyref{eq: x integral 1}, \prettyref{eq: x integral 2} and
one of 
\begin{eqnarray*}
\int_{0}^{t}ds\,\sinh\mu_{k}s\,\sinh\mu_{k}\left(t-s\right) & = & \frac{1}{2\mu_{k}}\left[\left(\mu_{k}t\right)\cosh\left(\mu_{k}t\right)-\sinh\left(\mu_{k}t\right)\right]\\
\int_{0}^{t}ds\,\cosh\mu_{k}s\,\sinh\mu_{k}\left(t-s\right) & = & \frac{1}{2\mu_{k}}\left(\mu_{k}t\right)\sinh\left(\mu_{k}t\right)\\
\int_{0}^{t}ds\,s\,\sinh\mu_{k}s\,\sinh\mu_{k}\left(t-s\right) & = & \frac{1}{\left(2\mu_{k}\right)^{2}}\left(\mu_{k}t\right)\left[\left(\mu_{k}t\right)\cosh\left(\mu_{k}t\right)-\sinh\left(\mu_{k}t\right)\right]\\
\int_{0}^{t}ds\,s\,\sinh\mu_{k}s\,\cosh\mu_{k}\left(t-s\right) & = & \frac{1}{\left(2\mu_{k}\right)^{2}}\left[\left(\mu_{k}t\right)\cosh\left(\mu_{k}t\right)-\sinh\left(\mu_{k}t\right)+\left(\mu_{k}t\right)^{2}\sinh\left(\mu_{k}t\right)\right]
\end{eqnarray*}
in \prettyref{eq: s-integral 1}, \prettyref{eq: s integral 2}. The
sum of \prettyref{eq: L1+L2 (0kk)} and \prettyref{eq: L3 0kk} now
gives 
\begin{eqnarray}
\mathtt{u}_{11}^{0kk}\left(0,0\right) & \coloneqq & \mathtt{l}_{1}^{0kk}+\mathtt{l}_{2}^{0kk}+\mathtt{l}_{3}^{0kk}\nonumber \\
 & = & \frac{1}{3}\frac{\left(\Pi_{j=1}^{m}\mu_{j}\right)}{\left(2\pi\right)^{m}}\frac{1}{\mu_{k}}\frac{1}{\sqrt{4\pi t}}\left[\frac{\mu_{k}t}{\left(\sinh\mu_{k}t\right)^{2}}-\frac{1}{\mu_{k}t}\right].\label{eq: u11 0kk}
\end{eqnarray}
A similar computation yields the same answer for $\mathtt{u}_{11}^{0kk}\left(0,0\right)$
if $k>m$. 

We now consider the possibility$l=k+m$ and compute 
\begin{align*}
 & \mathtt{l}_{1}^{0k\left(k+m\right)}+\mathtt{l}_{2}^{0k\left(k+m\right)}\\
 & =\mathtt{l}_{10}^{0kk}+\mathtt{l}_{11}^{0kk}+\mathtt{l}_{20}^{0kk}+\mathtt{l}_{21}^{0kk}\\
 & =-\int_{0}^{t}ds\int dxE\left(x;s,t\right)\frac{x_{k+m}^{2}}{\rho_{k+m}\left(s\right)}\textrm{tr}\left[\frac{i}{3}\gamma^{k}\gamma^{0}\gamma^{k+m}e^{it\mathtt{F}_{m}}\right]\\
 & +\int_{0}^{t}ds\int dxE\left(x;s,t\right)\left(i\mu_{k}x_{k+m}^{2}\right)\textrm{tr}\left[\frac{i}{3}\gamma^{k}\gamma^{j}\gamma^{k}e^{it\mathtt{F}_{m}}\right]\\
 & -\int_{0}^{t}ds\int dxE\left(x;s,t\right)\frac{x_{k}^{2}}{\rho_{k}\left(s\right)}\textrm{tr}\left[\frac{i}{3}\gamma^{k+m}\gamma^{j}\gamma^{k}e^{it\mathtt{F}_{m}}\right]\\
 & +\int_{0}^{t}ds\int dxE\left(x;s,t\right)\left(-i\mu_{k}x_{k}^{2}\right)\textrm{tr}\left[\frac{i}{3}\gamma^{k+m}\gamma^{j}\gamma^{k+m}e^{it\mathtt{F}_{m}}\right]\\
 & =0
\end{align*}
and 
\begin{eqnarray*}
 &  & \mathtt{l}_{3}^{0k\left(k+m\right)}\\
 & = & \frac{1}{t}\frac{1}{\sqrt{4\pi t}}\int E\left(x';s,t\right)x_{k}x_{k+m}\textrm{tr}\left[\frac{i}{3}\gamma^{0}e^{it\mathtt{F}_{m}}\right]\\
 & = & 0.
\end{eqnarray*}
Hence
\begin{equation}
\mathtt{u}_{11}^{0k\left(k+m\right)}\left(0,0\right)=0.\label{eq: u11 0k k+m}
\end{equation}
 A similar computation in the case $k=l+m$ shows $\mathtt{u}_{11}^{0\left(k+m\right)k}\left(0,0\right)=0$
.

\textbf{Case (ii) $k=0$}\\
Again as observed before we must have either $j=l$ or $\left|j-l\right|=m$.
If $1\leq j=l\leq m$, we compute
\begin{align}
 & \mathtt{l}_{1}^{j0j}+\mathtt{l}_{2}^{j0j}\nonumber \\
 & \coloneqq\mathtt{l}_{10}^{j0j}+\mathtt{l}_{11}^{j0j}+\mathtt{l}_{20}^{j0j}+\mathtt{l}_{21}^{j0j}\nonumber \\
 & =\int_{0}^{t}ds\int dx\tilde{E}\left(x;s,t\right)\left(\frac{\mu_{j}}{2\tanh\left(\mu_{j}s\right)}\right)x_{j}^{2}\textrm{tr}\left[\frac{i}{3}\gamma^{0}e^{it\mathtt{F}_{m}}\right]\nonumber \\
 & -\int_{0}^{t}ds\int dx\tilde{E}\left(x;s,t\right)\left(i\frac{\mu_{j}}{2}x_{j}^{2}\right)\textrm{tr}\left[\frac{i}{3}\gamma^{0}\gamma^{j}\gamma^{j+m}e^{it\mathtt{F}_{m}}\right]\nonumber \\
 & +\int_{0}^{t}ds\int dx\tilde{E}\left(x;s,t\right)\frac{x_{0}^{2}}{2s}\textrm{tr}\left[\frac{i}{3}\gamma^{0}e^{it\mathtt{F}_{m}}\right]\nonumber \\
 & =\frac{1}{3}\frac{1}{\sqrt{4\pi t}}\frac{\left(\Pi_{j=1}^{m}\mu_{j}\right)}{\left(2\pi\right)^{m}}\int_{0}^{t}ds\left(\frac{\mu_{j}}{\tanh\left(\mu_{j}s\right)}\right)\left[\frac{\sinh\mu_{j}s\,\sinh\mu_{j}\left(t-s\right)}{\mu_{j}\sinh\mu_{j}t}\right]\nonumber \\
 & -\frac{1}{3}\frac{1}{\sqrt{4\pi t}}\frac{\left(\Pi_{j=1}^{m}\mu_{j}\right)}{\left(2\pi\right)^{m}}\int_{0}^{t}ds\frac{\mu_{j}}{\tanh\left(\mu_{j}t\right)}\left[\frac{\sinh\mu_{j}s\,\sinh\mu_{j}\left(t-s\right)}{\mu_{j}\sinh\mu_{j}t}\right]\nonumber \\
 & +\frac{1}{3}\frac{1}{\sqrt{4\pi t}}\frac{\left(\Pi_{j=1}^{m}\mu_{j}\right)}{\left(2\pi\right)^{m}}\int_{0}^{t}ds\left(\frac{t-s}{t}\right)\nonumber \\
 & =\frac{1}{3}\frac{1}{\sqrt{4\pi t}}\frac{\left(\Pi_{j=1}^{m}\mu_{j}\right)}{\left(2\pi\right)^{m}}\frac{t}{2}\nonumber \\
 & -\frac{1}{3}\frac{1}{\sqrt{4\pi t}}\frac{\left(\Pi_{j=1}^{m}\mu_{j}\right)}{\left(2\pi\right)^{m}}\frac{1}{2\mu_{j}}\frac{1}{\tanh\left(\mu_{j}t\right)}\frac{1}{\sinh\mu_{j}t}\left[\left(\mu_{j}t\right)\cosh\mu_{j}t-\sinh\mu_{j}t\right]\nonumber \\
 & +\frac{1}{3}\frac{1}{\sqrt{4\pi t}}\frac{\left(\Pi_{j=1}^{m}\mu_{j}\right)}{\left(2\pi\right)^{m}}\frac{t}{2}\label{eq: L1 j0j + L2 j0j}
\end{align}
and 
\begin{align}
 & \mathtt{l}_{3}^{j0j}\nonumber \\
 & =-\int_{0}^{t}ds\int dx\tilde{E}\left(x;s,t\right)\frac{1}{2s}\frac{\mu_{j}}{\tanh\mu_{j}s}x_{0}^{2}x_{j}^{2}\textrm{tr}\left[\frac{i}{3}\gamma^{0}e^{it\mathtt{F}_{m}}\right]\nonumber \\
 & =-\frac{1}{3}\frac{2}{\sqrt{4\pi t}}\frac{\left(\Pi_{j=1}^{m}\mu_{j}\right)}{\left(2\pi\right)^{m}}\int_{0}^{t}ds\left(\frac{t-s}{t}\right)\left[\frac{\cosh\mu_{j}s\,\sinh\mu_{j}\left(t-s\right)}{\sinh\mu_{j}t}\right]\nonumber \\
 & =-\frac{1}{3}\frac{1}{\sqrt{4\pi t}}\frac{\left(\Pi_{j=1}^{m}\mu_{j}\right)}{\left(2\pi\right)^{m}}\frac{1}{\left(\mu_{j}t\right)\sinh\mu_{j}t}\frac{1}{2\mu_{j}}\left[\left(\mu_{j}t\right)\cosh\mu_{j}t-\sinh\mu_{j}t+\mu_{j}^{2}t^{2}\sinh\mu_{j}t\right]\label{eq: L3 j0j}
\end{align}
The sum of \prettyref{eq: L1 j0j + L2 j0j} and \prettyref{eq: L3 j0j}
now gives 
\begin{equation}
\mathtt{u}_{11}^{j0j}\left(0,0\right)=\frac{1}{3}\frac{\left(\Pi_{j=1}^{m}\mu_{j}\right)}{\left(2\pi\right)^{m}}\frac{1}{2\mu_{j}}\frac{1}{\sqrt{4\pi t}}\left[\frac{1}{\mu_{j}t}-\frac{\mu_{j}t}{\left(\sinh\mu_{j}t\right)^{2}}\right]\label{eq: u11 j0j}
\end{equation}

If $l=j+m$, we compute
\begin{align}
 & \mathtt{l}_{1}^{j0\left(j+m\right)}+\mathtt{l}_{2}^{j0\left(j+m\right)}\nonumber \\
 & \coloneqq\mathtt{l}_{10}^{j0\left(j+m\right)}+\mathtt{l}_{11}^{j0\left(j+m\right)}+\mathtt{l}_{20}^{j0\left(j+m\right)}+\mathtt{l}_{21}^{j0\left(j+m\right)}\nonumber \\
 & =-\int_{0}^{t}ds\int dx\tilde{E}\left(x;s,t\right)\left(\frac{\mu_{j}}{2\tanh\left(\mu_{j}s\right)}\right)x_{j+m}^{2}\textrm{tr}\left[\frac{i}{3}\gamma^{0}\gamma^{j}\gamma^{j+m}e^{it\mathtt{F}_{m}}\right]\nonumber \\
 & -\int_{0}^{t}ds\int dx\tilde{E}\left(x;s,t\right)\left(i\frac{\mu_{j}}{2}x_{j+m}^{2}\right)\textrm{tr}\left[\frac{i}{3}\gamma^{0}e^{it\mathtt{F}_{m}}\right]\nonumber \\
 & -\int_{0}^{t}ds\int dx\tilde{E}\left(x;s,t\right)\frac{x_{0}^{2}}{2s}\textrm{tr}\left[\frac{i}{3}\gamma^{j+m}\gamma^{j}\gamma^{0}e^{it\mathtt{F}_{m}}\right]\nonumber \\
 & =\frac{1}{3}\frac{i}{\sqrt{4\pi t}}\frac{\left(\Pi_{j=1}^{m}\mu_{j}\right)}{\left(2\pi\right)^{m}}\int_{0}^{t}ds\frac{\mu_{j}}{\tanh\left(\mu_{j}s\right)}\frac{1}{\tanh\left(\mu_{j}t\right)}\left[\frac{\sinh\mu_{j}s\,\sinh\mu_{j}\left(t-s\right)}{\mu_{j}\sinh\mu_{j}t}\right]\nonumber \\
 & -\frac{1}{3}\frac{i}{\sqrt{4\pi t}}\frac{\left(\Pi_{j=1}^{m}\mu_{j}\right)}{\left(2\pi\right)^{m}}\int_{0}^{t}ds\left[\frac{\sinh\mu_{j}s\,\sinh\mu_{j}\left(t-s\right)}{\mu_{j}\sinh\mu_{j}t}\right]\nonumber \\
 & -\frac{1}{3}\frac{i}{\sqrt{4\pi t}}\frac{\left(\Pi_{j=1}^{m}\mu_{j}\right)}{\left(2\pi\right)^{m}}\int_{0}^{t}ds\left(\frac{t-s}{t}\right)\frac{1}{\tanh\left(\mu_{j}t\right)}\nonumber \\
 & =-\frac{1}{3}\frac{i}{\sqrt{4\pi t}}\frac{\left(\Pi_{j=1}^{m}\mu_{j}\right)}{\left(2\pi\right)^{m}}\frac{1}{2\mu_{j}}\frac{\mu_{j}t}{\tanh\left(\mu_{j}t\right)}\nonumber \\
 & -\frac{1}{3}\frac{i}{\sqrt{4\pi t}}\frac{\left(\Pi_{j=1}^{m}\mu_{j}\right)}{\left(2\pi\right)^{m}}\frac{1}{2\mu_{j}^{2}\sinh\mu_{j}t}\left[\left(\mu_{j}t\right)\cosh\mu_{j}t-\sinh\mu_{j}t\right]\nonumber \\
 & \frac{1}{3}\frac{i}{\sqrt{4\pi t}}\frac{\left(\Pi_{j=1}^{m}\mu_{j}\right)}{\left(2\pi\right)^{m}}\frac{t}{2\tanh\left(\mu_{j}t\right)}\nonumber \\
 & =-\frac{1}{3}\frac{i}{\sqrt{4\pi t}}\frac{\left(\Pi_{j=1}^{m}\mu_{j}\right)}{\left(2\pi\right)^{m}}\frac{1}{2\mu_{j}}\left[\left(\mu_{j}t\right)\cosh\mu_{j}t-\sinh\mu_{j}t\right]\label{eq: l1 j0(j+m) + l2 j0(j+m)}
\end{align}
and 
\begin{align}
 & \mathtt{l}_{3}^{j0\left(j+m\right)}\nonumber \\
 & =\mathtt{l}_{31}^{j0\left(j+m\right)}\nonumber \\
 & =2\int_{0}^{t}ds\int dxE\left(x;s,t\right)\left(\frac{i\mu_{j}x_{0}^{2}x_{j+m}^{2}}{4s}\right)\textrm{tr}\left[\frac{i}{3}\gamma^{0}e^{it\mathtt{F}_{m}}\right]\nonumber \\
 & =\frac{1}{3}\frac{2i}{\sqrt{4\pi t}}\frac{\left(\Pi_{j=1}^{m}\mu_{j}\right)}{\left(2\pi\right)^{m}}\int_{0}^{t}ds\left(\frac{t-s}{t}\right)\left[\frac{\sinh\mu_{j}s\,\sinh\mu_{j}\left(t-s\right)}{\mu_{j}\sinh\mu_{j}t}\right]\nonumber \\
 & =\frac{1}{3}\frac{i}{\sqrt{4\pi t}}\frac{\left(\Pi_{j=1}^{m}\mu_{j}\right)}{\left(2\pi\right)^{m}}\frac{1}{2\mu_{j}^{2}\sinh\mu_{j}t}\left[\left(\mu_{j}t\right)\cosh\mu_{j}t-\sinh\mu_{j}t\right]\label{eq: l3 j0(j+m)}
\end{align}
The sum of \prettyref{eq: l1 j0(j+m) + l2 j0(j+m)} and \prettyref{eq: l3 j0(j+m)}
gives
\begin{equation}
\mathtt{u}_{11}^{j0\left(j+m\right)}\left(0,0\right)=0.\label{eq: u11 j0 j+m}
\end{equation}
 A similar computation also yields $\mathtt{u}_{11}^{\left(j+m\right)0j}\left(0,0\right)=0$.

\textbf{Case (iii) $l=0$}\\
Again as observed before we must have either $j=k$ or $\left|j-k\right|=m$.
The tensor $\mathtt{A}_{jkl}$ in \prettyref{eq: second term Dirac Taylor}
being anti-symmetric in $j,k$; we have $\mathtt{u}_{11}^{jj0}\left(0,0\right)=0$.
On the other hand, the expressions for $\mathtt{l}_{1}^{jkl}+\mathtt{l}_{2}^{jkl}$
and $\mathtt{l}_{3}^{jkl}$ are symmetric in $k,l$. Hence we find
\begin{eqnarray}
\mathtt{u}_{11}^{j\left(j+m\right)0}\left(0,0\right) & = & \mathtt{u}_{11}^{j0\left(j+m\right)}\left(0,0\right)=0\label{eq: u11 j j+m 0}\\
\mathtt{u}_{11}^{\left(j+m\right)j0}\left(0,0\right) & = & \mathtt{u}_{11}^{\left(j+m\right)0j}\left(0,0\right)=0\label{eq: u11 j+m j 0}
\end{eqnarray}
as in the previous case.

To sum up, from \prettyref{eq:computation O h1/2 term}, \prettyref{eq: breakup L1 L2 L3},
\prettyref{eq: L1}, \prettyref{eq: L2}, \prettyref{eq: L3}, \prettyref{eq: u11 0kk},
\prettyref{eq: u11 0k k+m}, \prettyref{eq: u11 j0j}, \prettyref{eq: u11 j0 j+m},
\prettyref{eq: u11 j j+m 0} and \prettyref{eq: u11 j+m j 0} we have
finally have 
\begin{eqnarray}
u_{1}\left(se^{-ts^{2}}\right) & = & -\textrm{tr }\mathtt{U_{11}}\left(0,0\right)\nonumber \\
 & = & -\mathtt{A}_{jkl}\mathtt{u}_{11}^{jkl}\nonumber \\
 & = & -\frac{\mathtt{A}_{0kk}}{3}\frac{\left(\Pi_{j=1}^{m}\mu_{j}\right)}{\left(2\pi\right)^{m}}\frac{1}{\mu_{k}}\frac{1}{\sqrt{4\pi t}}\left[\frac{\mu_{k}t}{\left(\sinh\mu_{k}t\right)^{2}}-\frac{1}{\mu_{k}t}\right]\nonumber \\
 &  & -\frac{\mathtt{A}_{j0j}}{3}\frac{\left(\Pi_{j=1}^{m}\mu_{j}\right)}{\left(2\pi\right)^{m}}\frac{1}{2\mu_{j}}\frac{1}{\sqrt{4\pi t}}\left[\frac{1}{\mu_{j}t}-\frac{\mu_{j}t}{\left(\sinh\mu_{j}t\right)^{2}}\right]\nonumber \\
 & = & -\mathtt{A}_{j0j}\frac{\left(\Pi_{j=1}^{m}\mu_{j}\right)}{\left(2\pi\right)^{m}}\frac{1}{2\mu_{j}}\frac{1}{\sqrt{4\pi t}}\left[\frac{1}{\mu_{j}t}-\frac{\mu_{j}t}{\left(\sinh\mu_{j}t\right)^{2}}\right].\label{eq: tr u11}
\end{eqnarray}

A simple computation using Laplace transforms now shows 
\begin{align}
-\mathtt{A}_{j0j}\frac{\left(\Pi_{j=1}^{m}\mu_{j}\right)}{\left(2\pi\right)^{m}}\frac{1}{2\mu_{j}^{2}t}\frac{1}{\sqrt{4\pi t}} & =-\frac{\left(\Pi_{j=1}^{m}\mu_{j}\right)}{\left(2\pi\right)^{m+1}}\,.\underbrace{\frac{\mathtt{A}_{j0j}}{\mu_{j}^{2}}}_{=\textrm{tr }\mathfrak{J}^{-2}\left(\nabla^{TX}\mathfrak{J}\right)^{0}}.\,v_{1}\left(se^{-ts^{2}}\right)\label{eq: first summand u1}\\
\mathtt{A}_{j0j}\frac{\left(\Pi_{j=1}^{m}\mu_{j}\right)}{\left(2\pi\right)^{m}}\frac{1}{2\mu_{j}}\frac{1}{\sqrt{4\pi t}}\frac{\mu_{j}t}{\left(\sinh\mu_{j}t\right)^{2}} & =\frac{\left(\Pi_{j=1}^{m}\mu_{j}\right)}{\left(2\pi\right)^{m}}\mathtt{A}_{j0j}\frac{\sqrt{t}}{\sqrt{\pi}}\left[\sum_{\tau=1}^{\infty}\tau e^{-2\tau\mu_{j}t}\right]\nonumber \\
 & =-\frac{\left(\Pi_{j=1}^{m}\mu_{j}\right)}{\left(2\pi\right)^{m+1}}\underbrace{\mathtt{A}_{j0j}}_{=\frac{1}{d_{j}}\textrm{tr }\left(\nabla^{TX}\mathfrak{J}\right)_{j}}\left[\sum_{\tau=1}^{\infty}\tau\left(v_{1,-2,0,\varLambda}+v_{0,-3,0,\varLambda}\right)\left(se^{-ts^{2}}\right)\right]\label{eq: second summand u1}
\end{align}
where $2\varLambda=2\tau\mu_{j}=2\tau\mu_{j}$ in the last equation
\prettyref{eq: second summand u1} above. 

Thus, \prettyref{eq: tr U10}, \prettyref{eq: tr u11}, \prettyref{eq: first summand u1}
and \prettyref{eq: second summand u1} show that the two sides of
\prettyref{eq: u1 formula} evaluate equally on test functions $e^{-ts^{2}}$,
$se^{-ts^{2}}$. Differentiating $k$ times and setting $t=1$; they
evaluate equally on test functions $s^{2k}e^{-s^{2}},$ $s^{2k+1}e^{-s^{2}}$
for each $k$. The density of this set of functions in Schwartz space
$\mathcal{S}\left(\mathbb{R}\right)$ now gives the result.
\end{proof}
We end with a corollary of the above computation useful in the next
section.
\begin{cor}
\label{cor: convergence improper integral}The improper integral converges
\[
\int_{0}^{\infty}u_{1}\left(se^{-ts^{2}}\right)\frac{dt}{\sqrt{\pi t}}=-\frac{1}{2}\frac{1}{\left(2\pi\right)^{m+1}}\frac{1}{m!}\int_{X}\left[\textrm{tr }\frac{1}{\left|\mathfrak{J}\right|}\left(\nabla^{TX}\mathfrak{J}\right)^{0}\right]a\wedge\left(da\right)^{m}.
\]
\end{cor}
\begin{proof}
This is a calculation from \prettyref{eq: tr u11}
\begin{eqnarray*}
\int_{0}^{\infty}u_{1}\left(se^{-ts^{2}}\right)\frac{dt}{\sqrt{\pi t}} & = & -\frac{1}{2}\int_{X}dx\mathtt{A}_{j0j}\frac{\left(\Pi_{j=1}^{m}\mu_{j}\right)}{\left(2\pi\right)^{m+1}}\,\int_{0}^{\infty}\frac{1}{\mu_{j}t}\left[\frac{1}{\mu_{j}t}-\frac{\mu_{j}t}{\left(\sinh\mu_{j}t\right)^{2}}\right]dt\\
 & = & -\frac{1}{2}\int_{X}dx\frac{\left(\Pi_{j=1}^{m}\mu_{j}\right)}{\left(2\pi\right)^{m+1}}\frac{\mathtt{A}_{j0j}}{\mu_{j}}\int_{0}^{\infty}\frac{1}{u}\left[\frac{1}{u}-\frac{u}{\left(\sinh u\right)^{2}}\right]du\\
 & = & -\frac{1}{2}\int_{X}dx\frac{\left(\Pi_{j=1}^{m}\mu_{j}\right)}{\left(2\pi\right)^{m+1}}\frac{\mathtt{A}_{j0j}}{\mu_{j}}\,\left[-\frac{1}{u}+\frac{2}{e^{2u}-1}\right]_{0}^{\infty}\\
 & = & \frac{1}{2}\int_{X}dx\frac{\left(\Pi_{j=1}^{m}\mu_{j}\right)}{\left(2\pi\right)^{m+1}}\frac{\mathtt{A}_{j0j}}{\mu_{j}}\,\left[\lim_{u\rightarrow0}\frac{1+2u-e^{2u}}{u\left(e^{2u}-1\right)}\right]\\
 & = & -\frac{1}{2}\frac{1}{\left(2\pi\right)^{m+1}}\int_{X}\underbrace{\frac{\mathtt{A}_{j0j}}{\mu_{j}}}_{=\textrm{tr }\frac{1}{\left|R\right|\left|\mathfrak{J}\right|}\left(\nabla^{TX}\mathfrak{J}\right)^{0}}\underbrace{\left(\Pi_{j=1}^{m}\mu_{j}\right)dx}_{=\frac{1}{m!}\left|R\right|a\wedge\left(da\right)^{m}}.
\end{eqnarray*}
\end{proof}

\section{\label{sec:Semiclassical-limit-of eta}Semiclassical limit of the
eta invariant}

In this section we prove the semiclassical limit formula for the eta
invariant of \prettyref{thm: eta semiclassical limit}. First, from
\cite{Savale2017-Koszul} Cor. 7.3, the distributions $u_{j}\in\mathcal{S}'\left(\mathbb{R}\right)$
of \prettyref{eq: local trace expansion} are smooth near $0$. Hence
\[
u_{j}^{\pm}\left(x\right)\coloneqq1_{\left[0,\infty\right)}\left(\pm x\right)u_{j}\left(x\right)\in\mathcal{S}'\left(\mathbb{R}\right)
\]
are well defined tempered distributions and we similarly define $f^{\pm}$
for any $f\in\mathcal{S}\left(\mathbb{R}\right)$. We now have two
term asymptotics for irregular functional traces similar to \prettyref{prop: local trace expansion}.
\begin{lem}
For any $f\in\mathcal{S}\left(\mathbb{R}\right)$,
\begin{equation}
\textrm{tr }f^{\pm}\left(\frac{D}{\sqrt{h}}\right)=h^{-m-\frac{1}{2}}u_{0}^{\pm}\left(f\right)+h^{-m}u_{1}^{\pm}\left(f\right)+o\left(h^{-m}\right).\label{eq: irreg funct tr exp}
\end{equation}
\end{lem}
\begin{proof}
We begin by proving an improved local Weyl law. To this end, choose
$\theta\in C_{c}^{\infty}\left(\mathbb{R};\left[0,1\right]\right)$
such that $\theta\left(x\right)=1$ near $0$ and $\check{\theta}\left(\xi\right)\geq\frac{1}{4}$
for $\left|\xi\right|\leq1$ in \prettyref{eq: Main trace expansion-1}.
For each $\epsilon>0$, set $\theta_{\epsilon}\left(x\right)=\theta\left(\epsilon x\right)$
and let $N\left(a,b\right)$ denote the number of eigenvalues of $D_{h}$
in the interval $\left(a,b\right)$. Choosing $f\left(x\right)\geq0$
with $f\left(0\right)=1$, the trace expansion \prettyref{eq: Main trace expansion-1}
with $\lambda=0$ now gives 
\[
\frac{1}{\epsilon h}N\left(-\epsilon h,\epsilon h\right)\left(\frac{1}{4}+O\left(\sqrt{\epsilon h}\right)\right)\leq\textrm{tr}\left[f\left(\frac{D}{\sqrt{h}}\right)\frac{1}{\epsilon h}\check{\theta}\left(\frac{-D}{\epsilon h}\right)\right]=h^{-m-1}\left[u_{0}\left(0\right)+O_{\epsilon}\left(h\right)\right].
\]
Hence for $\epsilon>0$ fixed and $h\ll1$ depending on $\epsilon$,
we have an improved local Weyl law
\begin{equation}
N\left(-\epsilon h,\epsilon h\right)=O\left(\epsilon h^{-m}\right).\label{eq: improves local Weyl law}
\end{equation}
From here \prettyref{eq:dimension kernel estimate} follows. 

Now, to prove \prettyref{eq: irreg funct tr exp} first observe that
by virtue of \prettyref{prop: local trace expansion} we may assume
$f\in C_{c}^{\infty}\left(-\sqrt{2\mu_{1}},\sqrt{2\mu_{1}}\right)$.
Next define the spectral measure $\mathfrak{M}_{f}\left(\lambda'\right)\coloneqq\sum_{\lambda\in\textrm{Spec}\left(\frac{D}{\sqrt{h}}\right)}f\left(\lambda\right)\delta\left(\lambda-\lambda'\right)$.
It is clear that the expansion \prettyref{eq: Main trace expansion-1}
to its first two terms may be written as 
\[
\mathfrak{M}_{f}\ast\left(\mathcal{F}_{h}^{-1}\theta_{\frac{1}{2}}\right)\left(\lambda\right)=h^{-m-\frac{1}{2}}\left(f\left(\lambda\right)u_{0}\left(\lambda\right)+h^{1/2}f\left(\lambda\right)u_{1}\left(\lambda\right)+O\left(h\right)\right)
\]
where $\theta_{\frac{1}{2}}\left(x\right)=\theta\left(\frac{x}{\sqrt{h}}\right)$.
Both sides above involving Schwartz functions in $\lambda$, the remainder
maybe replaced by $O\left(\frac{h}{\left\langle \lambda\right\rangle ^{2}}\right)$.
We may then integrate to obtain 
\begin{equation}
\int_{-\infty}^{0}d\lambda\int d\lambda'\left(\mathcal{F}_{h}^{-1}\theta_{\frac{1}{2}}\right)\left(\lambda-\lambda'\right)\mathfrak{M}_{f}\left(\lambda'\right)=h^{-m-\frac{1}{2}}\left(\int_{-\infty}^{0}d\lambda f\left(\lambda\right)u_{0}\left(\lambda\right)+h^{1/2}\int_{-\infty}^{0}d\lambda f\left(\lambda\right)u_{1}\left(\lambda\right)+O\left(h\right)\right).\label{eq:main trace exp integrated}
\end{equation}
Now note 
\begin{equation}
\int_{-\infty}^{0}d\lambda\left(\mathcal{F}_{h}^{-1}\theta_{\frac{1}{2}}\right)\left(\lambda-\lambda'\right)=1_{\left(-\infty,0\right]}\left(\lambda'\right)+\phi\left(\frac{\lambda'}{\sqrt{h}}\right)\label{eq: definition remainder}
\end{equation}
where $\phi\left(x\right)\coloneqq\int_{-\infty}^{0}dt\check{\theta}\left(t-x\right)-1_{\left(-\infty,0\right]}\left(x\right)$
is a function that is rapidly decaying with all derivatives, odd and
smooth on $\mathbb{R}_{x}\setminus0$. Next let $\chi\in C_{c}^{\infty}\left(\mathbb{R};\left[0,1\right]\right)$
be an even function equal to $1$ near $0$ and set $\phi_{R}\left(x\right)=\chi\left(\frac{x}{R}\right)\phi\left(x\right)$
for each $R>0$. We now compute 
\begin{align}
 & \int d\lambda'\left[\phi\left(\frac{\lambda'}{\sqrt{h}}\right)-\phi_{R}\left(\frac{\lambda'}{\sqrt{h}}\right)\right]\mathfrak{M}_{f}\left(\lambda'\right)\nonumber \\
= & \int d\lambda'\left[1-\chi\left(\frac{\lambda'}{R\sqrt{h}}\right)\right]\phi\left(\frac{\lambda'}{\sqrt{h}}\right)\mathfrak{M}_{f}\left(\lambda'\right)\nonumber \\
= & O\left(h^{-m}\sum_{k\geq R}\left\langle k\right\rangle ^{-\infty}\right)\nonumber \\
= & O\left(\frac{h^{-m}}{R}\right)\label{eq: cutoff remainder}
\end{align}
from the local Weyl law \prettyref{eq: improves local Weyl law}. 

Next for $\epsilon>0$, we observe
\begin{align}
 & \left|\phi_{R}\left(x\right)-\phi_{R}\ast\check{\theta}_{\epsilon}\left(x\right)\right|\nonumber \\
= & \left|\int dy\left[\phi_{R}\left(x\right)-\phi_{R}\left(x-\epsilon y\right)\right]\check{\theta}\left(y\right)\right|\nonumber \\
\leq & O_{N}\left(1\right)\left[\left\langle \frac{x}{\epsilon}\right\rangle ^{-N}+\epsilon\left\langle x\right\rangle ^{-N}\right]\quad\forall N\in\mathbb{N}.\label{eq: convolution distribution estimate}
\end{align}
Now consider a pairing corresponding to the first term above with
$\mathfrak{M}_{f}\left(\lambda'\right)$
\begin{align}
 & \int d\lambda'\left\langle \frac{\lambda'}{\epsilon\sqrt{h}}\right\rangle ^{-N}\mathfrak{M}_{f}\left(\lambda'\right)\nonumber \\
= & \int d\lambda'1_{\left[-R',R'\right]}\left(\frac{\lambda'}{\epsilon\sqrt{h}}\right)\left\langle \frac{\lambda'}{\epsilon\sqrt{h}}\right\rangle ^{-N}\mathfrak{M}_{f}\left(\lambda'\right)\nonumber \\
 & +\int d\lambda'\left(1-1_{\left[-R',R'\right]}\right)\left(\frac{\lambda'}{\epsilon\sqrt{h}}\right)\left\langle \frac{\lambda'}{\epsilon\sqrt{h}}\right\rangle ^{-N}\mathfrak{M}_{f}\left(\lambda'\right).\label{eq: first term}
\end{align}
The support of $1_{\left[-R',R'\right]}\left(\frac{\lambda'}{\epsilon\sqrt{h}}\right)$
can be covered by $O\left(R'\right)$ intervals of size $\epsilon\sqrt{h}$,
which combined with the local Weyl law gives that the first term above
is $O\left(R'\epsilon h^{-m}\right)$. The second term on the other
hand, observing $\left(1-1_{\left[-R',R'\right]}\right)\left(\frac{\lambda'}{\epsilon\sqrt{h}}\right)\left\langle \frac{\lambda'}{\epsilon\sqrt{h}}\right\rangle ^{-N}=O\left(\frac{1}{R'}\left\langle \frac{\lambda'}{\sqrt{h}}\right\rangle ^{-N+1}\right)$,
is $O\left(\frac{1}{R'}h^{-m}\right)$. On choosing $R'=\frac{1}{\sqrt{\epsilon}}$,
this gives \prettyref{eq: first term} is $O\left(\sqrt{\epsilon}h^{-m}\right)$.
A similar estimate 
\begin{equation}
\int d\lambda'\epsilon\left\langle \frac{\lambda'}{\sqrt{h}}\right\rangle ^{-N}\mathfrak{M}_{f}\left(\lambda'\right)=O\left(\epsilon h^{-m}\right)\label{eq: second term}
\end{equation}
combined with \prettyref{eq: convolution distribution estimate} gives
\begin{equation}
\int d\lambda'\left[\phi_{R}\left(\frac{\lambda'}{\sqrt{h}}\right)-\phi_{R}\ast\check{\theta}_{\epsilon}\left(\frac{\lambda'}{\sqrt{h}}\right)\right]\mathfrak{M}_{f}\left(\lambda'\right)=O_{R}\left(\sqrt{\epsilon}h^{-m}\right).\label{eq: convolution diff trace est}
\end{equation}
The second term above has an expansion on integrating \prettyref{eq: Main trace expansion-1}
against $\phi_{R}$ 
\begin{align}
\int d\lambda'\phi_{R}\ast\check{\theta}_{\epsilon}\left(\frac{\lambda'}{\sqrt{h}}\right)\mathfrak{M}_{f}\left(\lambda'\right) & =h^{-m}\left[\intop d\lambda\phi_{R}\left(\lambda\right)f\left(0\right)u_{0}\left(0\right)+O_{R,\epsilon}\left(h\right)\right]\nonumber \\
 & =O_{R,\epsilon}\left(h^{-m+1}\right).\label{eq: smoothed function expansion}
\end{align}
Finally putting together \prettyref{eq:main trace exp integrated},
\prettyref{eq: definition remainder}, \prettyref{eq: cutoff remainder},
\prettyref{eq: convolution diff trace est} and \prettyref{eq: smoothed function expansion}
gives 
\begin{align*}
\textrm{tr }f^{-}\left(\frac{D}{\sqrt{h}}\right) & =\int d\lambda'1_{\left(-\infty,0\right]}\left(\lambda'\right)\mathfrak{M}_{f}\left(\lambda'\right)\\
 & =h^{-m-\frac{1}{2}}\left(\int_{-\infty}^{0}d\lambda f\left(\lambda\right)u_{0}\left(\lambda\right)+h^{1/2}\int_{-\infty}^{0}d\lambda f\left(\lambda\right)u_{1}\left(\lambda\right)\right)\\
 & \quad+O\left(\frac{h^{-m}}{R}\right)+O_{R}\left(\sqrt{\epsilon}h^{-m}\right)+O_{R,\epsilon}\left(h^{-m+1}\right)
\end{align*}
from which \prettyref{eq: irreg funct tr exp} follows on choosing
each of $\frac{1}{R}$, $\epsilon$, $h$ sufficiently small depending
on the preceding parameters.
\end{proof}
We now come to the proof of \prettyref{thm: eta semiclassical limit}.
\begin{proof}[Proof of \prettyref{thm: eta semiclassical limit}]
 We begin by using the invariance of $\eta$ under positive scaling
to write 
\begin{eqnarray}
\eta_{h}=\eta\left(\frac{D}{\sqrt{h}}\right) & = & \int_{0}^{\infty}dt\frac{1}{\sqrt{\pi t}}\textrm{ tr}\left[\frac{D}{\sqrt{h}}e^{-\frac{t}{h}D^{2}}\right]\nonumber \\
 & = & \int_{0}^{\varepsilon}dt\frac{1}{\sqrt{\pi t}}\textrm{ tr}\left[\frac{D}{\sqrt{h}}e^{-\frac{t}{h}D^{2}}\right]+\int_{\varepsilon}^{\infty}dt\frac{1}{\sqrt{\pi t}}\textrm{ tr}\left[\frac{D}{\sqrt{h}}e^{-\frac{t}{h}D^{2}}\right].\label{eq: eta integral break up}
\end{eqnarray}
The equation 4.5 pg. 859 of \cite{Savale-Asmptotics} with $r=\frac{1}{h}$
translates to the estimate 
\begin{equation}
\textrm{ tr}\left[\frac{D}{\sqrt{h}}e^{-\frac{t}{h}D^{2}}\right]=O\left(h^{-m}e^{ct}\right)\label{eq: estimate on odd trace}
\end{equation}
giving that the first integral of \prettyref{eq: eta integral break up}
is $O\left(\sqrt{\varepsilon}h^{-m}\right)$. The second integral
is evaluated to be $\textrm{tr }E_{\varepsilon}\left(\frac{D}{\sqrt{h}}\right)=\textrm{tr }\frac{1}{\varepsilon}E\left(\frac{\varepsilon D}{\sqrt{h}}\right)$
where 
\[
E(x)=\text{sign}(x)\text{erfc}(|x|)=\text{sign}(x)\cdot\frac{2}{\sqrt{\pi}}\int_{|x|}^{\infty}e^{-s^{2}}ds
\]
with the convention $\text{sign}(0)=0$. The functions $E$, $E_{\varepsilon}$
are rapidly decaying with all derivatives, odd and smooth on $\mathbb{R}_{x}\setminus0$.
Hence \prettyref{eq: irreg funct tr exp} gives 
\[
\textrm{tr }E_{\varepsilon}\left(\frac{D}{\sqrt{h}}\right)=h^{-m-\frac{1}{2}}\left[u_{0}\left(E_{\varepsilon}\right)\right]+h^{-m}\left[u_{1}\left(E_{\varepsilon}\right)\right]+o\left(h^{-m}\right)
\]
where the evaluations above again make sense on account of the smoothness
of $u_{0}$, $u_{1}$ near $0$. As observed from \cite{Savale2017-Koszul}
Prop. 7.4, the coefficient $u_{0}$ is an even function of $\lambda$.
Since $E_{\varepsilon}$ is odd, the first evaluation above is $0$.
The second is evaluated from definition to 
\begin{align*}
u_{1}\left(E_{\varepsilon}\right) & =\int_{\varepsilon}^{\infty}u_{1}\left(se^{-ts^{2}}\right)\frac{dt}{\sqrt{\pi t}}\\
 & =-\frac{1}{2}\frac{1}{\left(2\pi\right)^{m+1}}\frac{1}{m!}\int_{X}\left[\textrm{tr }\frac{1}{\left|\mathfrak{J}\right|}\left(\nabla^{TX}\mathfrak{J}\right)^{0}\right]a\wedge\left(da\right)^{m}+O\left(\varepsilon\right)
\end{align*}
following the Corollary \prettyref{cor: convergence improper integral}.
Choosing $\varepsilon$ sufficiently small and putting everything
together 
\[
\eta_{h}=h^{-m}\left(-\frac{1}{2}\frac{1}{\left(2\pi\right)^{m+1}}\frac{1}{m!}\int_{X}\left[\textrm{tr }\frac{1}{\left|\mathfrak{J}\right|}\left(\nabla^{TX}\mathfrak{J}\right)^{0}\right]a\wedge\left(da\right)^{m}\right)+o\left(h^{-m}\right)
\]
as required.
\end{proof}
\textbf{Acknowledgments.} The author would like to thank the anonymous
referee for a careful reading and several constructive suggestions
and improvements. 

\bibliographystyle{siam}
\bibliography{biblio}

\end{document}